\documentclass{article}

\usepackage{mathptmx}

\usepackage{geometry}
\usepackage{amssymb}
\usepackage{latexsym, amsmath, amscd,amsthm}
\usepackage{graphicx}
\usepackage[percent]{overpic}
\usepackage{units}
\usepackage{xcolor}
\definecolor{linkblue}{HTML}{003d73}
\definecolor{linkgreen}{HTML}{006161}
\definecolor{linkred}{HTML}{a11950}
\usepackage[pagebackref]{hyperref}
\hypersetup{
	pdftitle={New Computations of the Superbridge Index},
	pdfauthor={Clayton Shonkwiler},
	pdfsubject={knot theory},
	pdfkeywords={knots, random polygons, superbridge index, knot invariants},
	colorlinks=true,
	linkcolor=linkblue,
	citecolor=linkgreen,
	urlcolor=linkred
}
\PassOptionsToPackage{caption=false,labelformat=empty}{subfig}
\usepackage[lofdepth]{subfig}
\usepackage[export]{adjustbox}
\usepackage{algorithm}
\usepackage{algorithmicx}
\usepackage{algpseudocode}
\usepackage{booktabs}
\usepackage{authblk}

\usepackage{supertabular,multicol,ifthen,multirow}

\usepackage{array}
\newcolumntype{L}[1]{>{\raggedright\let\newline\\\arraybackslash\hspace{0pt}}m{#1}}
\newcolumntype{C}[1]{>{\centering\let\newline\\\arraybackslash\hspace{0pt}}m{#1}}
\newcolumntype{R}[1]{>{\raggedleft\let\newline\\\arraybackslash\hspace{0pt}}m{#1}}

\usepackage{cite}
\usepackage[nameinlink]{cleveref}

\makeatletter
\let\mcnewpage=\newpage
\newcommand{\TrickSupertabularIntoMulticols}{%
\renewcommand\newpage{%
    \if@firstcolumn%
        \hrule width\linewidth height0pt%
            \columnbreak%
        \else%
          \mcnewpage%
        \fi%
}%
}
\makeatother

\newtheorem{theorem}{Theorem}
\newtheorem{lemma}[theorem]{Lemma}

\newtheorem{corollary}[theorem]{Corollary}
\newtheorem*{mainthm}{Theorem~\ref*{thm:main}}

\theoremstyle{definition}

\newtheorem{conjecture}[theorem]{Conjecture}

\newcommand{\R}{\mathbb{R}}

\newcommand{\stick}{\operatorname{stick}}

\newcommand{\bridge}{\operatorname{b}}
\newcommand{\superbridge}{\operatorname{sb}}

\title{New Computations of the Superbridge Index}
\author{Clayton Shonkwiler}
\affil{Department of Mathematics, Colorado State University, Fort Collins, CO}
\date{}

\begin{document}

\maketitle

\begin{abstract}
	The knots $8_1$, $8_2$, $8_3$, $8_5$, $8_6$, $8_7$, $8_8$, $8_{10}$, $8_{11}$, $8_{12}$, $8_{13}$, $8_{14}$, $8_{15}$, $9_7$, $9_{16}$, $9_{20}$, $9_{26}$, $9_{28}$, $9_{32}$, and $9_{33}$ all have superbridge index equal to 4. This follows from new upper bounds on superbridge index not coming from the stick number and increases the number of knots from the Rolfsen table for which superbridge index is known from 29 to 49. \Cref{sec:table} gives the current state of knowledge of superbridge index for prime knots through 10 crossings.
	
\end{abstract}

\section{Introduction} 
\label{sec:introduction}

For any tamely embedded closed curve $\gamma$ in $\R^3$, the \emph{superbridge number} of $\gamma$, denoted $\superbridge(\gamma)$, is the maximum number of local maxima of the projection of $\gamma$ to any line. Compare to the bridge number $\bridge(\gamma)$, which is the \emph{minimum} number of local maxima, and the total curvature, which is $2\pi$ times the \emph{average} number of local maxima. 

The superbridge number was introduced in Kuiper's 1987 paper~\cite{Kuiper:1987ki}, where he defined the \emph{superbridge index} to be the minimum of the superbridge number over all realizations of the knot; by construction this is a knot invariant, and the superbridge index of  a knot $K$ is denoted $\superbridge[K]$. While Kuiper computed it for all torus knots, the superbridge index is generally quite difficult to determine: for example, whereas the bridge index is known for all knots through 12 crossings~\cite{Blair:2020iz,knotinfo,Musick:2012uz}, the superbridge index is known for very few knots---as of September 1, 2020, the only knot invariant recorded by KnotInfo~\cite{knotinfo} which is known for fewer knots is the topological 4-dimensional crosscap number (a.k.a.\ nonorientable 4-ball genus)~\cite{Murakami:2000wb}.

The superbridge index often appears in conjunction with the \emph{stick number}---the minimum number of segments needed to construct a piecewise-linear realization of the knot---because of the bound ${\superbridge[K] \leq \frac{1}{2}\stick[K]}$ due to Jin~\cite{Jin:1997da}. Indeed, all determinations of superbridge indices of non-torus prime knots to date have come from showing this upper bound matches a lower bound on superbridge index.

However, this cannot be a winning strategy in general: as first observed by Furstenberg, Li, and Schneider~\cite{Furstenberg:1998fq}, the difference between the two sides of Jin's bound can be arbitrarily large.

\begin{figure}[t]
	\centering
		\subfloat[$8_{10}$]{\includegraphics[height=2in,valign=c]{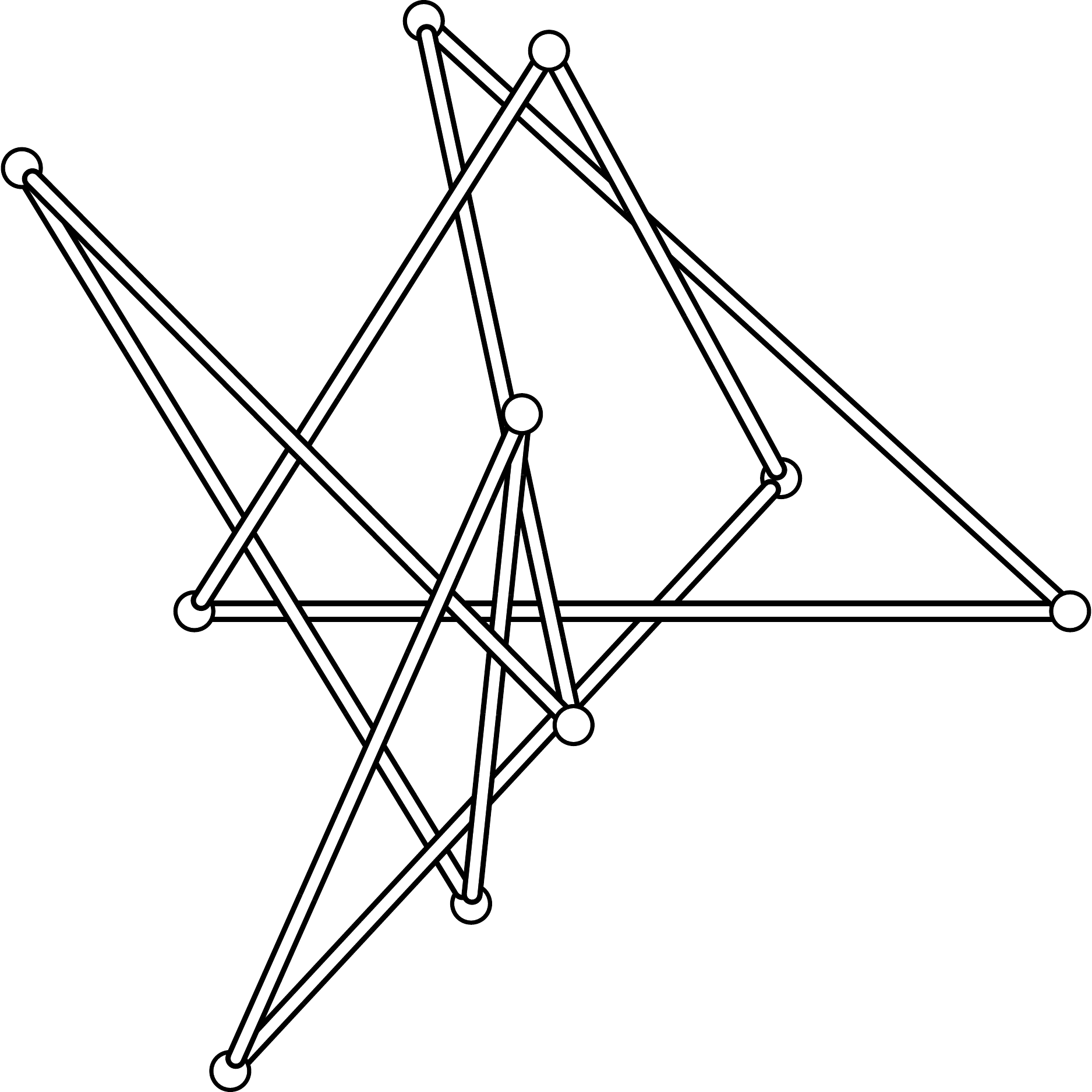}} \qquad \qquad
		\subfloat[$9_7$]{\includegraphics[height=2in,valign=c]{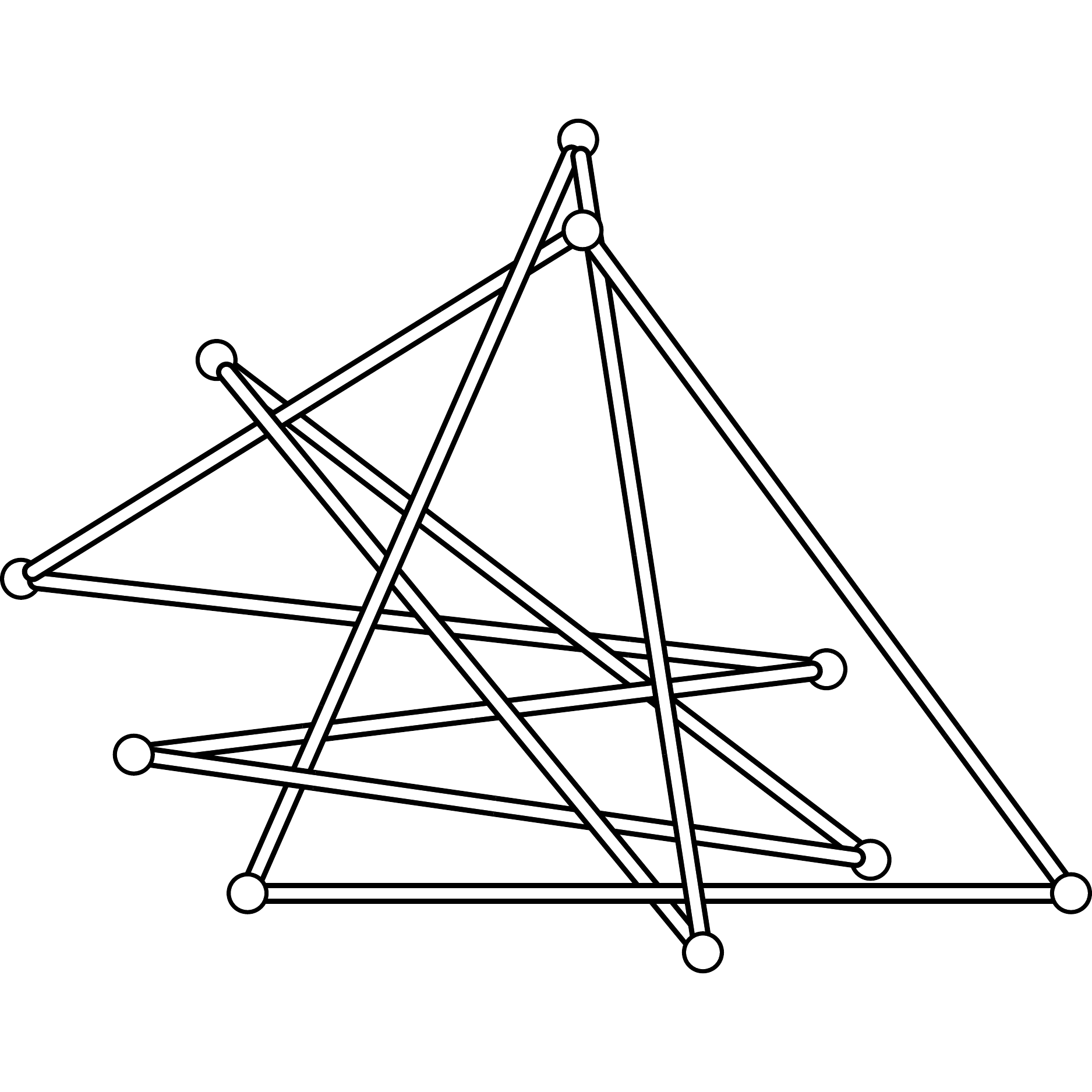}}
	\caption{Polygonal realizations of the $8_{10}$ and $9_7$ knots. These curves both have superbridge number equal to 4: each has 4 local maxima when projected to the $x$-axis, so the superbridge number is at least 4, and \Cref{cor:gordan bound} will imply that superbridge number is $\leq 4$. Each knot is shown in orthographic perspective, viewed from the direction of the positive $z$-axis relative to the vertex coordinates given in \Cref{sec:coords}.}
	\label{fig:examples}
\end{figure}

The main contribution of the present paper is a new approach to giving upper bounds on superbridge index. In some sense the approach is obvious: find a realization of a knot $K$ whose superbridge number is no bigger than some $m$, and conclude that $\superbridge[K]\leq m$. The challenge comes in showing that, for a given realization of $K$, there is no direction so that projecting to a line in that direction has more than $m$ local maxima. The strategy in this paper is to use \emph{polygonal} realizations and to observe that the existence of a direction with $m+1$ local maxima is equivalent to the feasibility of a system of linear inequalities. In other words, ruling out the existence of such directions is equivalent to showing that the system has no solutions, which will be done using Gordan's Theorem~\cite{Gordan:1873dz}, a classical linear programming tool for certifying the non-existence of solutions.

See \Cref{fig:examples} for two such polygonal realizations of the knots $8_{10}$ and $9_7$. Both will be shown to have superbridge number $\leq 4$, and hence it will follow that $\superbridge[8_{10}],\superbridge[9_7]\leq 4$. More generally, the main theorem of this paper is:

\begin{theorem}\label{thm:main}
	The knots $8_{1-15}$, $9_7$, $9_{16}$, $9_{20}$, $9_{26}$, $9_{28}$, $9_{32}$, and $9_{33}$ have superbridge index $\leq 4$, and $10_{76}$, $13n_{226}$, $13n_{328}$, $13n_{342}$, $13n_{343}$, $13n_{350}$, $13n_{512}$, $13n_{973}$, $13n_{2641}$, $13n_{5018}$, and $14n_{1753}$ have superbridge index $\leq 5$.
	
	This implies that the knots $8_1$, $8_2$, $8_3$, $8_5$, $8_6$, $8_7$, $8_8$, $8_{10}$, $8_{11}$, $8_{12}$, $8_{13}$, $8_{14}$, $8_{15}$, $9_7$, $9_{16}$, $9_{20}$, $9_{26}$, $9_{28}$, $9_{32}$, and $9_{33}$ have superbridge index equal to 4, and that $13n_{226}$, $13n_{328}$, $13n_{342}$, $13n_{343}$, $13n_{350}$, $13n_{512}$, $13n_{973}$, $13n_{2641}$, $13n_{5018}$, and $14n_{1753}$ have superbridge index equal to 5.
\end{theorem}

The particular polygonal realizations providing these bounds were found by generating large ensembles of random equilateral polygons in tight confinement using the algorithm described in the paper~\cite{TomClay} and implemented in the open-source {\tt stick-knot-gen} project~\cite{stick-knot-gen}. 

\Cref{sec:superbridge_index} below gives some background on superbridge index, including a survey of known bounds. The connection to Gordan's theorem is explained in \Cref{sec:a_new_approach}, which is where \Cref{thm:main} is proved, and \Cref{sec:conclusion} provides some discussion and open questions. \Cref{sec:table} consists of a table of values (or possible values, when the exact value is not known) of the superbridge index for all prime knots through 10 crossings, \Cref{sec:exact values} gives all prime knots through 16 crossings for which the exact value of superbridge index is known, \Cref{sec:coords} gives the coordinates for each of the polygonal knots discussed in this paper (which can also be downloaded from the {\tt stick-knot-gen} project~\cite{stick-knot-gen}), and \Cref{sec:snappy} gives diagrams for the 13- and 14-crossing knots and homomorphisms from each knot group to the symmetric group $S_5$ which will be used in the proof of \Cref{thm:main}. 

\section{Background on the Superbridge Index} 
\label{sec:superbridge_index}

The only infinite class of prime knots for which superbridge index is known is the class of torus knots:

\begin{theorem}[Kuiper~\cite{Kuiper:1987ki}]\label{torus}
	For relatively prime $2\leq p < q$, the superbridge index of the $(p,q)$-torus knot is
	\[
		\superbridge[T_{p,q}] = \min(2p,q).
	\]
\end{theorem}

Progress in determining the superbridge index of knots has been slow. Aside from torus knots, the superbridge index was known for only 41 prime knots prior to the present work. In particular, it was known for only 29 of the 249 nontrivial knots with 10 or fewer crossings. In virtually all cases, the strategy for computing superbridge index is to show that some upper bound matches a lower bound.

The most useful lower bounds on superbridge index are:

\begin{theorem}[Kuiper~\cite{Kuiper:1987ki}]\label{bridge bound}
	For any nontrivial knot $K$, the bridge index $\bridge[K] < \superbridge[K]$.
\end{theorem}

\begin{theorem}[Jeon--Jin~\cite{Jeon:2002gm}]\label{3 superbridge}
	Every knot except $3_1$ and $4_1$ and possibly $5_2$, $6_1$, $6_2$, $6_3$, $7_2$, $7_3$, $7_4$, $8_4$, and $8_9$ has superbridge index $\geq 4$.
\end{theorem}

This is slightly different than the statement in Jeon and Jin's paper, which included $8_7$ in the list of possible 3-superbridge knots. However, $8_7$ cannot have superbridge index equal to 3:

\begin{lemma}\label{lem:8_7}
	$\superbridge[8_7] \geq 4$.
\end{lemma}

\begin{proof}
	This result was first proved by Adams et al. in an early version of their paper~\cite{Adams:2020vm}; their proof, which is essentially the one given below, was also sketched in a talk given by Gyo Taek Jin in February, 2020~\cite{JinTalk}.
	
	Jeon and Jin's characterization of possible 3-superbridge knots begins by assuming that a given parametrization $\gamma: S^1 \to \R^3$ of a nontrivial knot has superbridge number 3, projecting to the orthogonal complement of a \emph{quadrisecant}---a line whose intersection with the image of $\gamma$ consists of at least 4 distinct components---and using the assumption $\superbridge(\gamma) = 3$ to constrain the resulting planar curve, and hence the possible knot types of~$\gamma$.
	
	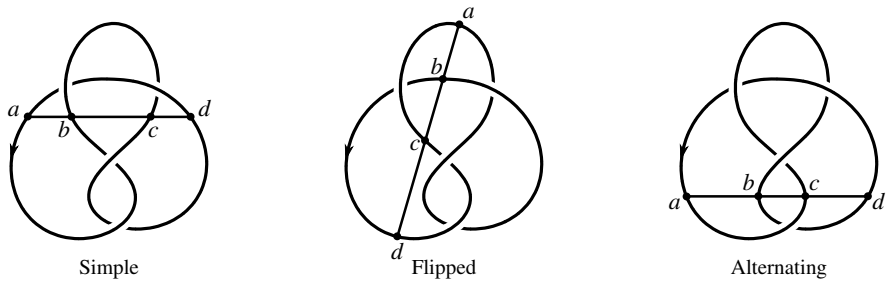
\begin{figure}[t]
		\centering
			\subfloat[Simple]{\includegraphics[height=1.2in]{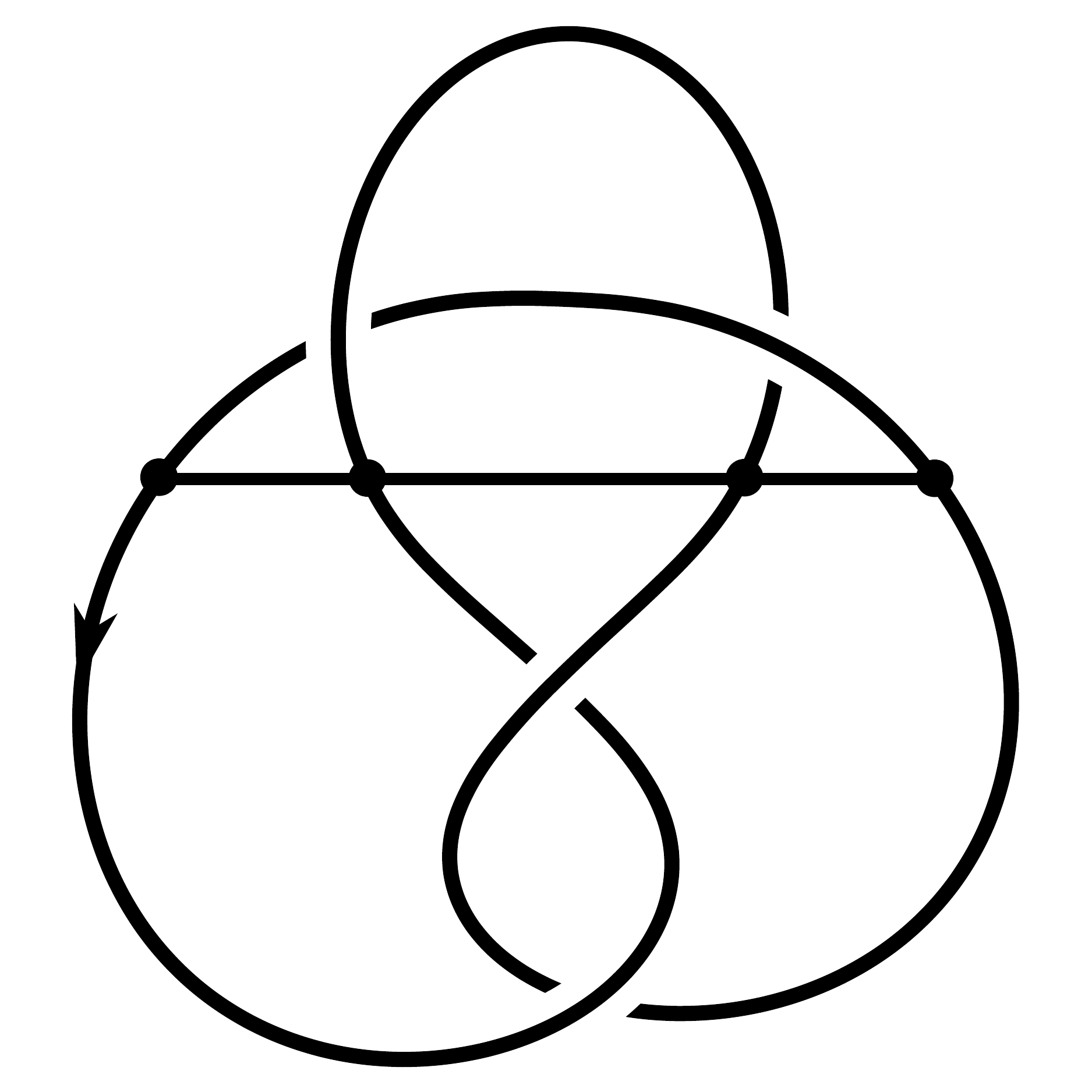}} \qquad \qquad
			\subfloat[Flipped]{\includegraphics[height=1.2in]{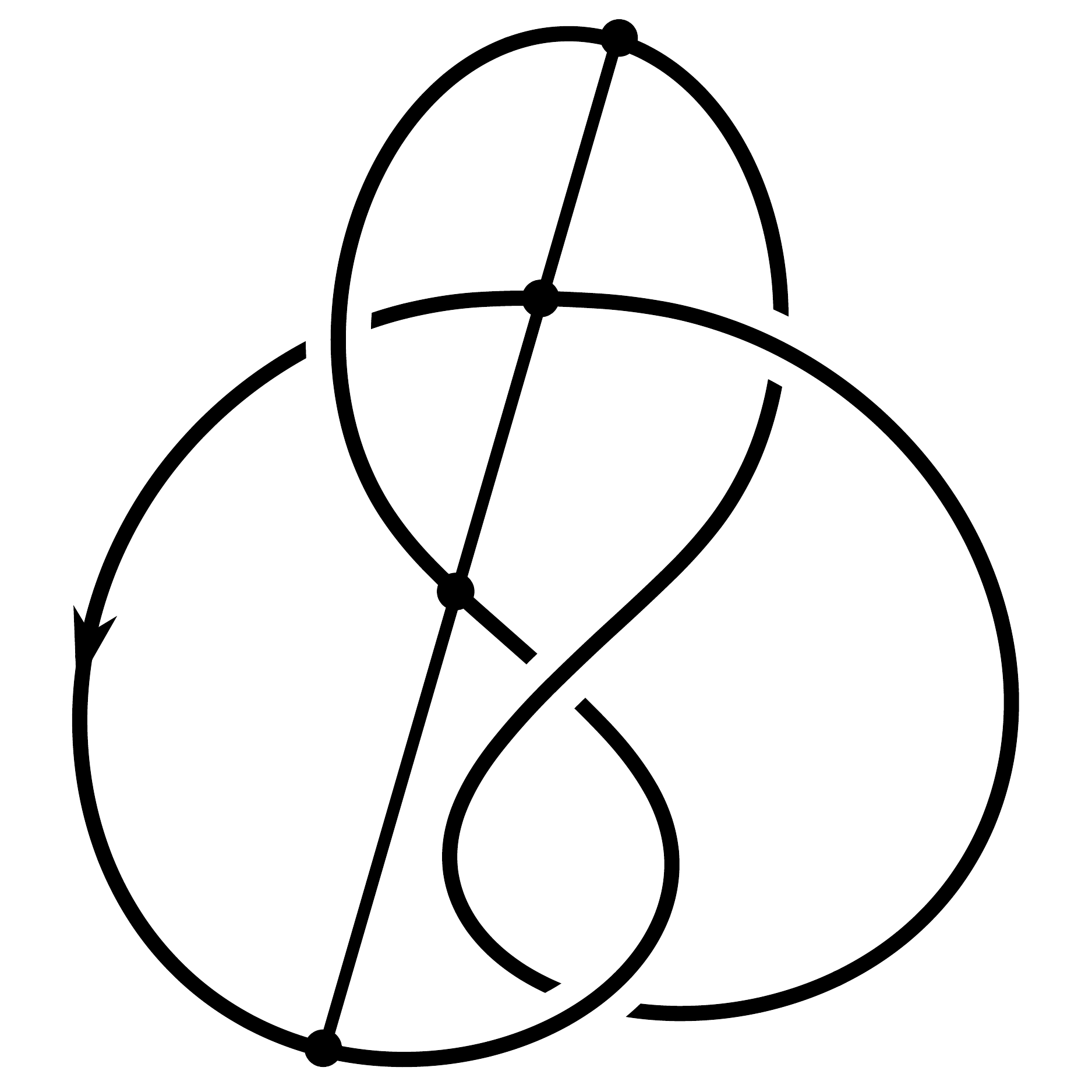}} \qquad \qquad
			\subfloat[Alternating]{\includegraphics[height=1.2in]{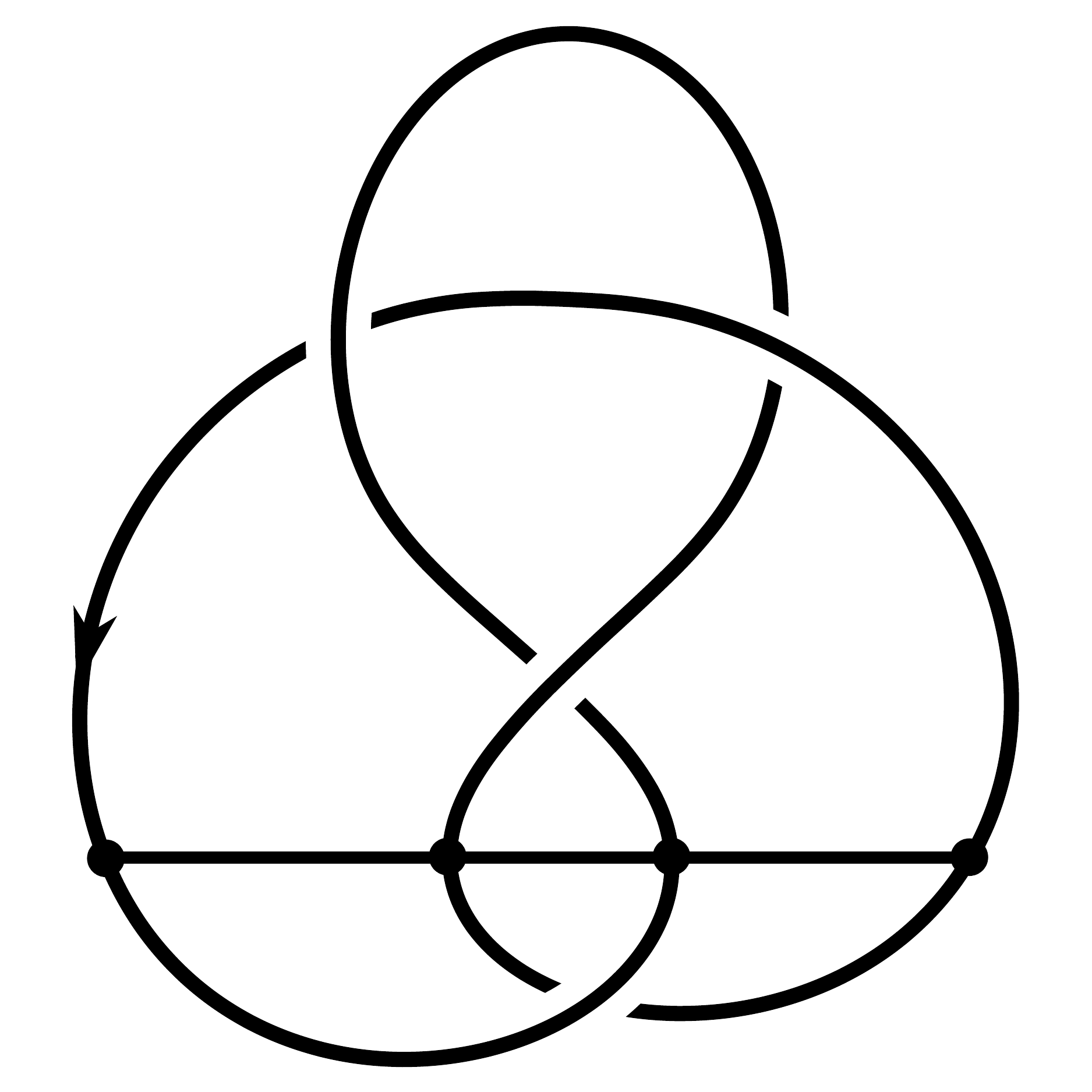}}
			\put(-85,14){\small $a$}
			\put(-57,21){\small $b$}
			\put(-32,21){\small $c$}
			\put(-8,14){\small $d$}
			\put(-163,86){\small $a$}
			\put(-175,65){\small $b$}
			\put(-183,36){\small $c$}
			\put(-190,-5){\small $d$}
			\put(-335,49){\small $a$}
			\put(-316,41){\small $b$}
			\put(-282,41){\small $c$}
			\put(-263,49){\small $d$}
		\caption{The three types of quadrisecant illustrated on the figure eight knot from KnotPlot~\cite{knotplot}. Traversing the curve in the direction of the indicated orientation starting at point $a$ produces the word $acbd$ for the alternating quadrisecant, as opposed to $abcd$ for the simple quadrisecant and $abdc$ for the flipped quadrisecant.}
		\label{fig:quadrisecants}
	\end{figure}
	
	The fact that $\gamma$ has a quadrisecant was proved by Kuperberg~\cite{Kuperberg:1994et}, building on work of Pannwitz~\cite{Pannwitz:1933ks} and Morton and Mond~\cite{Morton:1982dn}, but Denne has proved~\cite{Denne:2004wq,Denne:2005wq} that every nontrivial knot has an \emph{alternating} quadrisecant; see \Cref{fig:quadrisecants}. Therefore, Jeon and Jin's argument can be modified to require the quadrisecant of $\gamma$ be alternating. However, all the $8_7$ knots in their catalog (one in their Table~T and two in Table~V) are associated with a non-alternating quadrisecant, contradicting this assumption and proving that $8_7$ cannot have superbridge index equal to~3.
\end{proof}

By far the most useful upper bound on superbridge index is given in terms of stick number.

\begin{theorem}[Jin~\cite{Jin:1997da}]\label{stick bound}
	For any knot $K$, $\superbridge[K] \leq \frac{1}{2}\stick[K]$.
\end{theorem}

This follows because the projection of a polygonal knot to a line cannot have more critical points than vertices. To date, all determinations of superbridge indices of non-torus knots have come from matching this upper bound with one of the lower bounds from Theorems~\ref{bridge bound} or~\ref{3 superbridge}. In particular, this accounts for all exact values of superbridge index in \Cref{sec:exact values} aside from those for torus knots and those attributable to \Cref{thm:main}.

As mentioned in the introduction, this strategy cannot work in general: for any $n$ there are only finitely many knots with $\stick[K] \leq n$~\cite{Calvo:2001gv,Negami:1991gb}, but there are infinitely many 2-bridge knots and hence, by the next theorem, infinitely many knots with $\superbridge[K] \leq 5$.

\begin{theorem}[Adams et al.~\cite{Adams:2020vm}]\label{bridge upper bound}
	For any knot $K$, $\superbridge[K]\leq 3\bridge[K]-1$.
\end{theorem}

This bound is somewhat weaker than the one announced in~\cite{Adams:2009ha}, but it is still a significant improvement on the bound $\superbridge[K] \leq 5 \bridge[K]-3$ proved by Furstenberg, Li, and Schneider~\cite{Furstenberg:1998fq}. Notice, in particular, that this theorem implies $\superbridge[10_{37}] \leq 5$, even though the best extant bound on stick number is $\stick[10_{37}] \leq 12$~\cite{TomClay}.

Superbridge index is also bounded above by twice the braid index~\cite{Kuiper:1987ki} and by the harmonic index~\cite{Trautwein:1995wr,Trautwein:1998kn}, though these are less useful: the only knots through 11 crossings for which either of these bounds is better than those from Theorems~\ref{stick bound} and~\ref{bridge upper bound} are the 3-braid knots $11a_{240}$ and $11a_{338}$, both of which have bridge index 3~\cite{Musick:2012uz} and no known bound on stick number besides the general bound $\stick[K] \leq 18$ for all 11-crossing knots~\cite{Huh:2011co}.

\section{A New Approach} 
\label{sec:a_new_approach}

One of the major challenges in computing superbridge index is that the inequalities go the wrong way: if $\gamma$ is a closed curve in $\R^3$ and the projection of $\gamma$ to a line has $m$ local maxima, this shows that $\superbridge(\gamma) \geq m$. But this gives no information about $\superbridge[\gamma]$, since there could well be some other line on which the projection of $\gamma$ has $m+1$ local maxima (contrast with bridge index: in the hypothetical, it follows that $\bridge[\gamma] \leq \bridge(\gamma) \leq m$). This means it is not so easy to use particular realizations of a knot to give bounds on superbridge index.

Working with polygonal knots---that is, piecewise linear embeddings $\gamma: S^1 \to \R^3$---discretizes the problem and will turn out to yield bounds coming from Gordan's theorem. Rather than thinking in terms of the parametrization, it will be more convenient to represent polygonal knots by a list $\vec{e}_1, \dots , \vec{e}_n$ of edge vectors. 

In those terms,
\[
	\bridge_{\vec{v}}(\vec{e}_1, \dots , \vec{e}_n) = \#\{i \in [n] : \vec{e}_i \cdot \vec{v} >0 \text{ and } \vec{e}_{i+1} \cdot \vec{v} < 0\}
\]
with the convention that $\vec{e}_{n+1} = \vec{e}_1$. Here $\cdot$ is the usual dot product on $\R^3$.

The proof of \Cref{stick bound} implies that $\superbridge(\vec{e}_1, \dots , \vec{e}_n) \leq \frac{n}{2}$. Equality is achieved in this bound if and only if $n$ is even and the projection of each vertex to the line spanned by some $\vec{v} \in S^2$ is either a local minimum or a local maximum. In particular, this means that the list
\[
	\vec{v} \cdot \vec{e}_1, \vec{v} \cdot \vec{e}_2, \dots , \vec{v} \cdot \vec{e}_n
\]
must alternate signs. By replacing $\vec{v}$ with $-\vec{v}$ if necessary, it is no restriction to assume that $\vec{v} \cdot \vec{e}_1 > 0$, so that the sign pattern is $+1,-1,\dots ,-1$.

Equivalently, if $n$ is even and
\[
	E = \left[ \vec{e}_1 | -\vec{e}_2 | \cdots | -\vec{e}_n \right]
\]
is the $3 \times n$ matrix with the $(-1)^{i+1} \vec{e}_i$ as its columns, $\superbridge(\vec{e}_1, \dots , \vec{e}_n) < \frac{n}{2}$ if and only if there is no $\vec{v} \in \R^3$ so that $\vec{v}^T E$ has all positive entries. Contrapositively, the linear system $\vec{v}^T E > 0$ has a solution if and only if $\superbridge(\vec{e}_1, \dots , \vec{e}_n) = \frac{n}{2}$.

Gordan's theorem is a key tool for determining whether such systems of linear inequalities have solutions:

\begin{theorem}[Gordan~\cite{Gordan:1873dz}]\label{gordan}
	Suppose $A$ is a $k \times \ell$ real matrix. Then exactly one of the following is true:
	\begin{enumerate}
		\item \label{g1} There exists $\vec{v} \in \R^k$ so that $\vec{v}^T A$ has all positive entries.
		\item \label{g2} $A \vec{u} = \vec{0}$ for some nonzero vector $\vec{u} \in \R^\ell$ with nonnegative entries.
	\end{enumerate}
\end{theorem}

Among the class of theorems called theorems of the alternative (see, e.g.,~\cite{Tucker:1956ve} or~\cite[\S2.4]{Dantzig:2003ue}), Gordan's theorem was the first to appear, predating Farkas' Lemma~\cite{Farkas:1902bv} by almost 30 years. Here is an intuitive explanation of why Gordan's theorem is true: if \hyperref[g1]{(i)} is true, then the angles between $\vec{v}$ and each of the columns of $A$ are all less than $90^\circ$. But in this case the columns of $A$ all lie in the interior of a half-space, and no (nontrivial) conical combination produces the origin. On the other hand, \hyperref[g2]{(ii)} says exactly that the origin is a conical combination of the columns of $A$.

Gordan's theorem, combined with the preceding discussion, yields the following immediate corollary:

\begin{corollary}\label{cor:gordan bound}
	If $n$ is even and $\vec{e}_1, \dots , \vec{e}_n$ are the edge vectors of a closed polygonal curve in $\R^3$, then $\superbridge(\vec{e}_1, \dots , \vec{e}_n) < \frac{n}{2}$ if and only if there exists a nonzero vector $\vec{u} \in \R^n$ with nonnegative entries solving the matrix equation
	\begin{equation}\label{eq:gordan}
		\left[ \vec{e}_1 | -\vec{e}_2 | \cdots | -\vec{e}_n \right] \vec{u} = \vec{0}.
	\end{equation}
\end{corollary}

The vector $\vec{u}$ solving \eqref{eq:gordan} provides a certificate that there is no line onto which the polygonal curve can be projected with $\frac{n}{2}$ local maxima. This corollary is the essential result needed to prove \Cref{thm:main}.

\begin{mainthm}
	The knots $8_{1-15}$, $9_7$, $9_{16}$, $9_{20}$, $9_{26}$, $9_{28}$, $9_{32}$, and $9_{33}$ have superbridge index $\leq 4$, and $10_{76}$, $13n_{226}$, $13n_{328}$, $13n_{342}$, $13n_{343}$, $13n_{350}$, $13n_{512}$, $13n_{973}$, $13n_{2641}$, $13n_{5018}$, and $14n_{1753}$ have superbridge index $\leq 5$.
	
	This implies that the knots $8_1$, $8_2$, $8_3$, $8_5$, $8_6$, $8_7$, $8_8$, $8_{10}$, $8_{11}$, $8_{12}$, $8_{13}$, $8_{14}$, $8_{15}$, $9_7$, $9_{16}$, $9_{20}$, $9_{26}$, $9_{28}$, $9_{32}$, and $9_{33}$ have superbridge index equal to 4, and that $13n_{226}$, $13n_{328}$, $13n_{342}$, $13n_{343}$, $13n_{350}$, $13n_{512}$, $13n_{973}$, $13n_{2641}$, $13n_{5018}$, and $14n_{1753}$ have superbridge index equal to 5.
\end{mainthm}

\begin{proof}
	Coordinates of the vertices of polygonal realizations of each of these knots are given in \Cref{sec:coords} along with visualizations. For example, the entry for the $8_5$ knot is repeated in \Cref{tab:8_5}. The polygons were originally generated with coordinates given as double-precision floating point numbers, but to make it easier to verify the existence of exact solutions to \eqref{eq:gordan} these coordinates were rounded to three significant digits and converted to integers (while verifying that this did not change the knot type).
	
	\begin{table}
	\centering
	\begin{tabular*}{0.85\textwidth}{C{2.2in} R{.4in} R{.4in} R{.4in} | R{1in}}
	\multirow{10}{*}{\includegraphics[height=1.7in]{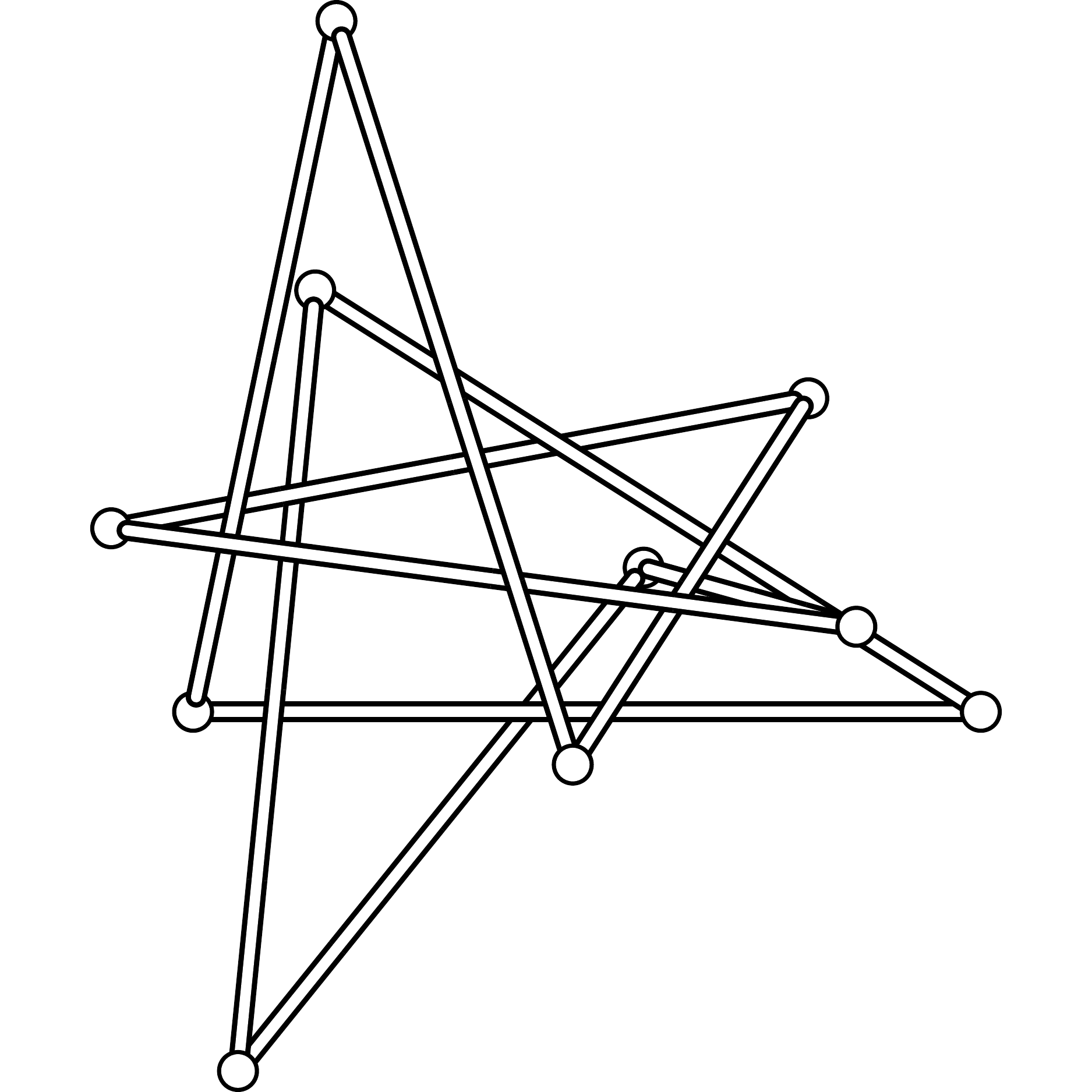}} & $0$ & $0$ & $0$ & $1$ \\
	& $1000$ & $0$ & $0$ & $1$ \\
	& $155$ & $535$ & $0$ & $8061667015$ \\
	& $57$ & $-456$ & $94$ & $1$ \\
	& $572$ & $183$ & $-478$ & $1$ \\
	& $842$ & $108$ & $482$ & $1$ \\
	& $-104$ & $233$ & $181$ & $496072961$ \\
	& $781$ & $398$ & $-254$ & $2237736971$ \\
	& $482$ & $-67$ & $579$ & $3514960071$ \\
	& $182$ & $877$ & $444$ & $4046282755$ \\
	\end{tabular*}
	\caption{Visualization of a 10-stick realization of the $8_5$ knot and the coordinates of its vertices. The rightmost column gives the entries of the vector $\vec{u}$ solving the equation in \Cref{cor:gordan bound}.}
	\label{tab:8_5}
	\end{table}

	In each case, the superbridge index bound will follow from \Cref{cor:gordan bound}, so the goal is to find an appropriate $\vec{u}$. In the case of $8_5$, it suffices to find a nonzero vector $\vec{u}$ with nonnegative entries so that
	\begin{equation}\label{eq:8_5}
		\begingroup
		\begin{bmatrix}
			1000 & 845 & -98 & -515 & 270 & 946 & 885 & 299 & -300 & 182 \\
			 0 & -535 & -991 & -639 & -75 & -125 & 165 & 465 & 944 & 877 \\
			 0 & 0 & 94 & 572 & 960 & 301 & -435 & -833 & -135 & 444 
		\end{bmatrix}
		\endgroup
		\vec{u} = 
		\begin{bmatrix}
			0 \\
			0 \\
			0
		\end{bmatrix}.
	\end{equation}
	The columns in the $3 \times 10$ matrix above are $(-1)^{i+1}\vec{e}_i$, where the $i$th edge vector $\vec{e}_i = \vec{v}_{i+1}-\vec{v}_i$, the vertices $\vec{v}_j$ have coordinates given in \Cref{tab:8_5}, and indices are computed cyclically so that $\vec{e}_{10} = \vec{v}_1-\vec{v}_{10}$. 
	
	It is straightforward to verify that
	\[
		\vec{u} =(1,\, 1,\, 8061667015,\, 1,\, 1,\, 1,\, 496072961,\, 2237736971,\, 3514960071,\, 4046282755)
	\]
	is a vector with all positive entries solving~\eqref{eq:8_5}, so this vector in conjunction with \Cref{cor:gordan bound} certifies that the superbridge number of this realization of $8_5$ is less than 5, and hence that $\superbridge[8_5] \leq 4$. The entries in $\vec{u}$ are given in the rightmost column in \Cref{tab:8_5}.
	
	A similar argument works for each of the knots in the statement of the theorem: particular vectors $\vec{u}$ solving the equation in \Cref{cor:gordan bound} are given in the rightmost column of each table in \Cref{sec:coords}. The certificate vectors $\vec{u}$ were found using Mathematica's {\tt FindInstance} function.

The second sentence in the statement of the theorem follows from observing that the upper bounds just proved match lower bounds coming from Theorems~\ref{bridge bound} and~\ref{3 superbridge}. Specifically, the knots $8_1$, $8_2$, $8_3$, $8_5$, $8_6$, $8_7$, $8_8$, $8_{10}$, $8_{11}$, $8_{12}$, $8_{13}$, $8_{14}$, $8_{15}$, $9_7$, $9_{16}$, $9_{20}$, $9_{26}$, $9_{32}$, and $9_{33}$ all have superbridge index $\geq 4$ by \Cref{3 superbridge}, and hence their superbridge index must be exactly 4. 

As for the higher-crossing knots, the result will follow if their bridge indices are bounded below by 4, since then \Cref{bridge bound} implies that their superbridge indices are at least 5. Using \Cref{lem:homomorphism}---stated below---in each case it suffices to find a surjective homomorphism $\pi_1(S^3\backslash K)\! \twoheadrightarrow S_5$  from the knot group $\pi_1(S^3\backslash K)$ to the symmetric group $S_5$ sending meridians to transpositions.  \Cref{tab:K13n350} shows a diagram of $13n_{350}$ with strand labels defining such a homomorphism. Specifically, labeling the strand $(-7,-5)$ with the transposition $(1\,2)$, the strand $(-8,9,-10)$ with $(1\,3)$, the strand $(-6,7,11,-9)$ with $(1\,4)$, and the strand $(-2,3,-1)$ with $(2\,5)$ induces a complete labeling of strands when the labels are propagated to the other strands via the Wirtinger relations. This labeling satisfies the Wirtinger relations by construction and $\{(1\,2), (1\,3), (1\,4), (2\,5)\}$ is a generating set for $S_5$, so this induces the desired surjective homomorphism $\pi_1(S^3\backslash 13n_{350})\! \twoheadrightarrow S_5$. 

\Cref{sec:snappy} gives similar homomorphisms for each of 13- and 14-crossing knots in the statement of the theorem, completing the proof. All of these homomorphisms were found using a preliminary version of {\tt Wirt\_Hm\_Suite}~\cite{Wirt-Hm-Suite}, which is being developed as part of forthcoming work by Blair, Kjuchukova, and Morrison~\cite{RyanNateSashka}.
\end{proof}

The following lemma giving a lower bound in bridge index is well-known. It is a classical example of a more general strategy that gives lower bounds on bridge index by finding a surjective homomorphism from the knot group to a group with nice properties~\cite{Baader:2019uf,Baader:2019vm}.

\begin{lemma}\label{lem:homomorphism}
		Let $K$ be a knot and $S_n$ the symmetric group on $n$ elements. If the knot group $\pi_1\!\left(S^3\, \backslash K\right)\!$ admits a surjective homomorphism to $S_n$ such that every meridian is sent to a transposition, then $\bridge[K] \geq n-1$.
\end{lemma}

\setlength{\tabcolsep}{2pt}

\begin{table}
	\begin{center}
	\begin{tabular*}{0.7\textwidth}{C{2.5in} R{1in} L{.45in}}
	\multicolumn{3}{c}{$13n_{350}$} \\
	\midrule
	\multirow{15}{*}{\includegraphics[scale=0.5]{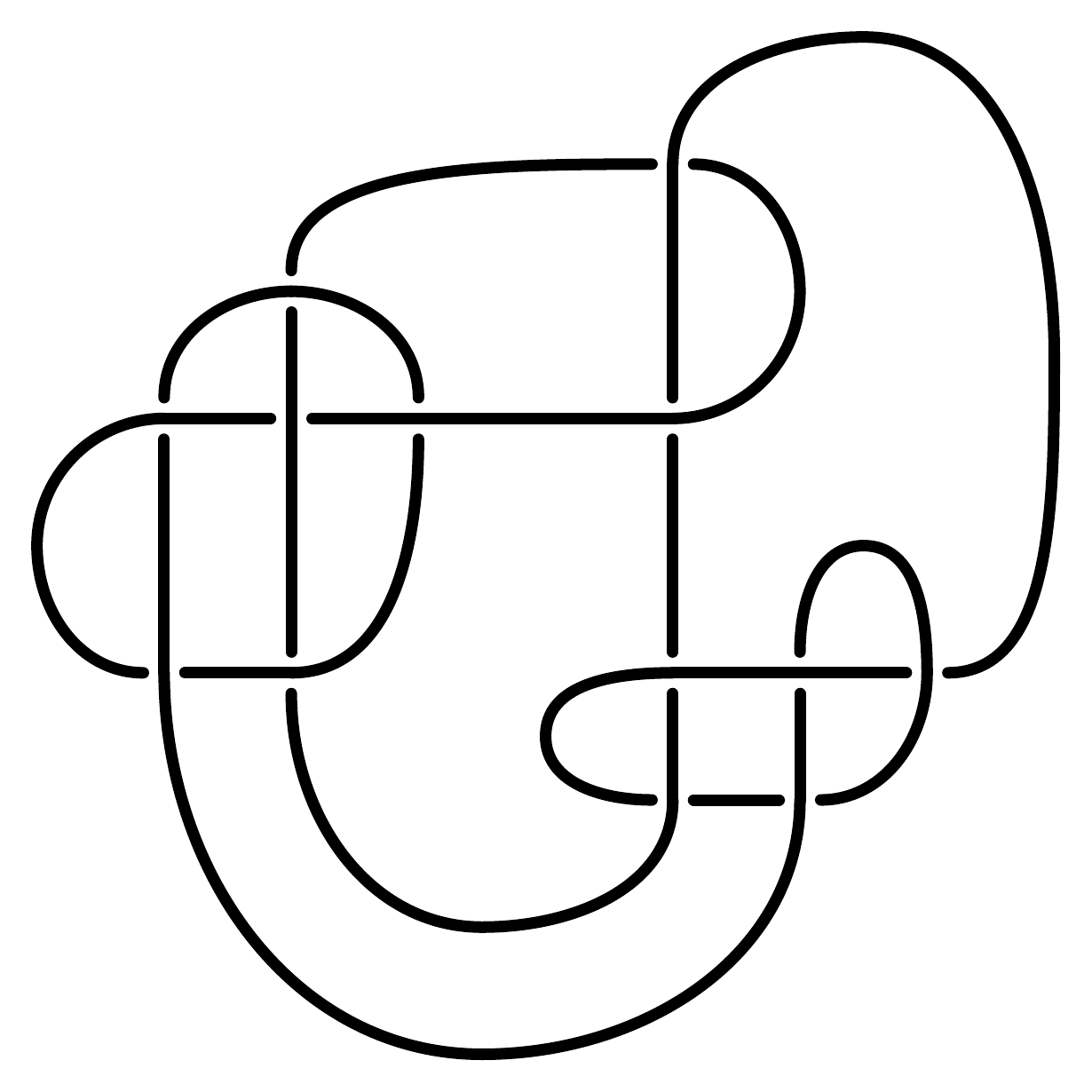}
			\put(-47,50){$1$}
			\put(-46,70){$2$}
			\put(-25,70){$3$}
			\put(-66,50){$4$}
			\put(-66,70){$5$}
			\put(-66,153){$6$}
			\put(-66,112){$7$}
			\put(-129,71){$8$}
			\put(-129,111){$9$}
			\put(-129,132){$10$}
			\put(-109,111){$11$}
			\put(-150,111){$12$}
			\put(-150,71){$13$}
			\put(-88,87){$\mathbf{(1\,2)}$}
			\put(-149.5,90){$\mathbf{(1\,3)}$}
			\put(-49,115){$\mathbf{(1\,4)}$}
			\put(-30,48){$\mathbf{(2\,5)}$}
			\put(-115,35){$(1\,2)$}
			\put(-118,155){$(1\,2)$}
			\put(-160,131){$(2\,3)$}
			\put(-112,78){$(2\,3)$}
			\put(-18,164){$(2\,4)$}
			\put(-155,17){$(2\,4)$}
			\put(-189,75){$(3\,4)$}
			\put(-68,38){$(4\,5)$}
			\put(-108,51){$(4\,5)$}
			}
	& $(-7, -5)$ & $\mapsto (1\,2)$ \\
	& $(-8, 9, -10)$ & $\mapsto (1\,3)$ \\
	& $(-6, 7, 11, -9)$ & $\mapsto (1\,4)$ \\
	& $(-2, 3, -1)$ & $\mapsto (2\,5)$\\
	& & \\
	& & \\
	& & \\
	& & \\
	& & \\
	& & \\
	& & \\
	& & \\
	& & \\
	\end{tabular*}
	\end{center}
	\caption{The SnapPy~\cite{snappy} diagram of $13n_{350}$ with crossings labeled $1,2,\dots,13$. A surjective homomorphism $\pi_1(S^3 \backslash 13n_{350}) \twoheadrightarrow S_5$ is defined by mapping Wirtinger generators to transpositions as indicated to the right. The strands in the diagram, which correspond to Wirtinger generators, are specified by indicating the under crossings at which each strand begins and ends, and the over crossings (if any) along the way. For example, $(-6,7,11,-9)$ starts at the under strand of crossing 6, passes over crossings 7 and 11, and then ends as the under strand at crossing 9. The images of the remaining Wirtinger generators are determined by applying the Wirtinger relations at each crossing. In this diagram (though not in the diagrams in \Cref{sec:snappy}), all strands are labelled with their image in $S_5$; the generating labels are bolded.}
	\label{tab:K13n350}
\end{table}

\section{Conclusion} 
\label{sec:conclusion}

The superbridge index remains frustratingly unknown for all the knots mentioned in \Cref{3 superbridge} except $3_1$ and $4_1$. \Cref{thm:main} eliminated the possibility that $\superbridge[8_4]$ or $\superbridge[8_9]$ is equal to 5, so each of the knots $5_2$, $6_1$, $6_2$, $6_3$, $7_2$, $7_3$, $7_4$, $8_4$, and $8_9$ has superbridge index equal to either 3 or 4.

\begin{conjecture}[Jeon--Jin~\cite{Jeon:2002gm}]\label{conj:jeon-jin}
	The only knots with superbridge index equal to 3 are $3_1$ and $4_1$.
\end{conjecture}

While the approach using Gordan's theorem described in \Cref{sec:a_new_approach} above cannot help prove this conjecture, it might be useful in finding a counterexample if one exists. 

The specific polygonal realizations of knots used in the proof of \Cref{thm:main} were found by generating large ensembles of random equilateral polygons in tight confinement using the algorithm described in~\cite{TomClay}. Between \Cref{thm:main} and previous work~\cite{Blair:2020kv,TomClay}, upper bounds coming from random polygons generated in this way have increased the total number of prime, non-torus knots for which superbridge index is known from 23 to 71. While there are limits to this approach, especially in the absence of more refined lower bounds on superbridge index, its usefulness is probably not yet entirely exhausted. In the future, it might be helpful to explore algorithms for generating polygons with variable (rather than fixed) edgelengths in tight confinement.

Finally, there is a paucity of useful lower bounds on superbridge index beyond those given in Theorems~\ref{bridge bound} and~\ref{3 superbridge} (Adams et al.~\cite{Adams:2010do} notwithstanding). Lower bounds coming from algebraic information---analogous to the bridge index bound from \Cref{lem:homomorphism}---would be particularly welcome and might dramatically expand our ability to compute superbridge indices of particular knots.

\subsection*{Acknowledgments} 
\label{sub:acknowledgments}

I am grateful to Jason Cantarella for introducing me to random polygons, to Tom Eddy for getting me serious about generating massive ensembles of them, and to Ryan Blair for sharing his knowledge about bridge index. As usual, thanks to Allison Moore and Chuck Livingston for maintaining KnotInfo~\cite{knotinfo}, which is an invaluable resource. This work was partially supported by grants from the Simons Foundation (\#354225 and \#709150).

\clearpage 

\appendix

\section{Table of Superbridge Indices}\label{sec:table}

Superbridge indices for prime knots through 10 crossings are given below, including references for where these results were proved. If the exact superbridge index is not known, the possible values, as determined by known upper and lower bounds, are given.

In all cases, the lower bound comes from either Kuiper's result $\bridge[K]<\superbridge[K]$~\cite[stated as \Cref{bridge bound} above]{Kuiper:1987ki} or Jeon and Jin's characterization of 3-superbridge knots~\cite[stated as \Cref{3 superbridge} above]{Jeon:2002gm}, so these references are not explicitly included in the table below. Most of the upper bounds come from Jin's bound $\superbridge[K] \leq \frac{1}{2}\stick[K]$~\cite[stated as \Cref{stick bound} above]{Jin:1997da}, but not all (e.g., $8_4$, $10_{37}$, and $10_{76}$), so this reference is included when relevant, along with reference(s) to the best upper bound on stick number. An up-to-date table of stick number bounds is given in~\cite{TomClay} or online at the {\tt stick-knot-gen} project~\cite{stick-knot-gen}.

\renewcommand{\theoremautorefname}{Thm.}

\setlength{\tabcolsep}{6pt}

\begin{center}
\begin{multicols*}{3}
\TrickSupertabularIntoMulticols

\tablefirsthead{
	\multicolumn{1}{l}{$K$} & \multicolumn{2}{l}{$\superbridge[K]$} \\
	\midrule
}
\tablehead{
	\multicolumn{1}{l}{$K$} & \multicolumn{2}{l}{$\superbridge[K]$} \\
	\midrule
}
\tablelasttail{\bottomrule}

\begin{supertabular*}{.26\textwidth}{lll} \label{tab:stick numbers}
$ 0_{1} $ & 1 &  \\
$ 3_{1} $ & 3 & \cite{Kuiper:1987ki} \\
$ 4_{1} $ & 3 & \cite{Jin:1997da,Randell:1994bx} \\
$ 5_{1} $ & 4 & \cite{Kuiper:1987ki} \\
$ 5_{2} $ & 3 or 4 & \cite{Jin:1997da} \\
$ 6_{1} $ & 3 or 4 & \cite{Jin:1997da,Negami:1991gb, Meissen:1998wu} \\
$ 6_{2} $ & 3 or 4 & \cite{Jin:1997da,Negami:1991gb, Meissen:1998wu} \\
$ 6_{3} $ & 3 or 4 & \cite{Jin:1997da,Negami:1991gb, Meissen:1998wu} \\
\midrule
$ 7_{1} $ & 4 & \cite{Kuiper:1987ki}\\
$ 7_{2} $ & 3 or 4 & \cite{Jin:1997da, Meissen:1998wu}\\
$ 7_{3} $ & 3 or 4 & \cite{Jin:1997da, Meissen:1998wu}\\
$ 7_{4} $ & 3 or 4 & \cite{Jin:1997da, Meissen:1998wu}\\
$ 7_{5} $ & 4 & \cite{Jin:1997da, Meissen:1998wu}\\
$ 7_{6} $ & 4 & \cite{Jin:1997da, Meissen:1998wu}\\
$ 7_{7} $ & 4 & \cite{Jin:1997da, Meissen:1998wu}\\
\midrule
$ 8_{1} $ & 4 & \autoref{thm:main} \\
$ 8_{2} $ & 4 & \autoref{thm:main} \\
$ 8_{3} $ & 4 & \autoref{thm:main} \\
$ 8_{4} $ & 3 or 4 & \autoref{thm:main} \\
$ 8_{5} $ & 4 & \autoref{thm:main} \\
$ 8_{6} $ & 4 & \autoref{thm:main} \\
$ 8_{7} $ & 4 & \autoref{thm:main} \\
$ 8_{8} $ & 4 & \autoref{thm:main} \\
$ 8_{9} $ & 3 or 4 & \autoref{thm:main} \\
$ 8_{10} $ & 4 & \autoref{thm:main} \\
$ 8_{11} $ & 4 & \autoref{thm:main} \\
$ 8_{12} $ & 4 & \autoref{thm:main} \\
$ 8_{13} $ & 4 & \autoref{thm:main} \\
$ 8_{14} $ & 4 & \autoref{thm:main} \\
$ 8_{15} $ & 4 & \autoref{thm:main} \\
$ 8_{16} $ & 4 & \cite{Jin:1997da, Rawdon:2002wj} \\
$ 8_{17} $ & 4 & \cite{Jin:1997da, Rawdon:2002wj} \\
$ 8_{18} $ & 4 &  \cite{Calvo:2001gv, Jin:1997da, Rawdon:2002wj} \\
$ 8_{19} $ & 4 & \cite{Kuiper:1987ki} \\
$ 8_{20} $ & 4 & \cite{Jin:1997da, Negami:1991gb, Meissen:1998wu} \\
$ 8_{21} $ & 4 & \cite{Jin:1997da, Meissen:1998wu}\\
\midrule
$ 9_{1} $ & 4 & \cite{Kuiper:1987ki} \\
$ 9_{2} $ & 4 or 5 & \cite{TomClay, Jin:1997da} \\
$ 9_{3} $ & 4 or 5 & \cite{TomClay, Jin:1997da} \\
$ 9_{4} $ & 4 or 5 & \cite{Jin:1997da, Rawdon:2002wj} \\
$ 9_{5} $ & 4 or 5 & \cite{Jin:1997da, Rawdon:2002wj} \\
$ 9_{6} $ & 4 or 5 & \cite{Jin:1997da, Rawdon:2002wj} \\
$ 9_{7} $ & 4 & \autoref{thm:main} \\
$ 9_{8} $ & 4 or 5 & \cite{Jin:1997da, Rawdon:2002wj} \\
$ 9_{9} $ & 4 or 5 & \cite{Jin:1997da, Rawdon:2002wj} \\
$ 9_{10} $ & 4 or 5 & \cite{Jin:1997da, Rawdon:2002wj} \\
$ 9_{11} $ & 4 or 5 & \cite{TomClay, Jin:1997da} \\
$ 9_{12} $ & 4 or 5 & \cite{Jin:1997da, Rawdon:2002wj} \\
$ 9_{13} $ & 4 or 5 & \cite{Jin:1997da, Rawdon:2002wj} \\
$ 9_{14} $ & 4 or 5 & \cite{Jin:1997da, Rawdon:2002wj} \\
$ 9_{15} $ & 4 or 5 & \cite{TomClay, Jin:1997da} \\
$ 9_{16} $ & 4 & \autoref{thm:main} \\
$ 9_{17} $ & 4 or 5 & \cite{Jin:1997da, Rawdon:2002wj} \\
$ 9_{18} $ & 4 or 5 & \cite{Jin:1997da, Rawdon:2002wj} \\
$ 9_{19} $ & 4 or 5 & \cite{Jin:1997da, Rawdon:2002wj} \\
$ 9_{20} $ & 4 & \autoref{thm:main} \\
$ 9_{21} $ & 4 or 5 & \cite{TomClay, Jin:1997da} \\
$ 9_{22} $ & 4 or 5 & \cite{Jin:1997da, Rawdon:2002wj} \\
$ 9_{23} $ & 4 or 5 & \cite{Jin:1997da, Rawdon:2002wj} \\
$ 9_{24} $ & 4 or 5 & \cite{Jin:1997da, Rawdon:2002wj} \\
$ 9_{25} $ & 4 or 5 & \cite{TomClay, Jin:1997da} \\
$ 9_{26} $ & 4 & \autoref{thm:main} \\
$ 9_{27} $ & 4 or 5 & \cite{TomClay, Jin:1997da} \\
$ 9_{28} $ & 4 & \autoref{thm:main} \\
$ 9_{29} $ & 4 & \cite{Jin:1997da, Scharein:1998tu} \\
$ 9_{30} $ & 4 or 5 & \cite{Jin:1997da, Rawdon:2002wj} \\
$ 9_{31} $ & 4 or 5 & \cite{Jin:1997da, Rawdon:2002wj} \\
$ 9_{32} $ & 4 & \autoref{thm:main} \\
$ 9_{33} $ & 4 & \autoref{thm:main} \\
$ 9_{34} $ & 4 & \cite{Jin:1997da, Rawdon:2002wj} \\
$ 9_{35} $ & 4 & \cite{TomClay, Jin:1997da} \\
$ 9_{36} $ & 4 or 5 & \cite{Jin:1997da, Rawdon:2002wj} \\
$ 9_{37} $ & 4 or 5 & \cite{Jin:1997da, Rawdon:2002wj} \\
$ 9_{38} $ & 4 or 5 & \cite{Jin:1997da, Rawdon:2002wj} \\
$ 9_{39} $ & 4 & \cite{TomClay, Jin:1997da} \\
$ 9_{40} $ & 4 & \cite{Jin:1997da} \\
$ 9_{41} $ & 4 & \cite{Jin:1997da} \\
$ 9_{42} $ & 4 & \cite{Jin:1997da} \\
$ 9_{43} $ & 4 & \cite{TomClay, Jin:1997da} \\
$ 9_{44} $ & 4 & \cite{Jin:1997da, Rawdon:2002wj} \\
$ 9_{45} $ & 4 & \cite{TomClay, Jin:1997da} \\
$ 9_{46} $ & 4 & \cite{Jin:1997da} \\
$ 9_{47} $ & 4 & \cite{Jin:1997da, Rawdon:2002wj} \\
$ 9_{48} $ & 4 & \cite{TomClay, Jin:1997da} \\
$ 9_{49} $ & 4 & \cite{Jin:1997da, Rawdon:2002wj} \\
\midrule
$ 10_{1} $ & 4 or 5 & \cite{Jin:1997da, Rawdon:2002wj} \\
$ 10_{2} $ & 4 or 5 & \cite{Jin:1997da, Rawdon:2002wj} \\
$ 10_{3} $ & 4 or 5 & \cite{TomClay, Jin:1997da} \\
$ 10_{4} $ & 4 or 5 & \cite{Jin:1997da, Rawdon:2002wj} \\
$ 10_{5} $ & 4 or 5 & \cite{Jin:1997da, Rawdon:2002wj} \\
$ 10_{6} $ & 4 or 5 & \cite{TomClay, Jin:1997da} \\
$ 10_{7} $ & 4 or 5 & \cite{TomClay, Jin:1997da} \\
$ 10_{8} $ & 4 or 5 & \cite{TomClay, Jin:1997da} \\
$ 10_{9} $ & 4 or 5 & \cite{Jin:1997da, Rawdon:2002wj} \\
$ 10_{10} $ & 4 or 5 & \cite{TomClay, Jin:1997da} \\
$ 10_{11} $ & 4 or 5 & \cite{Jin:1997da, Rawdon:2002wj} \\
$ 10_{12} $ & 4 or 5 & \cite{Jin:1997da, Rawdon:2002wj} \\
$ 10_{13} $ & 4 or 5 & \cite{Jin:1997da, Rawdon:2002wj} \\
$ 10_{14} $ & 4 or 5 & \cite{Jin:1997da, Rawdon:2002wj} \\
$ 10_{15} $ & 4 or 5 & \cite{TomClay, Jin:1997da} \\
$ 10_{16} $ & 4 or 5 & \cite{TomClay, Jin:1997da} \\
$ 10_{17} $ & 4 or 5 & \cite{TomClay, Jin:1997da} \\
$ 10_{18} $ & 4 or 5 & \cite{TomClay, Jin:1997da} \\
$ 10_{19} $ & 4 or 5 & \cite{Jin:1997da, Rawdon:2002wj} \\
$ 10_{20} $ & 4 or 5 & \cite{TomClay, Jin:1997da} \\
$ 10_{21} $ & 4 or 5 & \cite{TomClay, Jin:1997da} \\
$ 10_{22} $ & 4 or 5 & \cite{TomClay, Jin:1997da} \\
$ 10_{23} $ & 4 or 5 & \cite{TomClay, Jin:1997da} \\
$ 10_{24} $ & 4 or 5 & \cite{TomClay, Jin:1997da} \\
$ 10_{25} $ & 4 or 5 & \cite{Jin:1997da, Rawdon:2002wj} \\
$ 10_{26} $ & 4 or 5 & \cite{TomClay, Jin:1997da} \\
$ 10_{27} $ & 4 or 5 & \cite{Jin:1997da, Rawdon:2002wj} \\
$ 10_{28} $ & 4 or 5 & \cite{TomClay, Jin:1997da} \\
$ 10_{29} $ & 4 or 5 & \cite{Jin:1997da, Rawdon:2002wj} \\
$ 10_{30} $ & 4 or 5 & \cite{TomClay, Jin:1997da} \\
$ 10_{31} $ & 4 or 5 & \cite{TomClay, Jin:1997da} \\
$ 10_{32} $ & 4 or 5 & \cite{Jin:1997da, Rawdon:2002wj} \\
$ 10_{33} $ & 4 or 5 & \cite{Jin:1997da, Rawdon:2002wj} \\
$ 10_{34} $ & 4 or 5 & \cite{TomClay, Jin:1997da} \\
$ 10_{35} $ & 4 or 5 & \cite{TomClay, Jin:1997da} \\
$ 10_{36} $ & 4 or 5 & \cite{Jin:1997da, Rawdon:2002wj} \\
$ 10_{37} $ & 4 or 5 & \cite{Adams:2020vm} \\
$ 10_{38} $ & 4 or 5 & \cite{TomClay, Jin:1997da} \\
$ 10_{39} $ & 4 or 5 & \cite{TomClay, Jin:1997da} \\
$ 10_{40} $ & 4 or 5 & \cite{Jin:1997da, Rawdon:2002wj} \\
$ 10_{41} $ & 4 or 5 & \cite{Jin:1997da, Rawdon:2002wj} \\
$ 10_{42} $ & 4 or 5 & \cite{Jin:1997da, Rawdon:2002wj} \\
$ 10_{43} $ & 4 or 5 & \cite{TomClay, Jin:1997da} \\
$ 10_{44} $ & 4 or 5 & \cite{TomClay, Jin:1997da} \\
$ 10_{45} $ & 4 or 5 & \cite{Jin:1997da, Rawdon:2002wj} \\
$ 10_{46} $ & 4 or 5 & \cite{TomClay, Jin:1997da} \\
$ 10_{47} $ & 4 or 5 & \cite{TomClay, Jin:1997da} \\
$ 10_{48} $ & 4 or 5 & \cite{Jin:1997da, Rawdon:2002wj} \\
$ 10_{49} $ & 4 or 5 & \cite{Jin:1997da, Rawdon:2002wj} \\
$ 10_{50} $ & 4 or 5 & \cite{TomClay, Jin:1997da} \\
$ 10_{51} $ & 4 or 5 & \cite{TomClay, Jin:1997da} \\
$ 10_{52} $ & 4 or 5 & \cite{Jin:1997da, Rawdon:2002wj} \\
$ 10_{53} $ & 4 or 5 & \cite{TomClay, Jin:1997da} \\
$ 10_{54} $ & 4 or 5 & \cite{TomClay, Jin:1997da} \\
$ 10_{55} $ & 4 or 5 & \cite{TomClay, Jin:1997da} \\
$ 10_{56} $ & 4 or 5 & \cite{TomClay, Jin:1997da} \\
$ 10_{57} $ & 4 or 5 & \cite{TomClay, Jin:1997da} \\
$ 10_{58} $ & 4, 5, or 6 & \cite{Jin:1997da, Rawdon:2002wj} \\
$ 10_{59} $ & 4 or 5 & \cite{Jin:1997da, Rawdon:2002wj} \\
$ 10_{60} $ & 4 or 5 & \cite{Jin:1997da, Rawdon:2002wj} \\
$ 10_{61} $ & 4 or 5 & \cite{Jin:1997da, Rawdon:2002wj} \\
$ 10_{62} $ & 4 or 5 & \cite{TomClay, Jin:1997da} \\
$ 10_{63} $ & 4 or 5 & \cite{Jin:1997da, Rawdon:2002wj} \\
$ 10_{64} $ & 4 or 5 & \cite{TomClay, Jin:1997da} \\
$ 10_{65} $ & 4 or 5 & \cite{TomClay, Jin:1997da} \\
$ 10_{66} $ & 4, 5, or 6 & \cite{Jin:1997da, Rawdon:2002wj} \\
$ 10_{67} $ & 4 or 5 & \cite{Jin:1997da, Rawdon:2002wj} \\
$ 10_{68} $ & 4 or 5 & \cite{TomClay, Jin:1997da} \\
$ 10_{69} $ & 4 or 5 & \cite{Jin:1997da, Rawdon:2002wj} \\
$ 10_{70} $ & 4 or 5 & \cite{TomClay, Jin:1997da} \\
$ 10_{71} $ & 4 or 5 & \cite{TomClay, Jin:1997da} \\
$ 10_{72} $ & 4 or 5 & \cite{TomClay, Jin:1997da} \\
$ 10_{73} $ & 4 or 5 & \cite{TomClay, Jin:1997da} \\
$ 10_{74} $ & 4 or 5 & \cite{TomClay, Jin:1997da} \\
$ 10_{75} $ & 4 or 5 & \cite{TomClay, Jin:1997da} \\
$ 10_{76} $ & 4 or 5 & \autoref{thm:main} \\
$ 10_{77} $ & 4 or 5 & \cite{TomClay, Jin:1997da} \\
$ 10_{78} $ & 4 or 5 & \cite{TomClay, Jin:1997da} \\
$ 10_{79} $ & 4 or 5 & \cite{Jin:1997da, Scharein:1998tu} \\
$ 10_{80} $ & 4, 5, or 6 & \cite{TomClay, Jin:1997da} \\
$ 10_{81} $ & 4 or 5 & \cite{Jin:1997da, Rawdon:2002wj} \\
$ 10_{82} $ & 4 or 5 & \cite{TomClay, Jin:1997da} \\
$ 10_{83} $ & 4 or 5 & \cite{TomClay, Jin:1997da} \\
$ 10_{84} $ & 4 or 5 & \cite{TomClay, Jin:1997da} \\
$ 10_{85} $ & 4 or 5 & \cite{TomClay, Jin:1997da} \\
$ 10_{86} $ & 4 or 5 & \cite{Jin:1997da, Rawdon:2002wj} \\
$ 10_{87} $ & 4 or 5 & \cite{Jin:1997da, Rawdon:2002wj} \\
$ 10_{88} $ & 4 or 5 & \cite{Jin:1997da, Rawdon:2002wj} \\
$ 10_{89} $ & 4 or 5 & \cite{Jin:1997da, Rawdon:2002wj} \\
$ 10_{90} $ & 4 or 5 & \cite{TomClay, Jin:1997da} \\
$ 10_{91} $ & 4 or 5 & \cite{TomClay, Jin:1997da} \\
$ 10_{92} $ & 4 or 5 & \cite{Jin:1997da, Rawdon:2002wj} \\
$ 10_{93} $ & 4 or 5 & \cite{Jin:1997da, Rawdon:2002wj} \\
$ 10_{94} $ & 4 or 5 & \cite{TomClay, Jin:1997da} \\
$ 10_{95} $ & 4 or 5 & \cite{TomClay, Jin:1997da} \\
$ 10_{96} $ & 4 or 5 & \cite{Jin:1997da, Rawdon:2002wj} \\
$ 10_{97} $ & 4 or 5 & \cite{TomClay, Jin:1997da} \\
$ 10_{98} $ & 4 or 5 & \cite{Jin:1997da, Rawdon:2002wj} \\
$ 10_{99} $ & 4 or 5 & \cite{Jin:1997da, Rawdon:2002wj} \\
$ 10_{100} $ & 4 or 5 & \cite{TomClay, Jin:1997da} \\
$ 10_{101} $ & 4 or 5 & \cite{TomClay, Jin:1997da} \\
$ 10_{102} $ & 4 or 5 & \cite{Jin:1997da, Rawdon:2002wj} \\
$ 10_{103} $ & 4 or 5 & \cite{TomClay, Jin:1997da} \\
$ 10_{104} $ & 4 or 5 & \cite{Jin:1997da, Rawdon:2002wj} \\
$ 10_{105} $ & 4 or 5 & \cite{TomClay, Jin:1997da} \\
$ 10_{106} $ & 4 or 5 & \cite{TomClay, Jin:1997da} \\
$ 10_{107} $ & 4 or 5 & \cite{TomClay, Jin:1997da, Scharein:1998tu} \\
$ 10_{108} $ & 4 or 5 & \cite{Jin:1997da, Rawdon:2002wj} \\
$ 10_{109} $ & 4 or 5 & \cite{Jin:1997da, Rawdon:2002wj} \\
$ 10_{110} $ & 4 or 5 & \cite{TomClay, Jin:1997da} \\
$ 10_{111} $ & 4 or 5 & \cite{TomClay, Jin:1997da} \\
$ 10_{112} $ & 4 or 5 & \cite{TomClay, Jin:1997da} \\
$ 10_{113} $ & 4 or 5 & \cite{Jin:1997da, Rawdon:2002wj} \\
$ 10_{114} $ & 4 or 5 & \cite{Jin:1997da, Rawdon:2002wj} \\
$ 10_{115} $ & 4 or 5 & \cite{TomClay, Jin:1997da} \\
$ 10_{116} $ & 4 or 5 & \cite{Jin:1997da, Rawdon:2002wj} \\
$ 10_{117} $ & 4 or 5 & \cite{TomClay, Jin:1997da} \\
$ 10_{118} $ & 4 or 5 & \cite{TomClay, Jin:1997da} \\
$ 10_{119} $ & 4 or 5 & \cite{TomClay, Jin:1997da, Scharein:1998tu} \\
$ 10_{120} $ & 4 or 5 & \cite{Jin:1997da, Rawdon:2002wj} \\
$ 10_{121} $ & 4 or 5 & \cite{Jin:1997da, Rawdon:2002wj} \\
$ 10_{122} $ & 4 or 5 & \cite{Jin:1997da, Rawdon:2002wj} \\
$ 10_{123} $ & 4 or 5 & \cite{Jin:1997da, Rawdon:2002wj} \\
$ 10_{124} $ & 5 & \cite{Kuiper:1987ki} \\
$ 10_{125} $ & 4 or 5 & \cite{Jin:1997da, Rawdon:2002wj} \\
$ 10_{126} $ & 4 or 5 & \cite{TomClay, Jin:1997da} \\
$ 10_{127} $ & 4 or 5 & \cite{Jin:1997da, Rawdon:2002wj} \\
$ 10_{128} $ & 4 or 5 & \cite{Jin:1997da, Rawdon:2002wj} \\
$ 10_{129} $ & 4 or 5 & \cite{Jin:1997da, Rawdon:2002wj} \\
$ 10_{130} $ & 4 or 5 & \cite{Jin:1997da, Rawdon:2002wj} \\
$ 10_{131} $ & 4 or 5 & \cite{TomClay, Jin:1997da} \\
$ 10_{132} $ & 4 or 5 & \cite{Jin:1997da, Rawdon:2002wj} \\
$ 10_{133} $ & 4 or 5 & \cite{TomClay, Jin:1997da} \\
$ 10_{134} $ & 4 or 5 & \cite{Jin:1997da, Rawdon:2002wj} \\
$ 10_{135} $ & 4 or 5 & \cite{Jin:1997da, Rawdon:2002wj} \\
$ 10_{136} $ & 4 or 5 & \cite{Jin:1997da, Rawdon:2002wj} \\
$ 10_{137} $ & 4 or 5 & \cite{TomClay, Jin:1997da} \\
$ 10_{138} $ & 4 or 5 & \cite{TomClay, Jin:1997da} \\
$ 10_{139} $ & 4 or 5 & \cite{Jin:1997da, Rawdon:2002wj} \\
$ 10_{140} $ & 4 or 5 & \cite{Jin:1997da, Rawdon:2002wj} \\
$ 10_{141} $ & 4 or 5 & \cite{Jin:1997da, Rawdon:2002wj} \\
$ 10_{142} $ & 4 or 5 & \cite{TomClay, Jin:1997da} \\
$ 10_{143} $ & 4 or 5 & \cite{TomClay, Jin:1997da} \\
$ 10_{144} $ & 4 or 5 & \cite{Jin:1997da, Rawdon:2002wj} \\
$ 10_{145} $ & 4 or 5 & \cite{Jin:1997da, Rawdon:2002wj} \\
$ 10_{146} $ & 4 or 5 & \cite{Jin:1997da, Rawdon:2002wj} \\
$ 10_{147} $ & 4 or 5 & \cite{TomClay, Jin:1997da, Scharein:1998tu} \\
$ 10_{148} $ & 4 or 5 & \cite{TomClay, Jin:1997da} \\
$ 10_{149} $ & 4 or 5 & \cite{TomClay, Jin:1997da} \\
$ 10_{150} $ & 4 or 5 & \cite{Jin:1997da, Rawdon:2002wj} \\
$ 10_{151} $ & 4 or 5 & \cite{Jin:1997da, Rawdon:2002wj} \\
$ 10_{152} $ & 4 or 5 & \cite{Jin:1997da, Rawdon:2002wj} \\
$ 10_{153} $ & 4 or 5 & \cite{TomClay, Jin:1997da} \\
$ 10_{154} $ & 4 or 5 & \cite{Jin:1997da, Rawdon:2002wj} \\
$ 10_{155} $ & 4 or 5 & \cite{Jin:1997da, Rawdon:2002wj} \\
$ 10_{156} $ & 4 or 5 & \cite{Jin:1997da, Rawdon:2002wj} \\
$ 10_{157} $ & 4 or 5 & \cite{Jin:1997da, Rawdon:2002wj} \\
$ 10_{158} $ & 4 or 5 & \cite{Jin:1997da, Rawdon:2002wj} \\
$ 10_{159} $ & 4 or 5 & \cite{Jin:1997da, Rawdon:2002wj} \\
$ 10_{160} $ & 4 or 5 & \cite{Jin:1997da, Rawdon:2002wj} \\
$ 10_{161} $ & 4 or 5 & \cite{Jin:1997da, Rawdon:2002wj} \\
$ 10_{162} $ & 4 or 5 & \cite{Jin:1997da, Rawdon:2002wj} \\
$ 10_{163} $ & 4 or 5 & \cite{Jin:1997da, Rawdon:2002wj} \\
$ 10_{164} $ & 4 or 5 & \cite{TomClay, Jin:1997da} \\
$ 10_{165} $ & 4 or 5 & \cite{Jin:1997da, Rawdon:2002wj} \\
\end{supertabular*}
\end{multicols*}

\end{center}

\section{Exact Values of Superbridge Index}\label{sec:exact values}

The prime knots through 16 crossings for which the exact value of superbridge index is known.


\begin{center}
\begin{multicols*}{3}
\TrickSupertabularIntoMulticols

\tablefirsthead{
	\multicolumn{1}{l}{$K$} & \multicolumn{2}{l}{$\superbridge[K]$} \\
	\midrule
}
\tablehead{
	\multicolumn{1}{l}{$K$} & \multicolumn{2}{l}{$\superbridge[K]$} \\
	\midrule
}
\tablelasttail{\bottomrule}

\begin{supertabular*}{.26\textwidth}{lll} \label{tab:exact stick numbers}
$ 0_{1} $ & 1 &  \\
$ 3_{1} $ & 3 & \cite{Kuiper:1987ki} \\
$ 4_{1} $ & 3 & \cite{Jin:1997da,Randell:1994bx} \\
$ 5_{1} $ & 4 & \cite{Kuiper:1987ki} \\
$ 7_{1} $ & 4 & \cite{Kuiper:1987ki}\\
$ 7_{5} $ & 4 & \cite{Jin:1997da, Meissen:1998wu}\\
$ 7_{6} $ & 4 & \cite{Jin:1997da, Meissen:1998wu}\\
$ 7_{7} $ & 4 & \cite{Jin:1997da, Meissen:1998wu}\\
$ 8_{1} $ & 4 & \autoref{thm:main} \\
$ 8_{2} $ & 4 & \autoref{thm:main} \\
$ 8_{3} $ & 4 & \autoref{thm:main} \\
$ 8_{5} $ & 4 & \autoref{thm:main} \\
$ 8_{6} $ & 4 & \autoref{thm:main} \\
$ 8_{7} $ & 4 & \autoref{thm:main} \\
$ 8_{8} $ & 4 & \autoref{thm:main} \\
$ 8_{10} $ & 4 & \autoref{thm:main} \\
$ 8_{11} $ & 4 & \autoref{thm:main} \\
$ 8_{12} $ & 4 & \autoref{thm:main} \\
$ 8_{13} $ & 4 & \autoref{thm:main} \\
$ 8_{14} $ & 4 & \autoref{thm:main} \\
$ 8_{15} $ & 4 & \autoref{thm:main} \\
$ 8_{16} $ & 4 & \cite{Jin:1997da, Rawdon:2002wj} \\
$ 8_{17} $ & 4 & \cite{Jin:1997da, Rawdon:2002wj} \\
$ 8_{18} $ & 4 &  \cite{Calvo:2001gv, Jin:1997da, Rawdon:2002wj} \\
$ 8_{19} $ & 4 & \cite{Kuiper:1987ki} \\
$ 8_{20} $ & 4 & \cite{Jin:1997da, Negami:1991gb, Meissen:1998wu} \\
$ 8_{21} $ & 4 & \cite{Jin:1997da, Meissen:1998wu}\\
$ 9_{1} $ & 4 & \cite{Kuiper:1987ki} \\
$ 9_{7} $ & 4 & \autoref{thm:main} \\
$ 9_{16} $ & 4 & \autoref{thm:main} \\
$ 9_{20} $ & 4 & \autoref{thm:main} \\
$ 9_{26} $ & 4 & \autoref{thm:main} \\
$ 9_{28} $ & 4 & \autoref{thm:main} \\
$ 9_{29} $ & 4 & \cite{Jin:1997da, Scharein:1998tu} \\
$ 9_{32} $ & 4 & \autoref{thm:main} \\
$ 9_{33} $ & 4 & \autoref{thm:main} \\
$ 9_{34} $ & 4 & \cite{Jin:1997da, Rawdon:2002wj} \\
$ 9_{35} $ & 4 & \cite{TomClay, Jin:1997da} \\
$ 9_{39} $ & 4 & \cite{TomClay, Jin:1997da} \\
$ 9_{40} $ & 4 & \cite{Jin:1997da} \\
$ 9_{41} $ & 4 & \cite{Jin:1997da} \\
$ 9_{42} $ & 4 & \cite{Jin:1997da} \\
$ 9_{43} $ & 4 & \cite{TomClay, Jin:1997da} \\
$ 9_{44} $ & 4 & \cite{Jin:1997da, Rawdon:2002wj} \\
$ 9_{45} $ & 4 & \cite{TomClay, Jin:1997da} \\
$ 9_{46} $ & 4 & \cite{Jin:1997da} \\
$ 9_{47} $ & 4 & \cite{Jin:1997da, Rawdon:2002wj} \\
$ 9_{48} $ & 4 & \cite{TomClay, Jin:1997da} \\
$ 9_{49} $ & 4 & \cite{Jin:1997da, Rawdon:2002wj} \\
$ 10_{124} $ & 5 & \cite{Kuiper:1987ki} \\
$ 11a_{367} $ & 4 & \cite{Kuiper:1987ki} \\
$ 11n_{71} $ & 5 & \cite{TomClay, Jin:1997da} \\
$ 11n_{73} $ & 5 & \cite{TomClay, Jin:1997da} \\
$ 11n_{74} $ & 5 & \cite{TomClay, Jin:1997da} \\
$ 11n_{75} $ & 5 & \cite{TomClay, Jin:1997da} \\
$ 11n_{76} $ & 5 & \cite{TomClay, Jin:1997da} \\
$ 11n_{78} $ & 5 & \cite{TomClay, Jin:1997da} \\
$ 11n_{81} $ & 5 & \cite{TomClay, Jin:1997da} \\
$ 13a_{4878} $ & 4 & \cite{Kuiper:1987ki} \\
$ 13n_{226} $ & 5 & \autoref{thm:main} \\
$ 13n_{285} $ & 5 & \cite{Blair:2020kv, Jin:1997da} \\
$ 13n_{293} $ & 5 & \cite{Blair:2020kv, Jin:1997da} \\
$ 13n_{328} $ & 5 & \autoref{thm:main} \\
$ 13n_{342} $ & 5 & \autoref{thm:main} \\
$ 13n_{343} $ & 5 & \autoref{thm:main} \\
$ 13n_{350} $ & 5 & \autoref{thm:main} \\
$ 13n_{512} $ & 5 & \autoref{thm:main} \\
$ 13n_{587} $ & 5 & \cite{Blair:2020kv, Jin:1997da} \\
$ 13n_{592} $ & 5 & \cite{Blair:2020kv, Jin:1997da} \\
$ 13n_{607} $ & 5 & \cite{Blair:2020kv, Jin:1997da} \\
$ 13n_{611} $ & 5 & \cite{Blair:2020kv, Jin:1997da} \\
$ 13n_{835} $ & 5 & \cite{Blair:2020kv, Jin:1997da} \\
$ 13n_{973} $ & 5 & \autoref{thm:main} \\
$ 13n_{1177} $ & 5 & \cite{Blair:2020kv, Jin:1997da} \\
$ 13n_{1192} $ & 5 & \cite{Blair:2020kv, Jin:1997da} \\
$ 13n_{2641} $ & 5 & \autoref{thm:main} \\
$ 13n_{5018} $ & 5 & \autoref{thm:main} \\
$ 14n_{1753} $ & 5 & \autoref{thm:main} \\
$ 14n_{21881}$ & 6 & \cite{Kuiper:1987ki} \\
$ 15a_{85263} $ & 4 & \cite{Kuiper:1987ki} \\
$ 15n_{41126} $ & 5 & \cite{Blair:2020kv, Jin:1997da} \\
$ 15n_{41127} $ & 5 & \cite{Blair:2020kv, Jin:1997da} \\
$ 15n_{41185} $ & 5 & \cite{Kuiper:1987ki} \\
$ 16n_{783154} $ & 6 & \cite{Kuiper:1987ki} \\

\end{supertabular*}
\end{multicols*}

\end{center}

\section{Knot Images and Coordinates} 
\label{sec:coords}

Images of and vertex coordinates for each of the knots mentioned in \Cref{thm:main}. Each knot is shown in orthographic perspective from the direction of the positive $z$-axis. The first three columns to the right of the image are the coordinates of the vertices, which have been normalized so that the first vertex is at the origin, the second vertex is on the positive $x$-axis, and the third vertex lies in the $xy$-plane with positive $y$-coordinate. The last column gives the entries in the vector $\vec{u}$ satisfying the equation from \Cref{cor:gordan bound}.

The original floating-point coordinates can be downloaded from the {\tt stick-knot-gen} project~\cite{stick-knot-gen}.

\begin{center}
\begin{tabular*}{0.85\textwidth}{C{2.2in} R{.4in} R{.4in} R{.4in} | R{1in}}
\multicolumn{5}{c}{$8_{1}$} \\
\cline{1-5}\noalign{\smallskip}
\multirow{10}{*}{\includegraphics[height=1.8in]{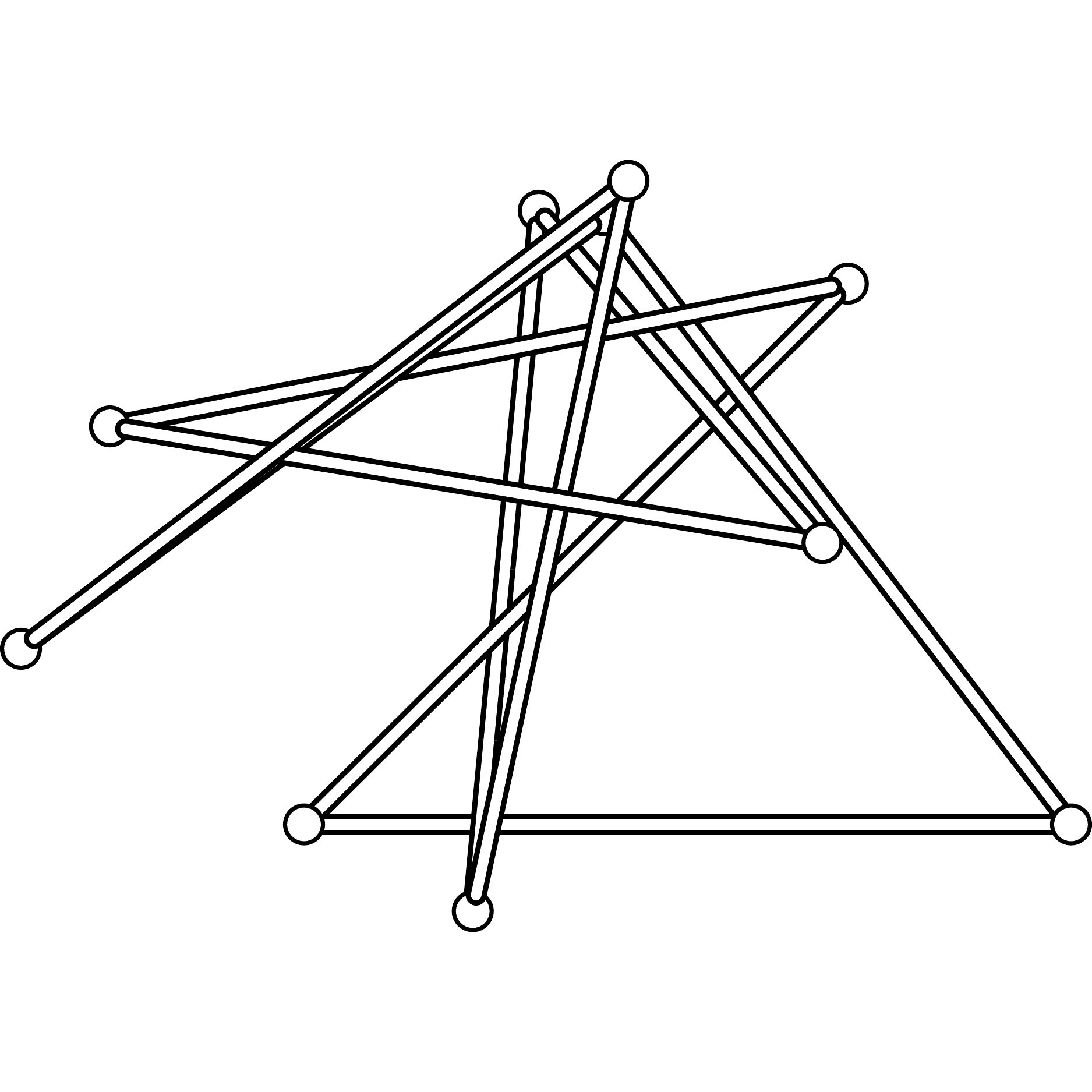}} & $0$ & $0$ & $0$ & $1$ \\
& $1000$ & $0$ & $0$ & $1$ \\
& $392$ & $794$ & $0$ & $1$ \\
& $-369$ & $229$ & $319$ & $2864186130$ \\
& $423$ & $839$ & $330$ & $1$ \\
& $220$ & $-113$ & $101$ & $5$ \\
& $306$ & $800$ & $-299$ & $48024948$ \\
& $676$ & $367$ & $523$ & $236163582$ \\
& $-253$ & $519$ & $185$ & $280265709$ \\
& $709$ & $705$ & $-13$ & $2484703250$ \\
\end{tabular*}

\medskip

\begin{tabular*}{0.85\textwidth}{C{2.2in} R{.4in} R{.4in} R{.4in} | R{1in}}
\multicolumn{5}{c}{$8_{2}$} \\
\cline{1-5}\noalign{\smallskip}
\multirow{10}{*}{\includegraphics[height=1.8in]{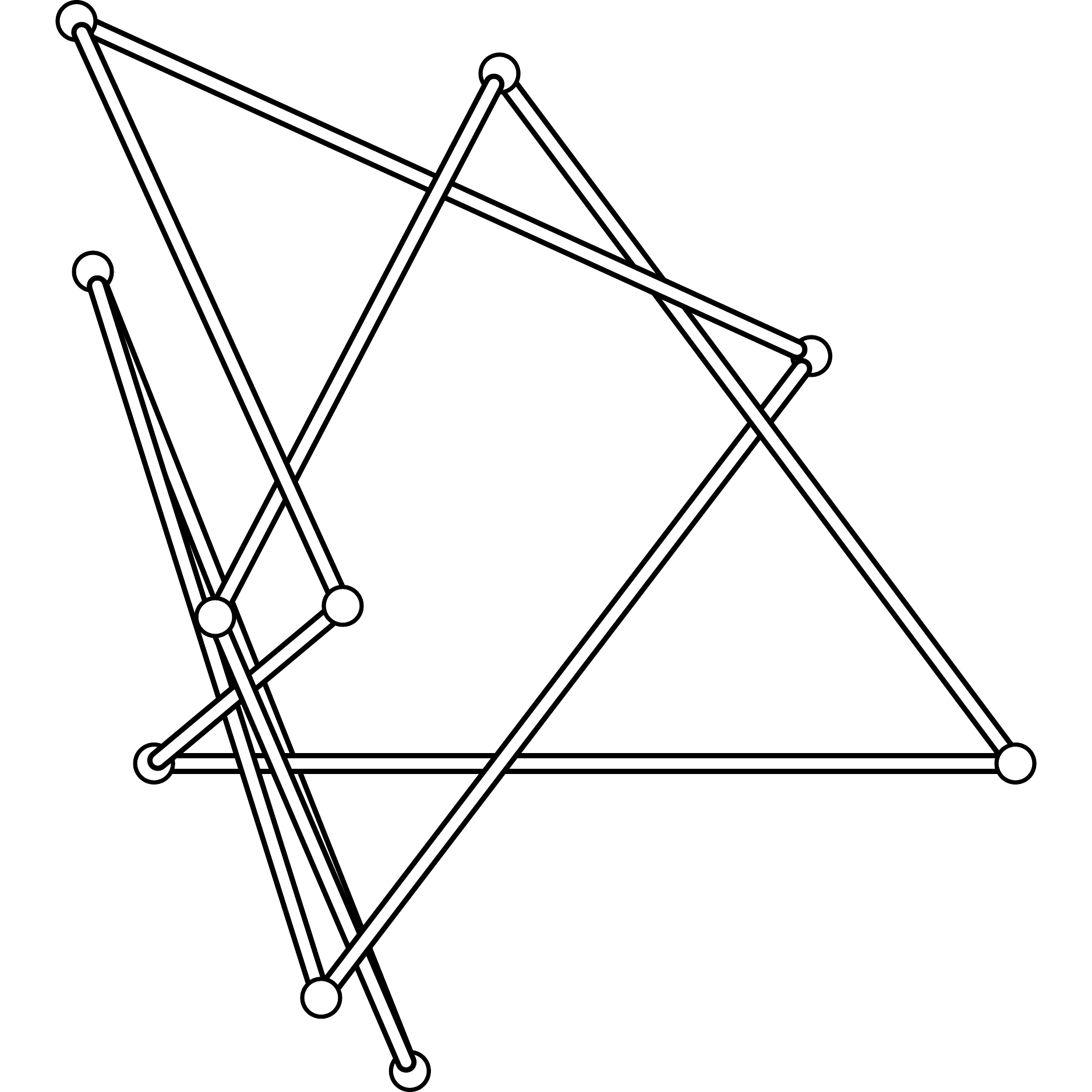}} & $0$ & $0$ & $0$ & $1$ \\
& $1000$ & $0$ & $0$ & $1$ \\
& $401$ & $801$ & $0$ & $27676293106$ \\
& $71$ & $170$ & $702$ & $2$ \\
& $297$ & $-357$ & $-117$ & $1$ \\
& $-71$ & $571$ & $-174$ & $29466942625$ \\
& $194$ & $-272$ & $294$ & $359004357$ \\
& $763$ & $473$ & $-56$ & $19211515592$ \\
& $-90$ & $862$ & $293$ & $494790273$ \\
& $219$ & $183$ & $958$ & $901050440$ \\
\end{tabular*}

\medskip

\begin{tabular*}{0.85\textwidth}{C{2.2in} R{.4in} R{.4in} R{.4in} | R{1in}}
\multicolumn{5}{c}{$8_{3}$} \\
\cline{1-5}\noalign{\smallskip}
\multirow{10}{*}{\includegraphics[height=1.8in]{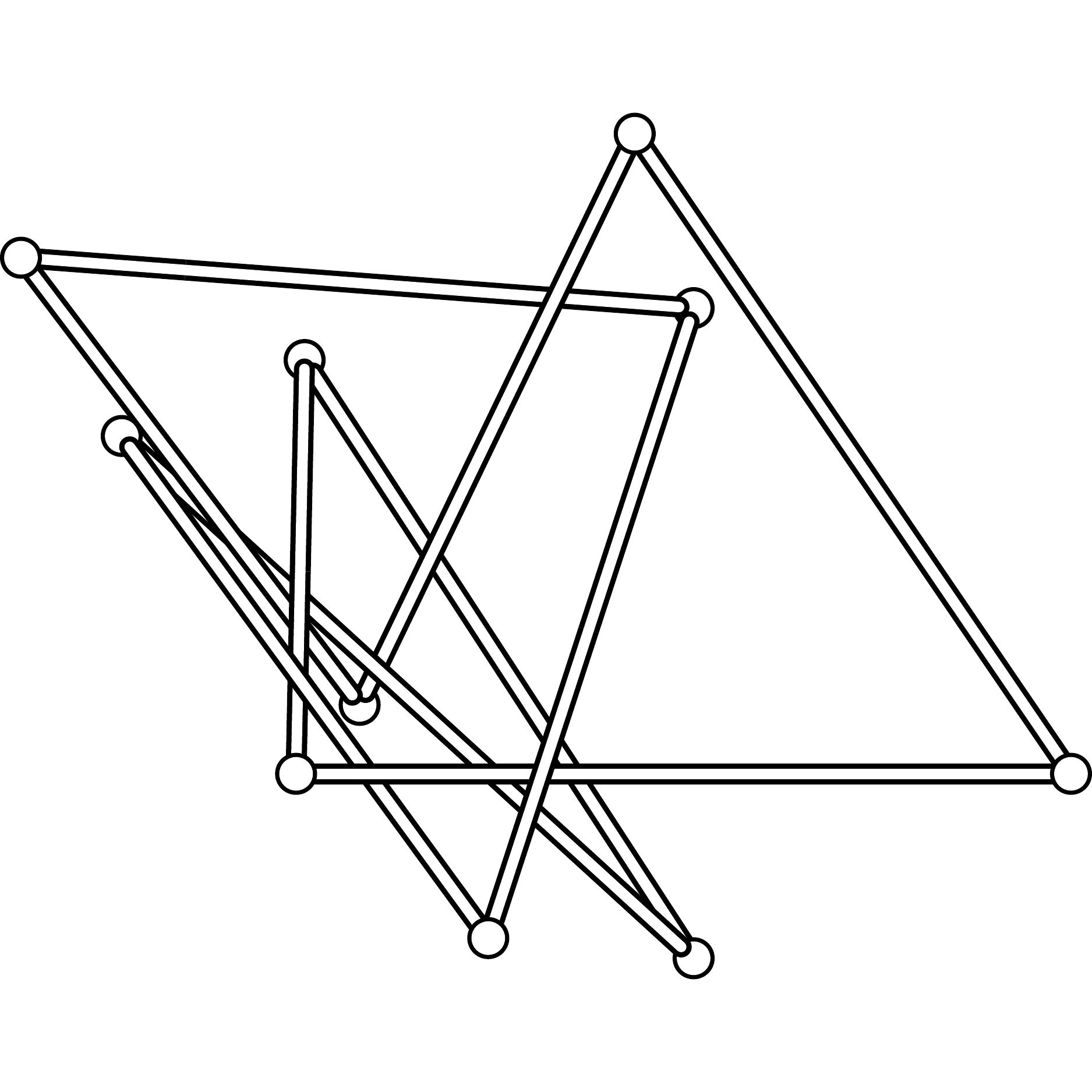}} & $0$ & $0$ & $0$ & $11523814716$ \\
& $1000$ & $0$ & $0$ & $33922449104$ \\
& $437$ & $826$ & $0$ & $1$ \\
& $82$ & $89$ & $-575$ & $1$ \\
& $-355$ & $666$ & $115$ & $1$ \\
& $513$ & $601$ & $-376$ & $1$ \\
& $248$ & $-212$ & $142$ & $9680960$ \\
& $-225$ & $436$ & $-454$ & $41470509758$ \\
& $513$ & $-238$ & $-456$ & $26311260$ \\
& $11$ & $534$ & $-846$ & $79089567$ \\
\end{tabular*}

\medskip

\begin{tabular*}{0.85\textwidth}{C{2.2in} R{.4in} R{.4in} R{.4in} | R{1in}}
\multicolumn{5}{c}{$8_{4}$} \\
\cline{1-5}\noalign{\smallskip}
\multirow{10}{*}{\includegraphics[height=1.8in]{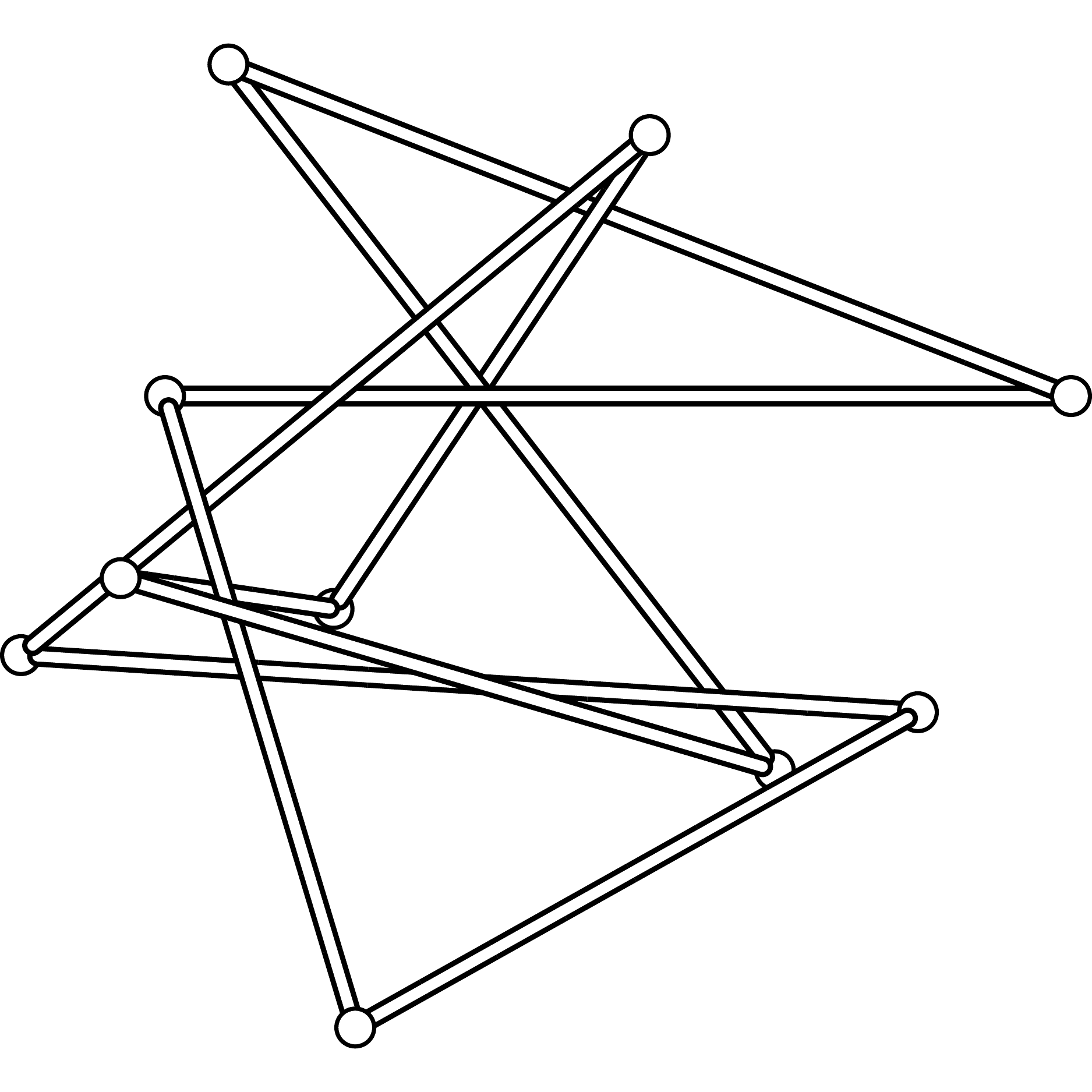}} & $0$ & $0$ & $0$ & $82224356775$ \\
& $1000$ & $0$ & $0$ & $1$ \\
& $70$ & $366$ & $0$ & $1$ \\
& $673$ & $-413$ & $-168$ & $1$ \\
& $-49$ & $-201$ & $491$ & $1$ \\
& $186$ & $-235$ & $-481$ & $1$ \\
& $535$ & $288$ & $297$ & $12935105$ \\
& $-159$ & $-286$ & $-137$ & $75530285415$ \\
& $831$ & $-349$ & $-17$ & $12222683491$ \\
& $210$ & $-697$ & $685$ & $713758069$ \\
\end{tabular*}

\medskip

\begin{tabular*}{0.85\textwidth}{C{2.2in} R{.4in} R{.4in} R{.4in} | R{1in}}
\multicolumn{5}{c}{$8_{5}$} \\
\cline{1-5}\noalign{\smallskip}
\multirow{10}{*}{\includegraphics[height=1.8in]{8_5.pdf}} & $0$ & $0$ & $0$ & $1$ \\
& $1000$ & $0$ & $0$ & $1$ \\
& $155$ & $535$ & $0$ & $8061667015$ \\
& $57$ & $-456$ & $94$ & $1$ \\
& $572$ & $183$ & $-478$ & $1$ \\
& $842$ & $108$ & $482$ & $1$ \\
& $-104$ & $233$ & $181$ & $496072961$ \\
& $781$ & $398$ & $-254$ & $2237736971$ \\
& $482$ & $-67$ & $579$ & $3514960071$ \\
& $182$ & $877$ & $444$ & $4046282755$ \\
\end{tabular*}

\medskip

\begin{tabular*}{0.85\textwidth}{C{2.2in} R{.4in} R{.4in} R{.4in} | R{1in}}
\multicolumn{5}{c}{$8_{6}$} \\
\cline{1-5}\noalign{\smallskip}
\multirow{10}{*}{\includegraphics[height=1.8in]{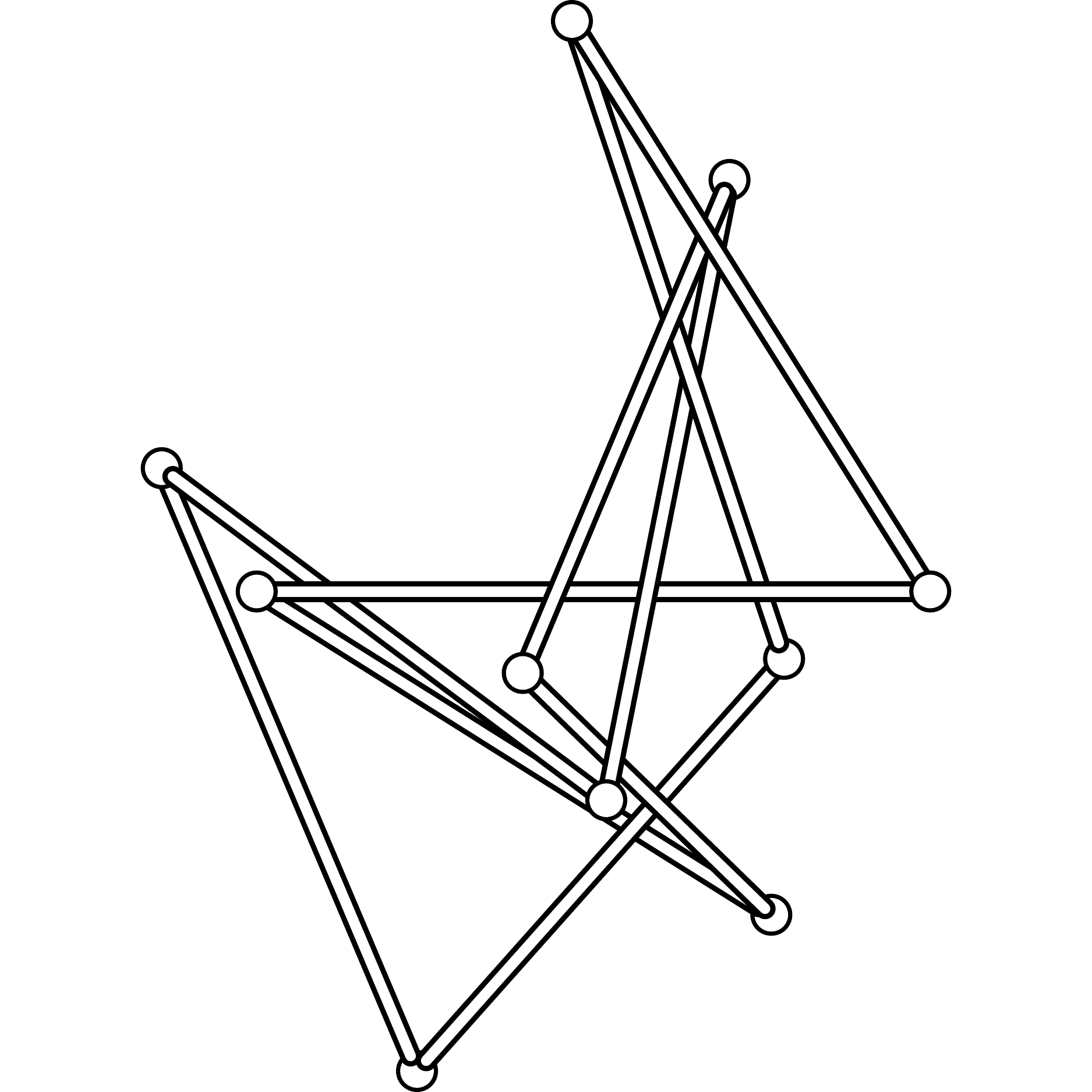}} & $0$ & $0$ & $0$ & $1$ \\
& $1000$ & $0$ & $0$ & $1$ \\
& $468$ & $847$ & $0$ & $13996103539$ \\
& $783$ & $-100$ & $-68$ & $1$ \\
& $238$ & $-712$ & $-640$ & $12357524765$ \\
& $-141$ & $183$ & $-406$ & $744803262$ \\
& $519$ & $-310$ & $161$ & $1157872283$ \\
& $702$ & $611$ & $-183$ & $1254460956$ \\
& $395$ & $-121$ & $425$ & $402371992$ \\
& $764$ & $-480$ & $-432$ & $27243476$ \\
\end{tabular*}

\medskip

\begin{tabular*}{0.85\textwidth}{C{2.2in} R{.4in} R{.4in} R{.4in} | R{1in}}
\multicolumn{5}{c}{$8_{7}$} \\
\cline{1-5}\noalign{\smallskip}
\multirow{10}{*}{\includegraphics[height=1.8in]{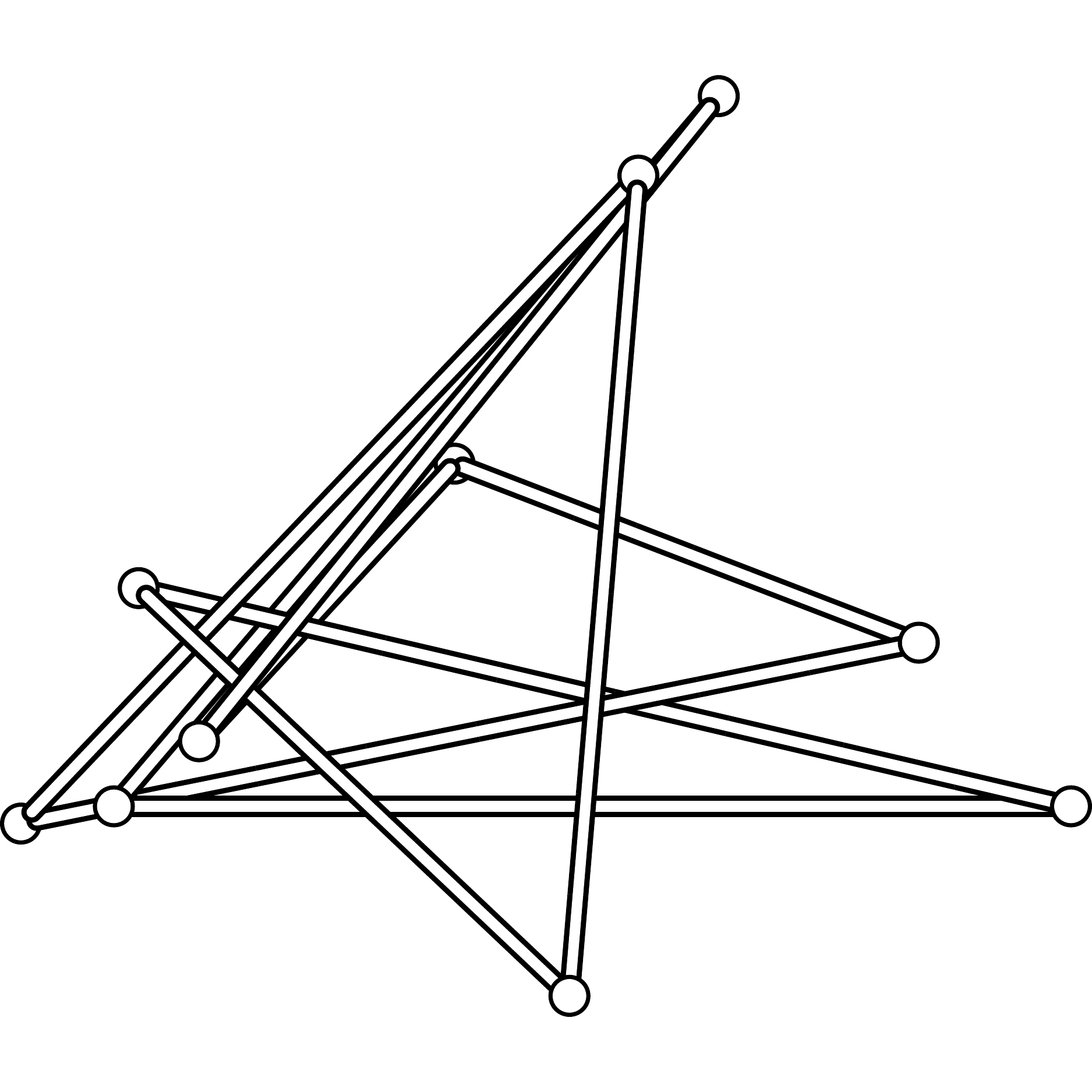}} & $0$ & $0$ & $0$ & $1$ \\
& $1000$ & $0$ & $0$ & $1$ \\
& $26$ & $228$ & $0$ & $993497797$ \\
& $476$ & $-198$ & $785$ & $1$ \\
& $548$ & $659$ & $276$ & $3$ \\
& $-97$ & $-18$ & $-77$ & $1004972499$ \\
& $841$ & $171$ & $212$ & $16597981$ \\
& $356$ & $358$ & $-642$ & $135013158$ \\
& $89$ & $68$ & $277$ & $601139748$ \\
& $632$ & $742$ & $-224$ & $223374715$ \\
\end{tabular*}

\medskip

\begin{tabular*}{0.85\textwidth}{C{2.2in} R{.4in} R{.4in} R{.4in} | R{1in}}
\multicolumn{5}{c}{$8_{8}$} \\
\cline{1-5}\noalign{\smallskip}
\multirow{10}{*}{\includegraphics[height=1.8in]{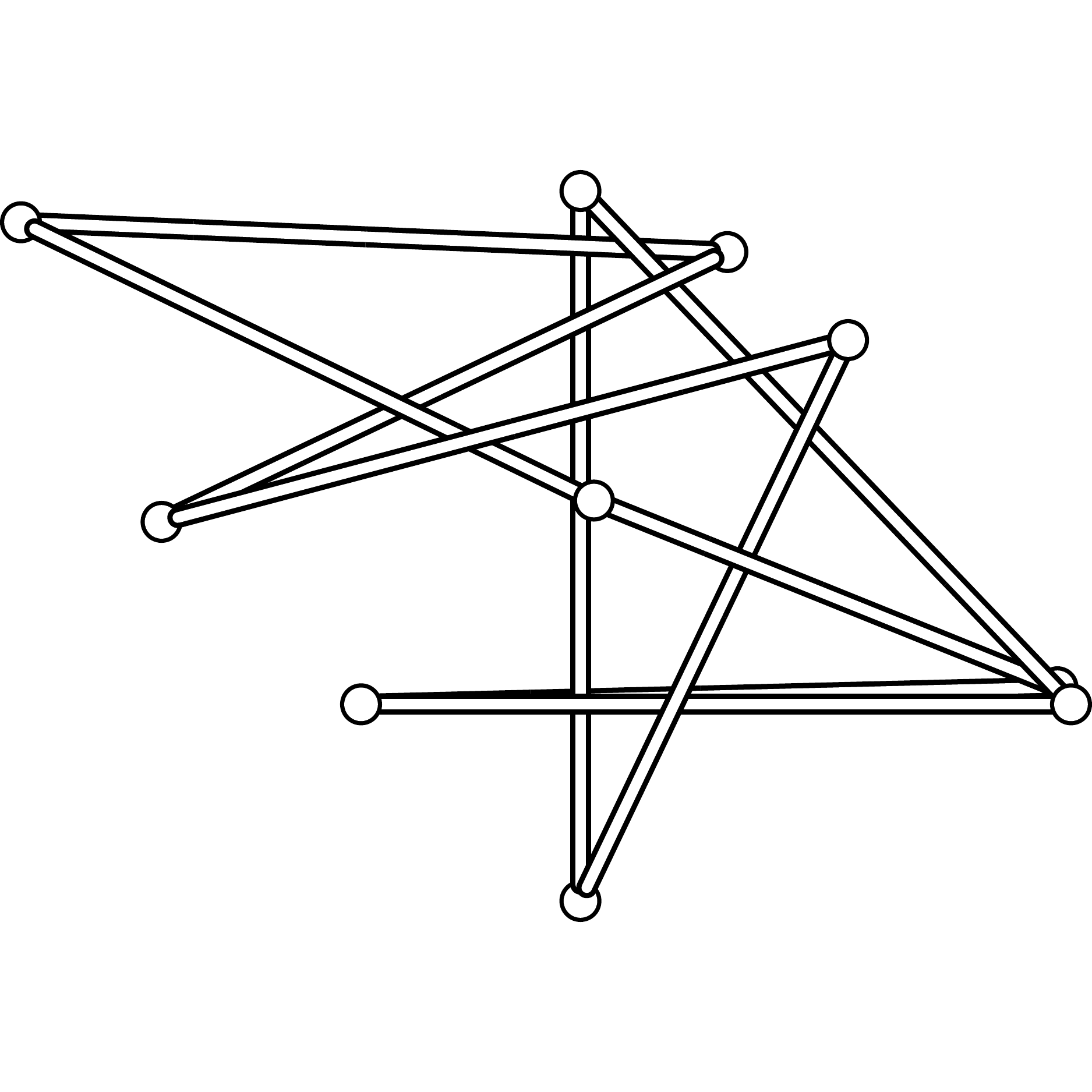}} & $0$ & $0$ & $0$ & $1261677133211$ \\
& $10000$ & $0$ & $0$ & $1$ \\
& $3086$ & $7225$ & $0$ & $130835438045$ \\
& $3093$ & $-2773$ & $-206$ & $1$ \\
& $6857$ & $5129$ & $4632$ & $1$ \\
& $-2809$ & $2567$ & $4514$ & $1$ \\
& $5157$ & $6374$ & $-183$ & $1430647762564$ \\
& $-4789$ & $6789$ & $767$ & $166227086780$ \\
& $3281$ & $2871$ & $5184$ & $3491957024$ \\
& $9811$ & $242$ & $-1919$ & $298663293981$ \\
\end{tabular*}

\medskip

\begin{tabular*}{0.85\textwidth}{C{2.2in} R{.4in} R{.4in} R{.4in} | R{1in}}
\multicolumn{5}{c}{$8_{9}$} \\
\cline{1-5}\noalign{\smallskip}
\multirow{10}{*}{\includegraphics[height=1.8in]{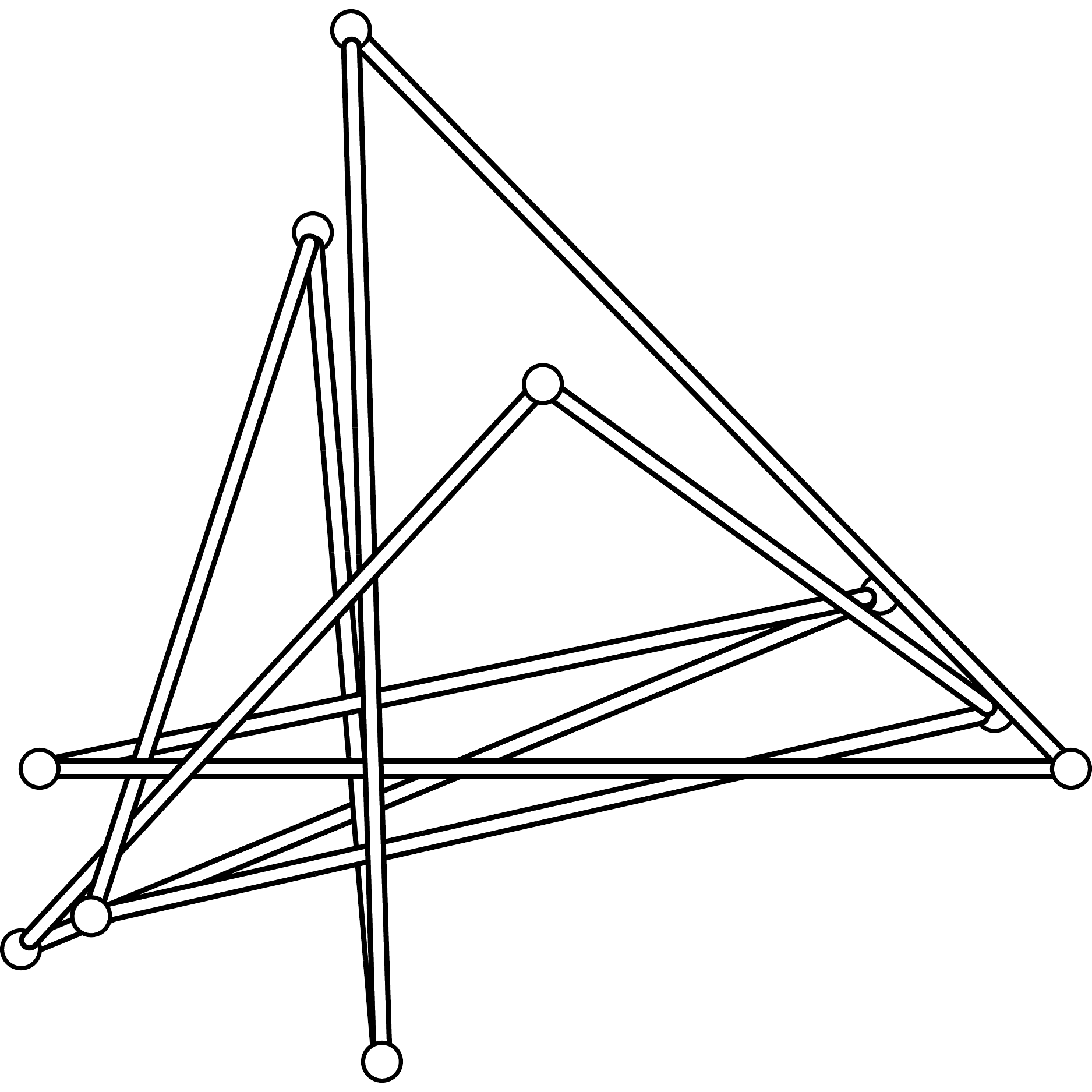}} & $0$ & $0$ & $0$ & $1$ \\
& $1000$ & $0$ & $0$ & $1$ \\
& $302$ & $716$ & $0$ & $1$ \\
& $332$ & $-284$ & $19$ & $1$ \\
& $265$ & $520$ & $-573$ & $12525679$ \\
& $50$ & $-143$ & $144$ & $129465454$ \\
& $926$ & $55$ & $-295$ & $7720205$ \\
& $488$ & $373$ & $546$ & $5545652$ \\
& $-18$ & $-175$ & $-121$ & $24812819$ \\
& $815$ & $169$ & $-554$ & $117804943$ \\
\end{tabular*}

\medskip

\begin{tabular*}{0.85\textwidth}{C{2.2in} R{.4in} R{.4in} R{.4in} | R{1in}}
\multicolumn{5}{c}{$8_{10}$} \\
\cline{1-5}\noalign{\smallskip}
\multirow{10}{*}{\includegraphics[height=1.8in]{8_10.pdf}} & $0$ & $0$ & $0$ & $1$ \\
& $1000$ & $0$ & $0$ & $1$ \\
& $262$ & $674$ & $0$ & $1$ \\
& $433$ & $-130$ & $569$ & $1$ \\
& $-197$ & $506$ & $123$ & $1$ \\
& $316$ & $-334$ & $-54$ & $5014949497$ \\
& $374$ & $225$ & $773$ & $250995749$ \\
& $41$ & $-525$ & $203$ & $370150061$ \\
& $670$ & $152$ & $-179$ & $2517248393$ \\
& $405$ & $640$ & $653$ & $3146528061$ \\
\end{tabular*}

\medskip

\begin{tabular*}{0.85\textwidth}{C{2.2in} R{.4in} R{.4in} R{.4in} | R{1in}}
\multicolumn{5}{c}{$8_{11}$} \\
\cline{1-5}\noalign{\smallskip}
\multirow{10}{*}{\includegraphics[height=1.8in]{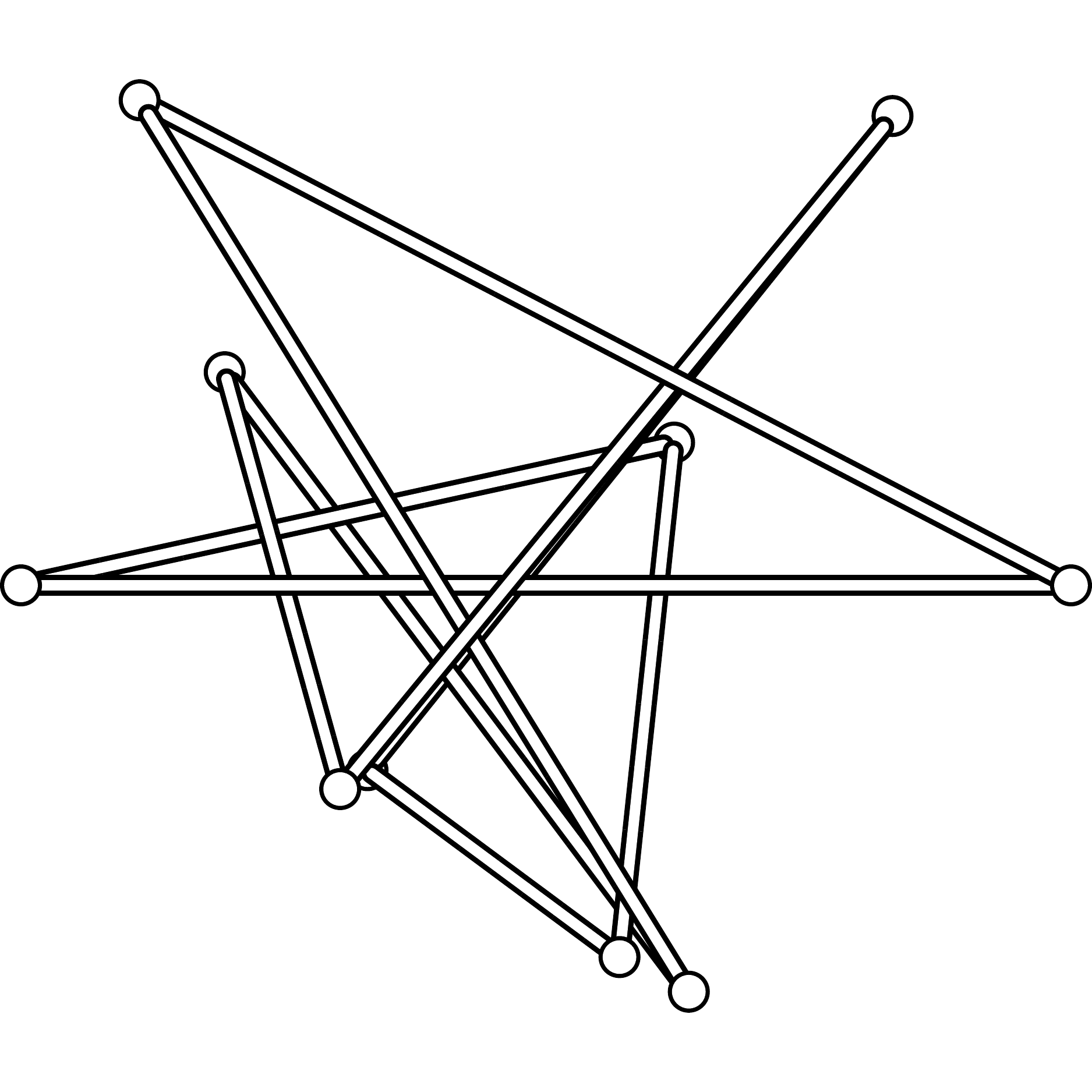}} & $0$ & $0$ & $0$ & $1$ \\
& $1000$ & $0$ & $0$ & $1$ \\
& $113$ & $462$ & $0$ & $1$ \\
& $636$ & $-387$ & $71$ & $1$ \\
& $194$ & $203$ & $-605$ & $29374052043$ \\
& $304$ & $-194$ & $307$ & $9395661114$ \\
& $830$ & $447$ & $-253$ & $53279952$ \\
& $330$ & $-176$ & $-855$ & $793938322$ \\
& $570$ & $-354$ & $99$ & $35840609905$ \\
& $622$ & $136$ & $-771$ & $103610091$ \\
\end{tabular*}

\medskip

\begin{tabular*}{0.85\textwidth}{C{2.2in} R{.4in} R{.4in} R{.4in} | R{1in}}
\multicolumn{5}{c}{$8_{12}$} \\
\cline{1-5}\noalign{\smallskip}
\multirow{10}{*}{\includegraphics[height=1.8in]{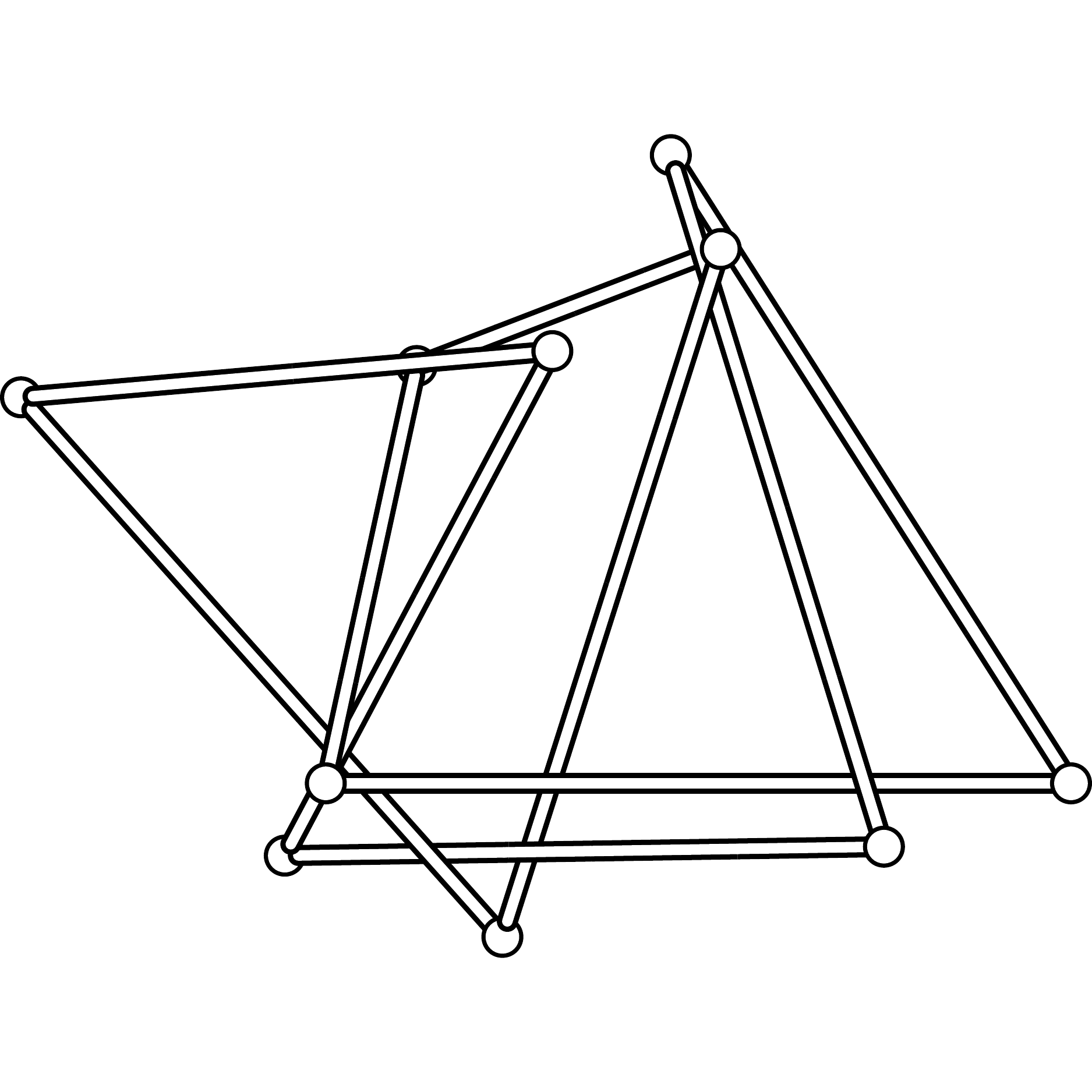}} & $0$ & $0$ & $0$ & $1$ \\
& $1000$ & $0$ & $0$ & $1$ \\
& $463$ & $843$ & $0$ & $1$ \\
& $749$ & $-85$ & $237$ & $1$ \\
& $-55$ & $-97$ & $-358$ & $14648534110$ \\
& $304$ & $580$ & $284$ & $1$ \\
& $-409$ & $518$ & $-414$ & $272953947$ \\
& $237$ & $-206$ & $-171$ & $10328384040$ \\
& $530$ & $717$ & $80$ & $6335908245$ \\
& $122$ & $560$ & $-819$ & $1443566304$ \\
\end{tabular*}

\medskip

\begin{tabular*}{0.85\textwidth}{C{2.2in} R{.4in} R{.4in} R{.4in} | R{1in}}
\multicolumn{5}{c}{$8_{13}$} \\
\cline{1-5}\noalign{\smallskip}
\multirow{10}{*}{\includegraphics[height=1.8in]{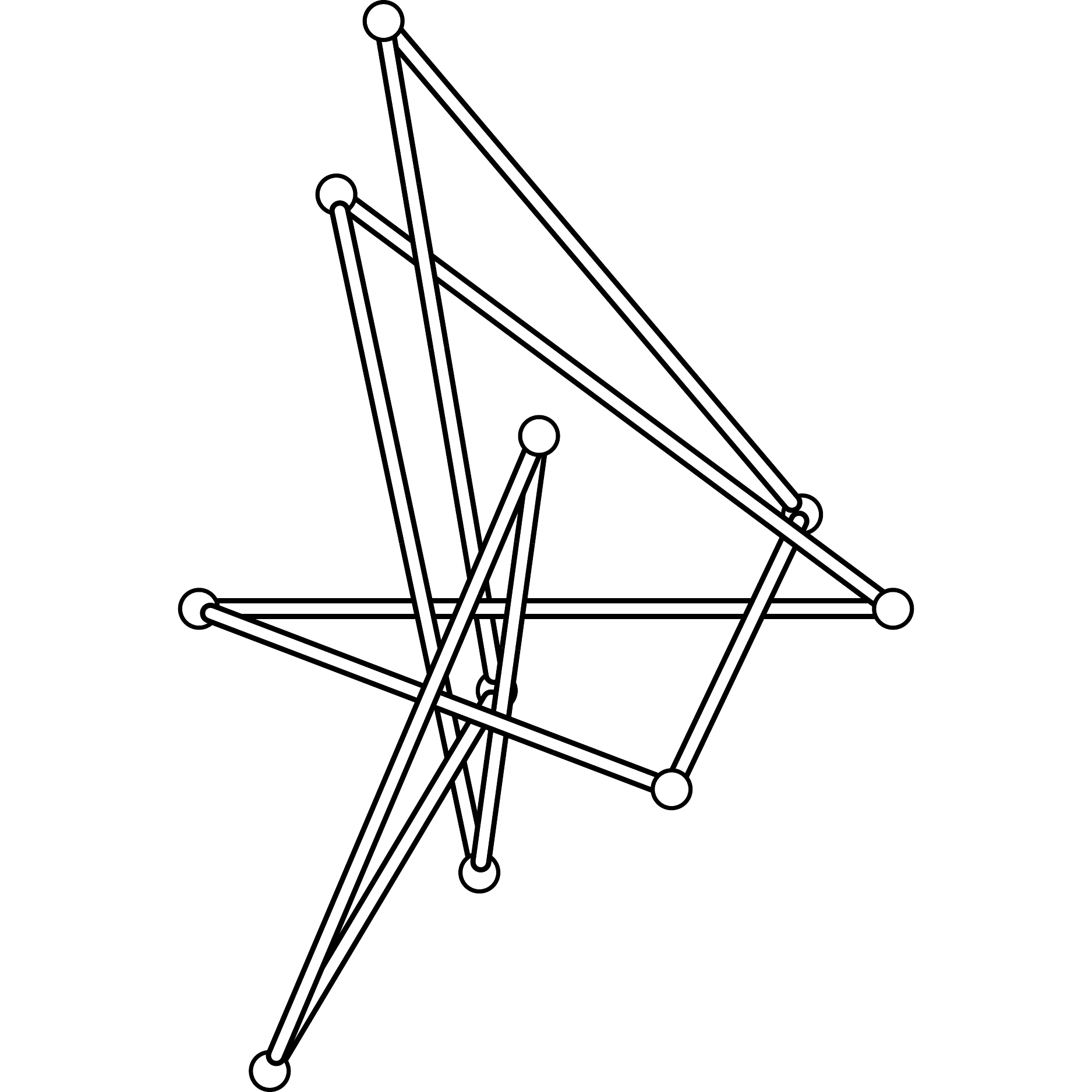}} & $0$ & $0$ & $0$ & $1$ \\
& $1000$ & $0$ & $0$ & $1$ \\
& $198$ & $597$ & $0$ & $1$ \\
& $404$ & $-380$ & $51$ & $560401500$ \\
& $490$ & $249$ & $823$ & $1$ \\
& $102$ & $-666$ & $715$ & $7$ \\
& $429$ & $-118$ & $-55$ & $265466226$ \\
& $266$ & $847$ & $147$ & $283945161$ \\
& $869$ & $136$ & $-214$ & $11280336$ \\
& $681$ & $-260$ & $685$ & $388847518$ \\
\end{tabular*}

\medskip

\begin{tabular*}{0.85\textwidth}{C{2.2in} R{.4in} R{.4in} R{.4in} | R{1in}}
\multicolumn{5}{c}{$8_{14}$} \\
\cline{1-5}\noalign{\smallskip}
\multirow{10}{*}{\includegraphics[height=1.8in]{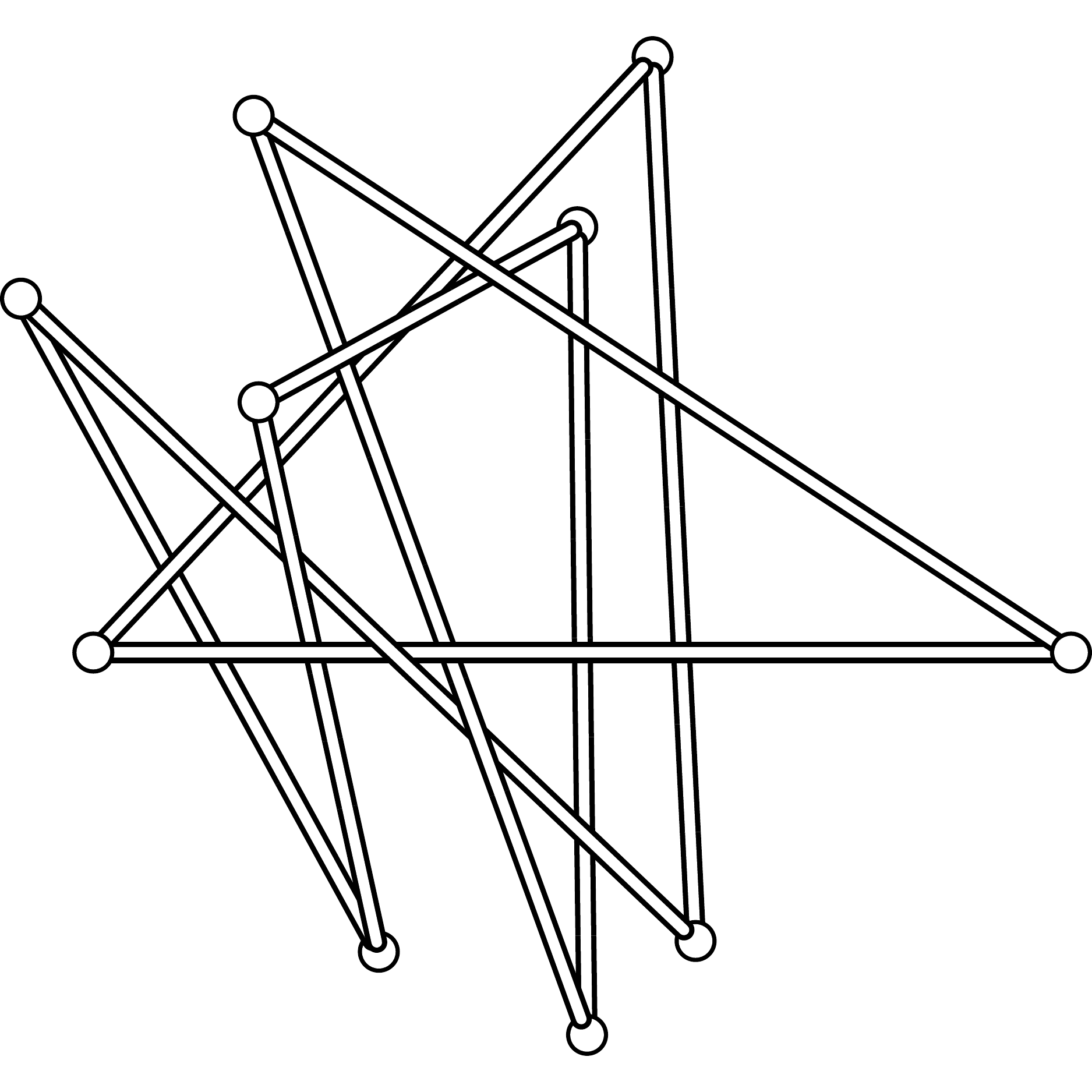}} & $0$ & $0$ & $0$ & $1$ \\
& $1000$ & $0$ & $0$ & $1$ \\
& $164$ & $549$ & $0$ & $22986044$ \\
& $505$ & $-391$ & $-14$ & $1$ \\
& $495$ & $435$ & $-577$ & $1$ \\
& $169$ & $256$ & $351$ & $112$ \\
& $292$ & $-306$ & $-467$ & $2101619$ \\
& $-74$ & $362$ & $181$ & $17826677$ \\
& $616$ & $-295$ & $-123$ & $3072051$ \\
& $572$ & $609$ & $-549$ & $9382065$ \\
\end{tabular*}

\medskip

\begin{tabular*}{0.85\textwidth}{C{2.2in} R{.4in} R{.4in} R{.4in} | R{1in}}
\multicolumn{5}{c}{$8_{15}$} \\
\cline{1-5}\noalign{\smallskip}
\multirow{10}{*}{\includegraphics[height=1.8in]{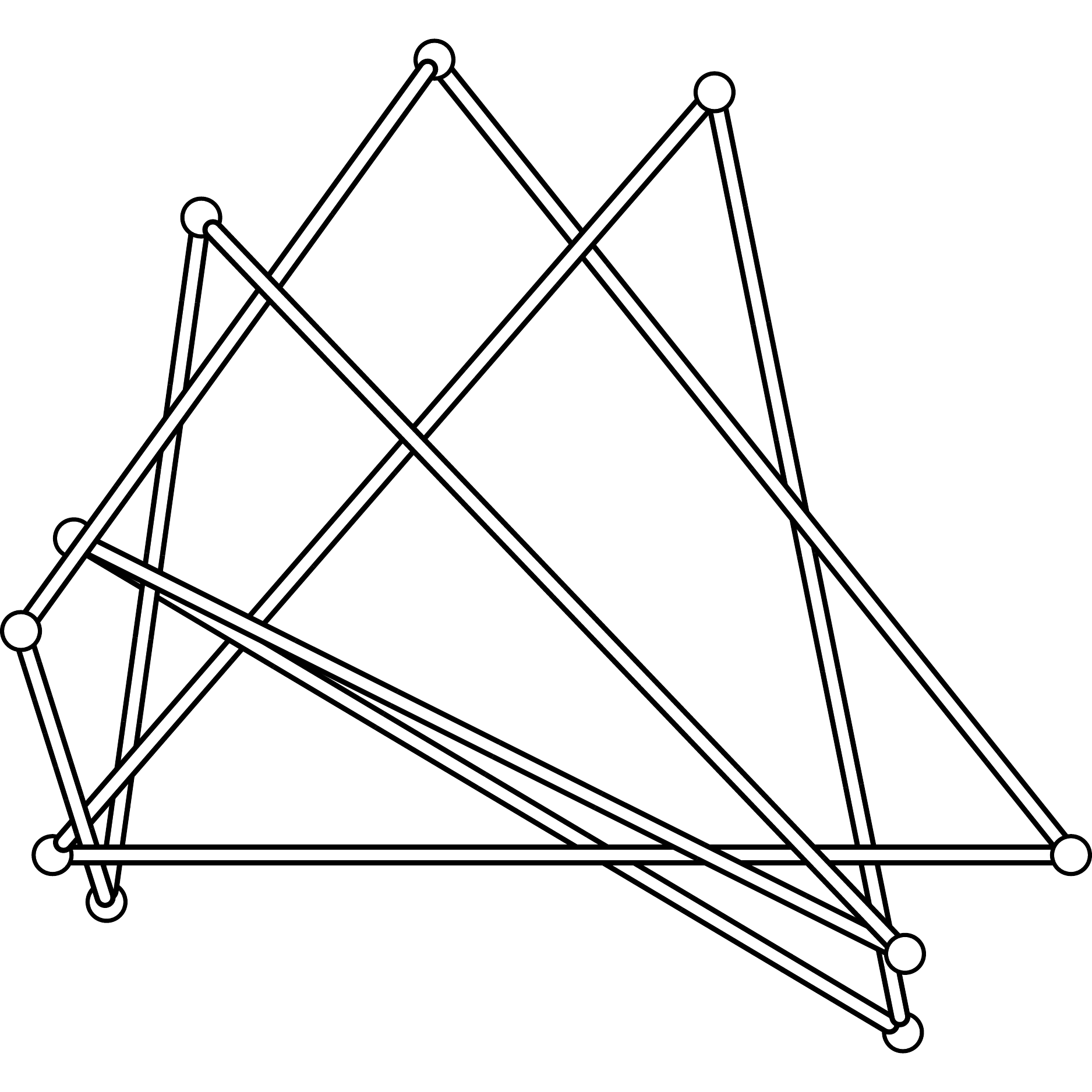}} & $0$ & $0$ & $0$ & $1$ \\
& $1000$ & $0$ & $0$ & $1197682194$ \\
& $375$ & $781$ & $0$ & $1$ \\
& $-31$ & $220$ & $722$ & $1$ \\
& $53$ & $-46$ & $-239$ & $2$ \\
& $146$ & $626$ & $496$ & $3$ \\
& $837$ & $-97$ & $517$ & $647898942$ \\
& $21$ & $311$ & $109$ & $373521904$ \\
& $835$ & $-174$ & $-208$ & $345796927$ \\
& $650$ & $749$ & $129$ & $227928351$ \\
\end{tabular*}

\medskip

\begin{tabular*}{0.85\textwidth}{C{2.2in} R{.4in} R{.4in} R{.4in} | R{1in}}
\multicolumn{5}{c}{$9_{7}$} \\
\cline{1-5}\noalign{\smallskip}
\multirow{10}{*}{\includegraphics[height=1.8in]{9_7.pdf}} & $0$ & $0$ & $0$ & $1$ \\
& $1000$ & $0$ & $0$ & $1422508933$ \\
& $407$ & $805$ & $0$ & $1$ \\
& $-275$ & $382$ & $-597$ & $1$ \\
& $703$ & $272$ & $-420$ & $60828916$ \\
& $-138$ & $168$ & $111$ & $60828918$ \\
& $756$ & $41$ & $-318$ & $928107766$ \\
& $-38$ & $648$ & $-346$ & $10477918$ \\
& $553$ & $-72$ & $20$ & $420725157$ \\
& $401$ & $915$ & $-34$ & $172196045$ \\
\end{tabular*}

\medskip

\begin{tabular*}{0.85\textwidth}{C{2.2in} R{.4in} R{.4in} R{.4in} | R{1in}}
\multicolumn{5}{c}{$9_{16}$} \\
\cline{1-5}\noalign{\smallskip}
\multirow{10}{*}{\includegraphics[height=1.8in]{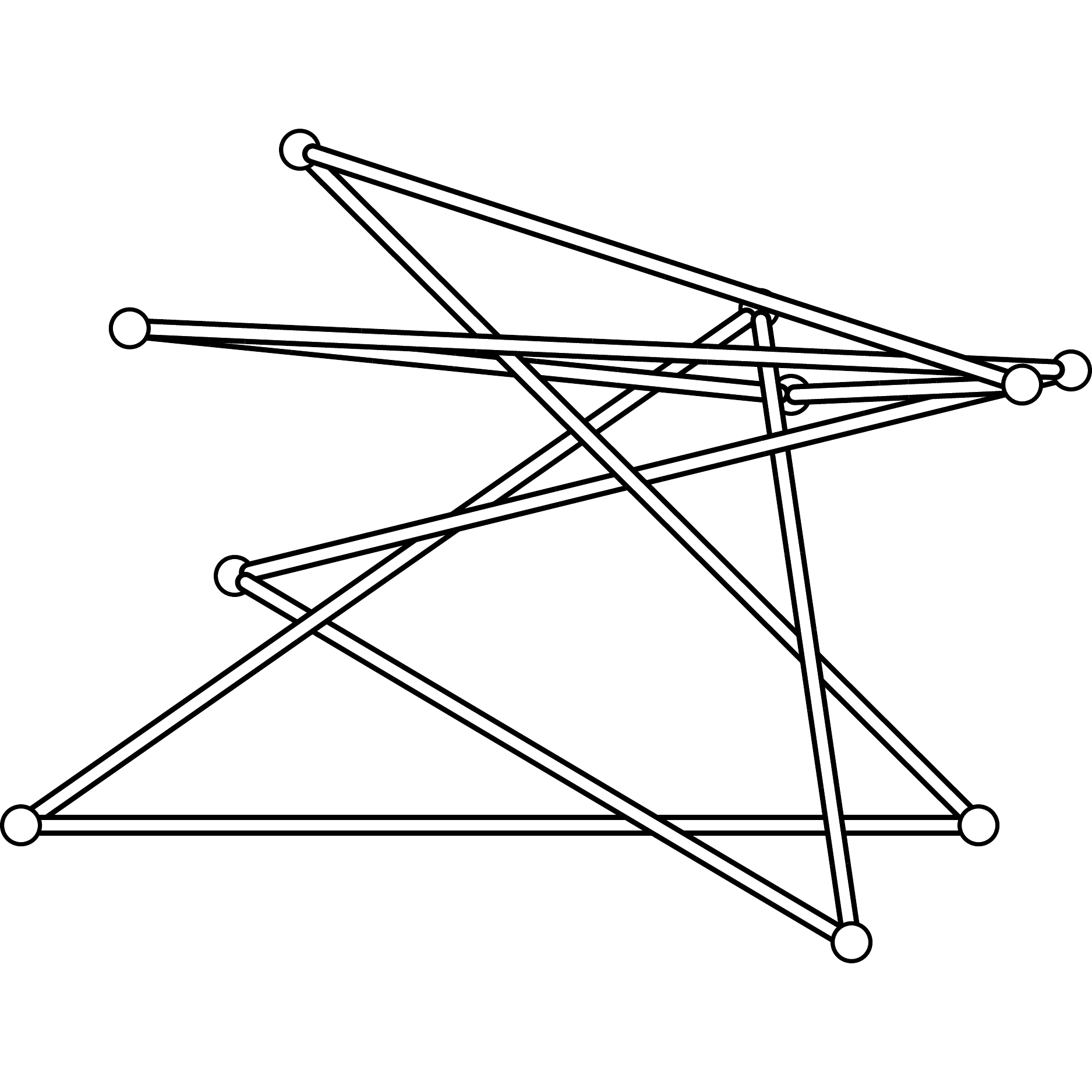}} & $0$ & $0$ & $0$ & $1$ \\
& $1000$ & $0$ & $0$ & $1326961294$ \\
& $291$ & $706$ & $0$ & $1$ \\
& $1046$ & $460$ & $609$ & $1$ \\
& $804$ & $450$ & $-361$ & $1338798152$ \\
& $114$ & $519$ & $359$ & $1$ \\
& $1097$ & $475$ & $180$ & $47421860$ \\
& $223$ & $261$ & $-257$ & $618252967$ \\
& $867$ & $-122$ & $405$ & $442800719$ \\
& $771$ & $539$ & $-339$ & $603197992$ \\
\end{tabular*}

\medskip

\begin{tabular*}{0.85\textwidth}{C{2.2in} R{.4in} R{.4in} R{.4in} | R{1in}}
\multicolumn{5}{c}{$9_{20}$} \\
\cline{1-5}\noalign{\smallskip}
\multirow{10}{*}{\includegraphics[height=1.8in]{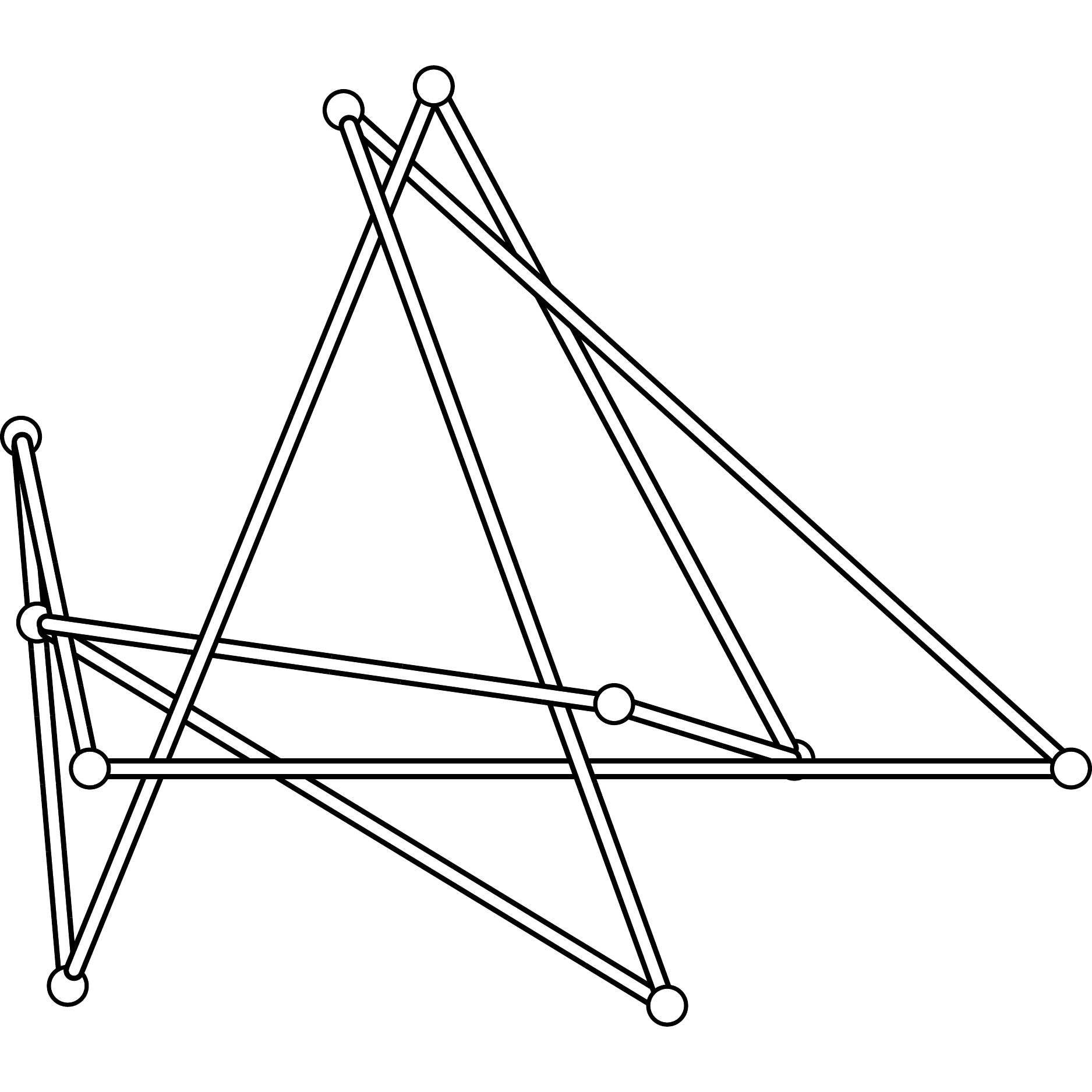}} & $0$ & $0$ & $0$ & $1$ \\
& $1000$ & $0$ & $0$ & $1$ \\
& $259$ & $671$ & $0$ & $380655073$ \\
& $588$ & $-242$ & $241$ & $1$ \\
& $-54$ & $149$ & $-419$ & $3$ \\
& $534$ & $66$ & $386$ & $485280290$ \\
& $718$ & $9$ & $-596$ & $15474847$ \\
& $350$ & $695$ & $32$ & $23424346$ \\
& $-23$ & $-222$ & $-111$ & $298844208$ \\
& $-70$ & $338$ & $-938$ & $356297277$ \\
\end{tabular*}

\medskip

\begin{tabular*}{0.85\textwidth}{C{2.2in} R{.4in} R{.4in} R{.4in} | R{1in}}
\multicolumn{5}{c}{$9_{26}$} \\
\cline{1-5}\noalign{\smallskip}
\multirow{10}{*}{\includegraphics[height=1.8in]{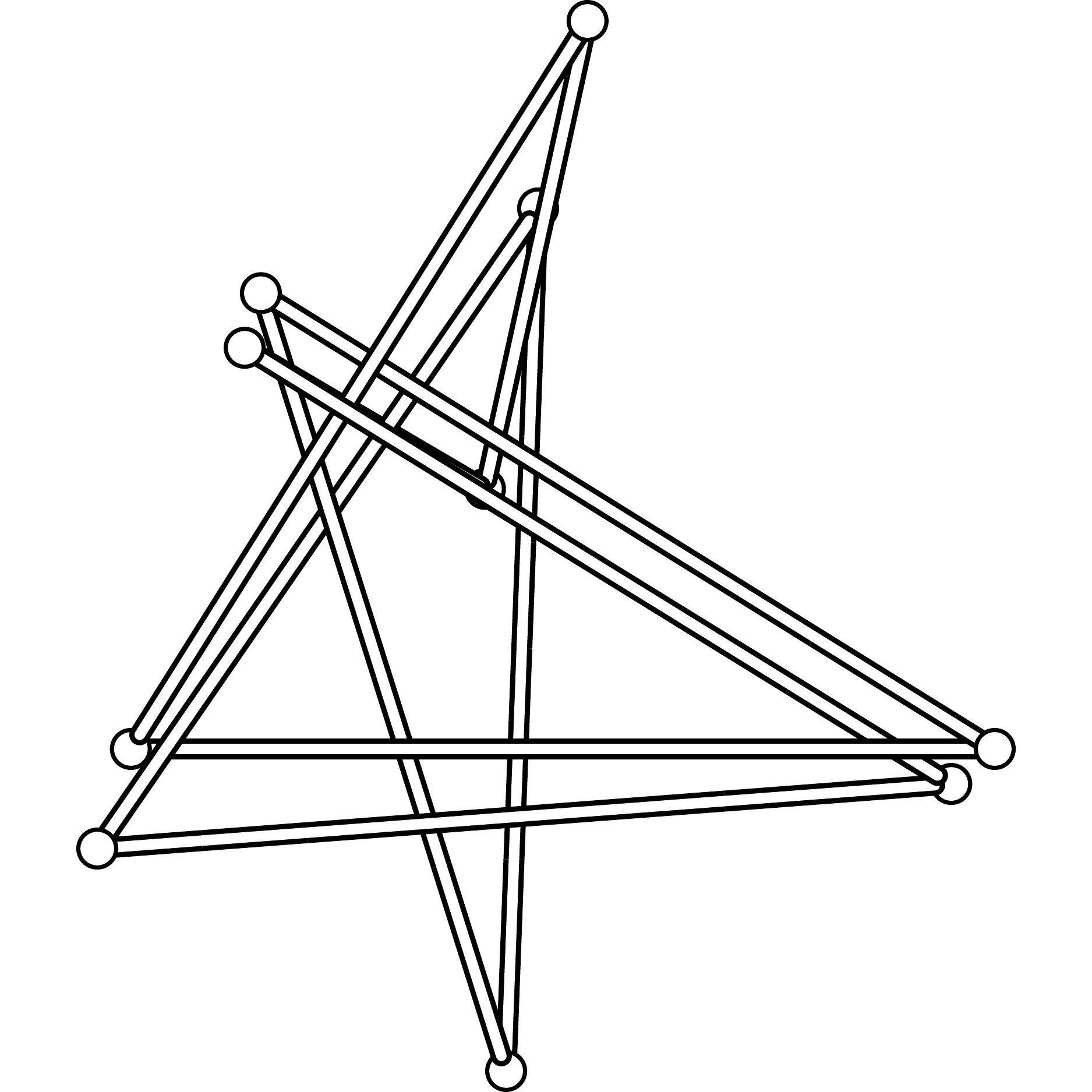}} & $0$ & $0$ & $0$ & $1$ \\
& $1000$ & $0$ & $0$ & $1$ \\
& $151$ & $528$ & $0$ & $3396197047$ \\
& $434$ & $-373$ & $-330$ & $1$ \\
& $471$ & $626$ & $-366$ & $1$ \\
& $-39$ & $-115$ & $70$ & $2906507747$ \\
& $950$ & $-40$ & $-60$ & $8905548$ \\
& $132$ & $464$ & $216$ & $293002848$ \\
& $410$ & $301$ & $-731$ & $99723212$ \\
& $529$ & $843$ & $101$ & $3762353240$ \\
\end{tabular*}

\medskip

\begin{tabular*}{0.85\textwidth}{C{2.2in} R{.4in} R{.4in} R{.4in} | R{1in}}
\multicolumn{5}{c}{$9_{28}$} \\
\cline{1-5}\noalign{\smallskip}
\multirow{10}{*}{\includegraphics[height=1.8in]{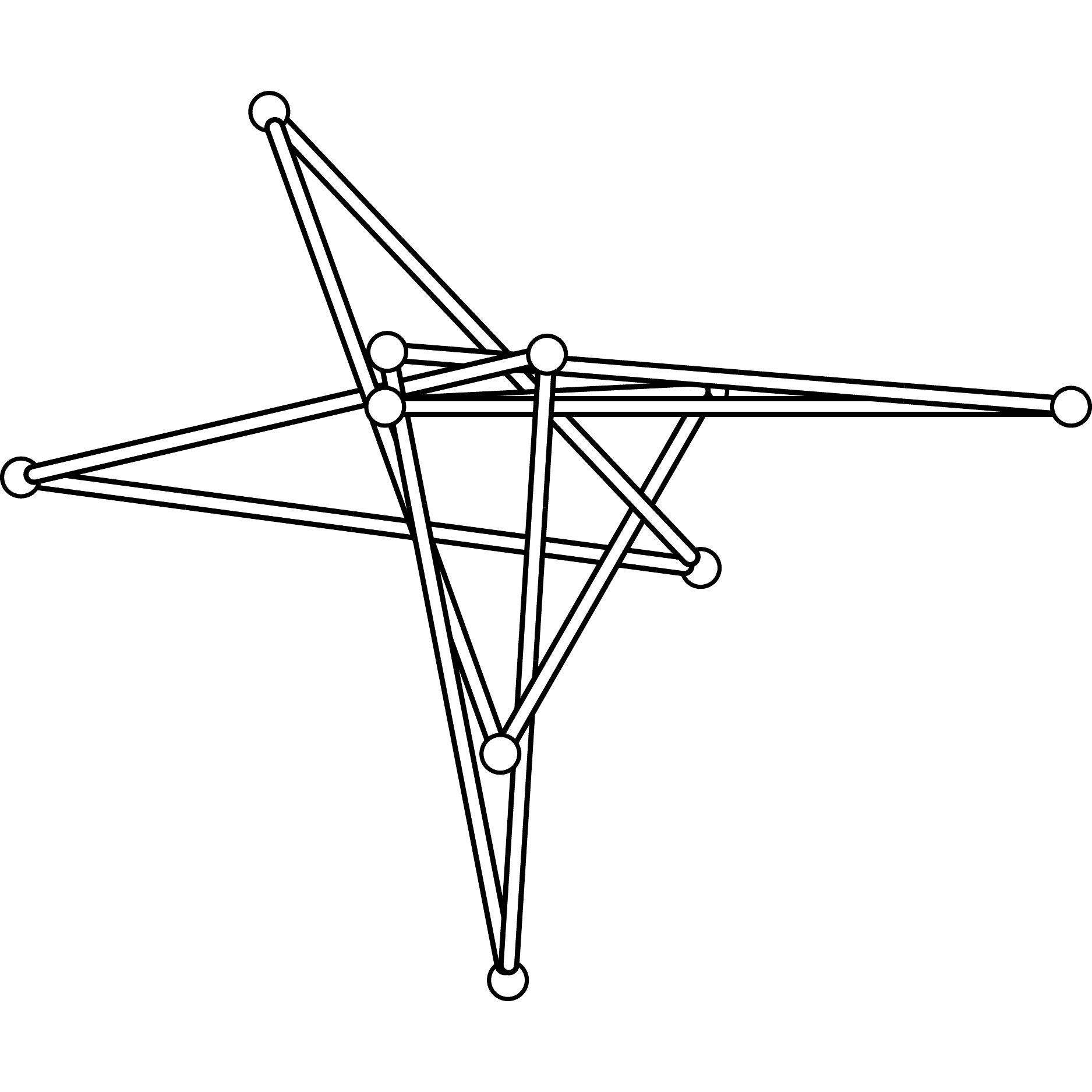}} & $0$ & $0$ & $0$ & $1245264519$ \\
& $1000$ & $0$ & $0$ & $1$ \\
& $3$ & $80$ & $0$ & $2262566461$ \\
& $179$ & $-836$ & $-361$ & $1$ \\
& $236$ & $76$ & $45$ & $1$ \\
& $-531$ & $-103$ & $-570$ & $1$ \\
& $460$ & $-235$ & $-600$ & $2609807812$ \\
& $-169$ & $430$ & $-197$ & $282221245$ \\
& $168$ & $-506$ & $-89$ & $132784622$ \\
& $471$ & $23$ & $-882$ & $112457999$ \\
\end{tabular*}

\medskip

\begin{tabular*}{0.85\textwidth}{C{2.2in} R{.4in} R{.4in} R{.4in} | R{1in}}
\multicolumn{5}{c}{$9_{32}$} \\
\cline{1-5}\noalign{\smallskip}
\multirow{10}{*}{\includegraphics[height=1.8in]{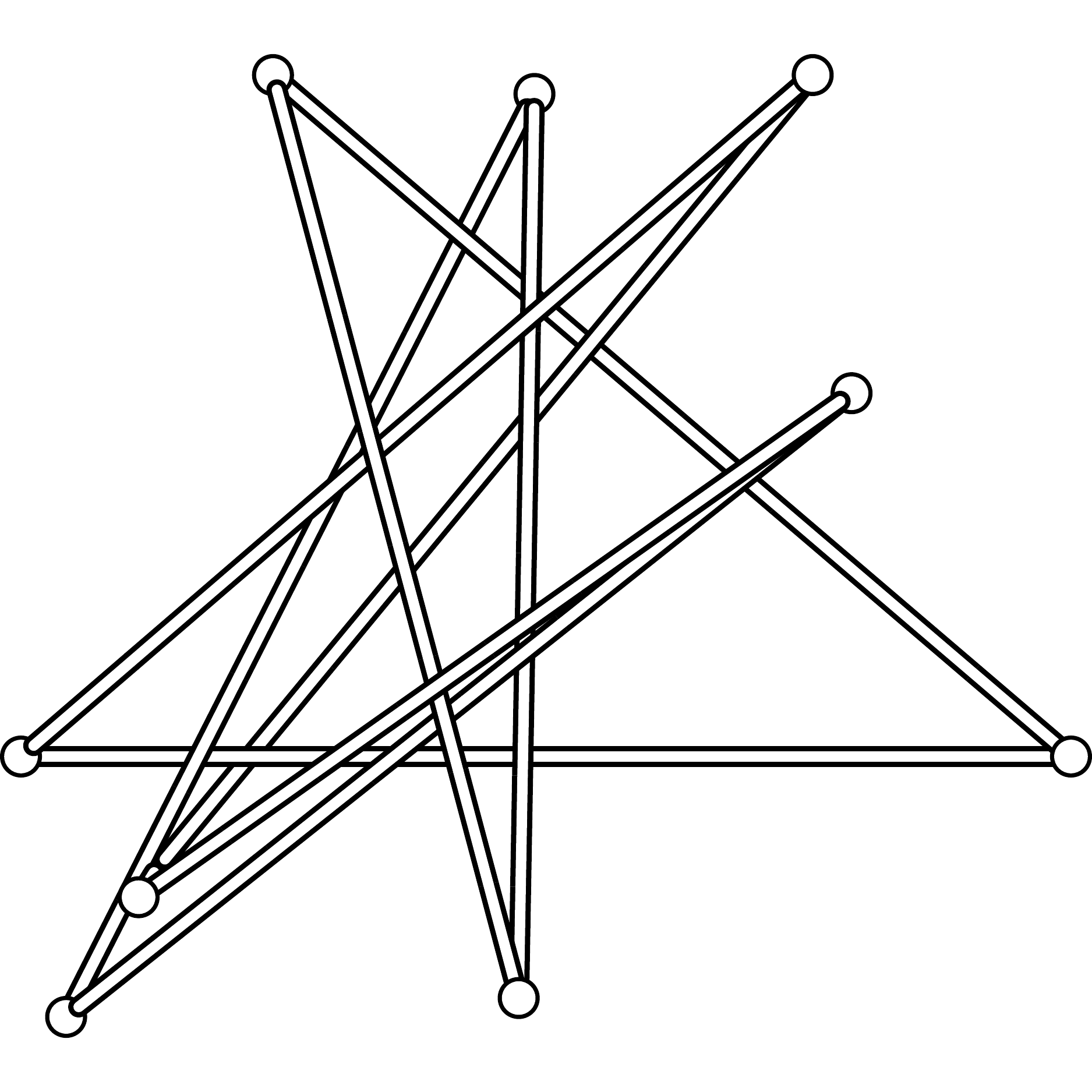}} & $0$ & $0$ & $0$ & $1$ \\
& $1000$ & $0$ & $0$ & $126894982$ \\
& $240$ & $649$ & $0$ & $1$ \\
& $474$ & $-230$ & $415$ & $1$ \\
& $489$ & $631$ & $-94$ & $1$ \\
& $43$ & $-248$ & $76$ & $2020047830$ \\
& $791$ & $346$ & $373$ & $163099052$ \\
& $112$ & $-134$ & $929$ & $275776694$ \\
& $127$ & $-110$ & $-71$ & $231257739$ \\
& $754$ & $649$ & $105$ & $1836124528$ \\
\end{tabular*}

\medskip

\begin{tabular*}{0.85\textwidth}{C{2.2in} R{.4in} R{.4in} R{.4in} | R{1in}}
\multicolumn{5}{c}{$9_{33}$} \\
\cline{1-5}\noalign{\smallskip}
\multirow{10}{*}{\includegraphics[height=1.8in]{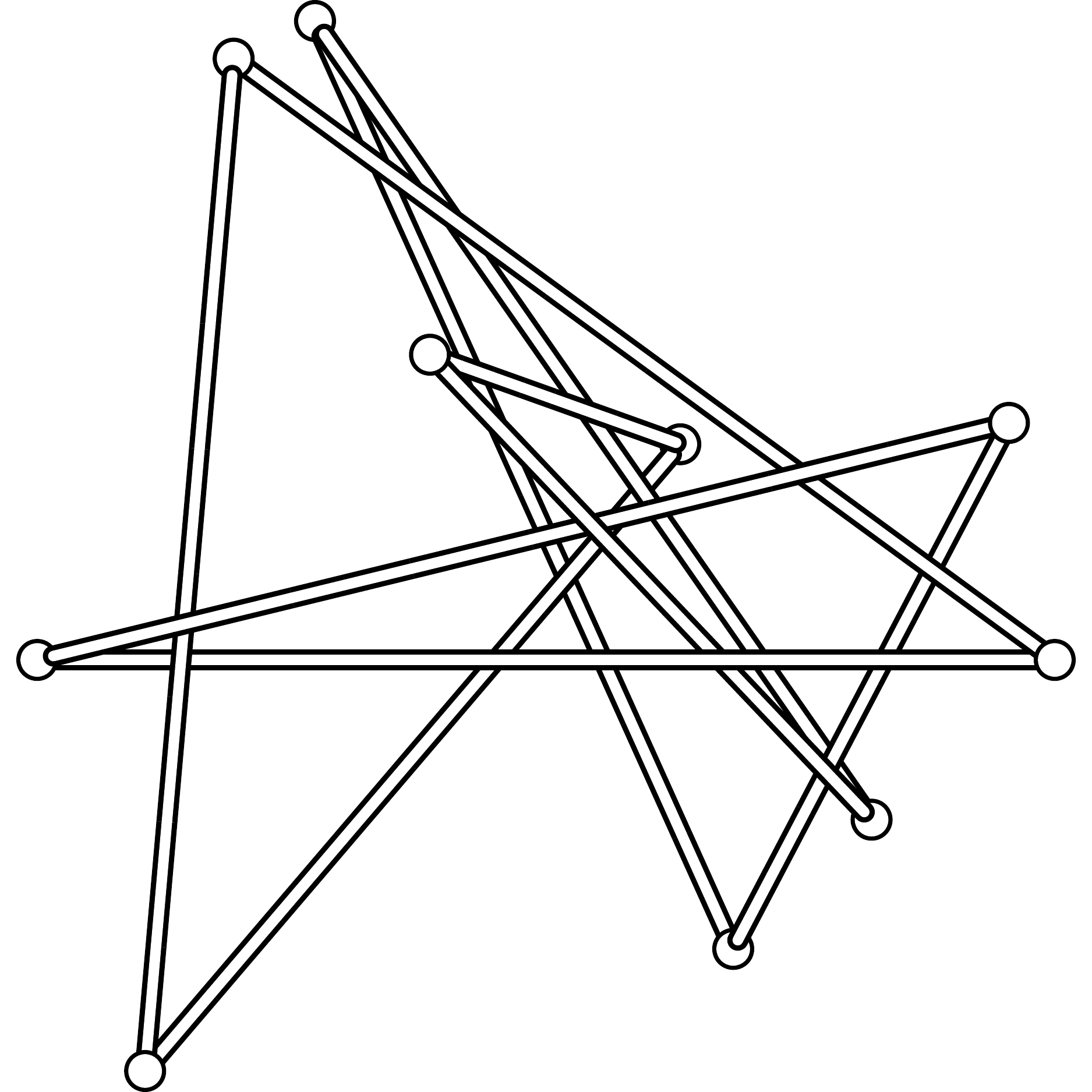}} & $0$ & $0$ & $0$ & $1$ \\
& $1000$ & $0$ & $0$ & $1669432578$ \\
& $193$ & $591$ & $0$ & $1723299318$ \\
& $106$ & $-404$ & $29$ & $1$ \\
& $632$ & $212$ & $-556$ & $1$ \\
& $386$ & $300$ & $409$ & $1$ \\
& $820$ & $-157$ & $-367$ & $747563966$ \\
& $273$ & $628$ & $-658$ & $2156250645$ \\
& $684$ & $-284$ & $-630$ & $275261458$ \\
& $955$ & $233$ & $182$ & $24330471$ \\
\end{tabular*}

\medskip

\begin{tabular*}{0.85\textwidth}{C{2.2in} R{.4in} R{.4in} R{.4in} | R{1in}}
\multicolumn{5}{c}{$10_{76}$} \\
\cline{1-5}\noalign{\smallskip}
\multirow{12}{*}{\includegraphics[height=1.8in]{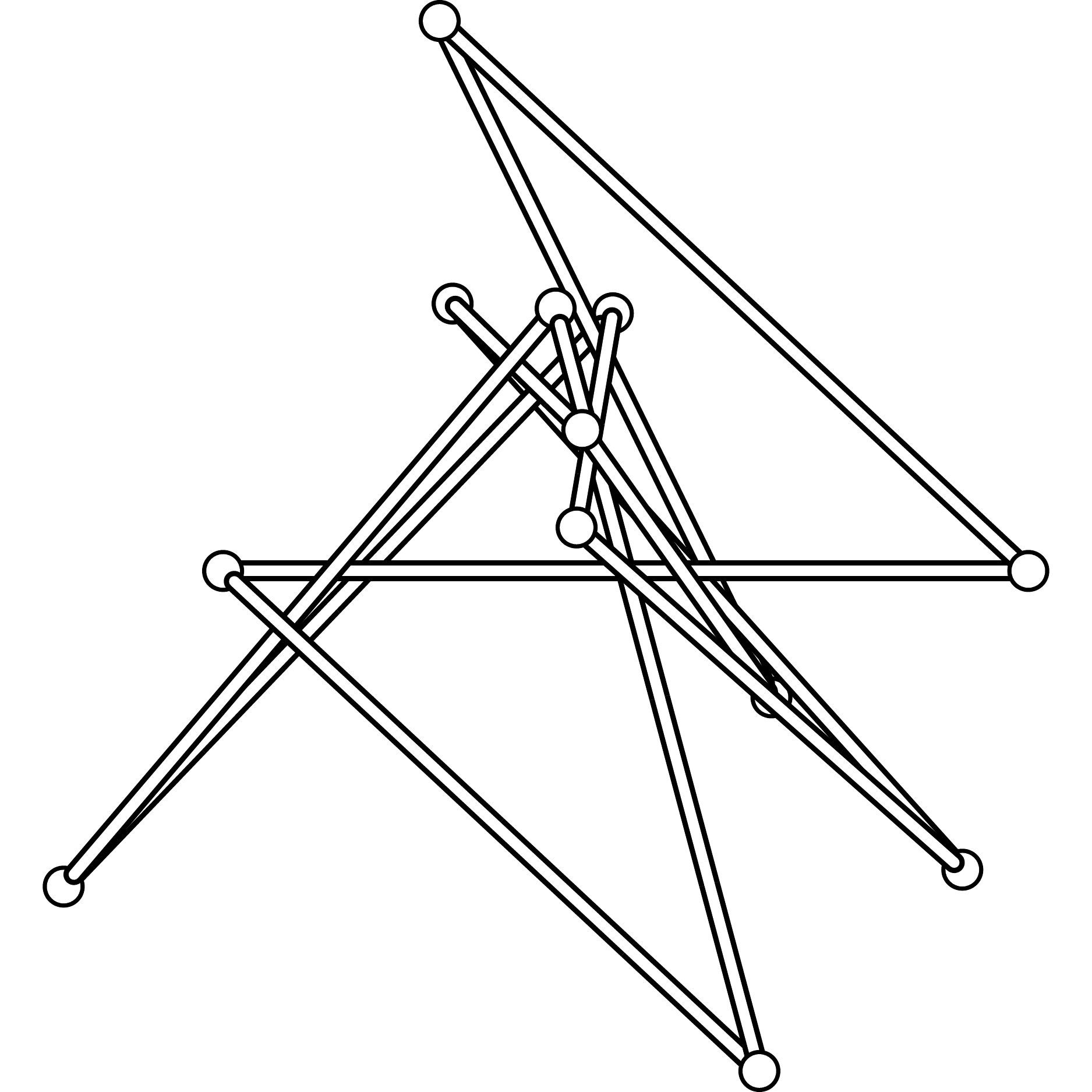}} & $0$ & $0$ & $0$ & $117631072$ \\
& $1000$ & $0$ & $0$ & $1$ \\
& $269$ & $683$ & $0$ & $1$ \\
& $681$ & $-157$ & $-354$ & $1$ \\
& $446$ & $175$ & $560$ & $225822505$ \\
& $285$ & $332$ & $-415$ & $1$ \\
& $918$ & $-371$ & $-89$ & $248306906$ \\
& $439$ & $54$ & $679$ & $11$ \\
& $484$ & $320$ & $-284$ & $1236977$ \\
& $-198$ & $-392$ & $-118$ & $22138732$ \\
& $413$ & $326$ & $215$ & $106439303$ \\
& $666$ & $-621$ & $413$ & $37697596$ \\
\end{tabular*}

\medskip

\begin{tabular*}{0.85\textwidth}{C{2.2in} R{.4in} R{.4in} R{.4in} | R{1in}}
\multicolumn{5}{c}{$13n_{226}$} \\
\cline{1-5}\noalign{\smallskip}
\multirow{12}{*}{\includegraphics[height=1.8in]{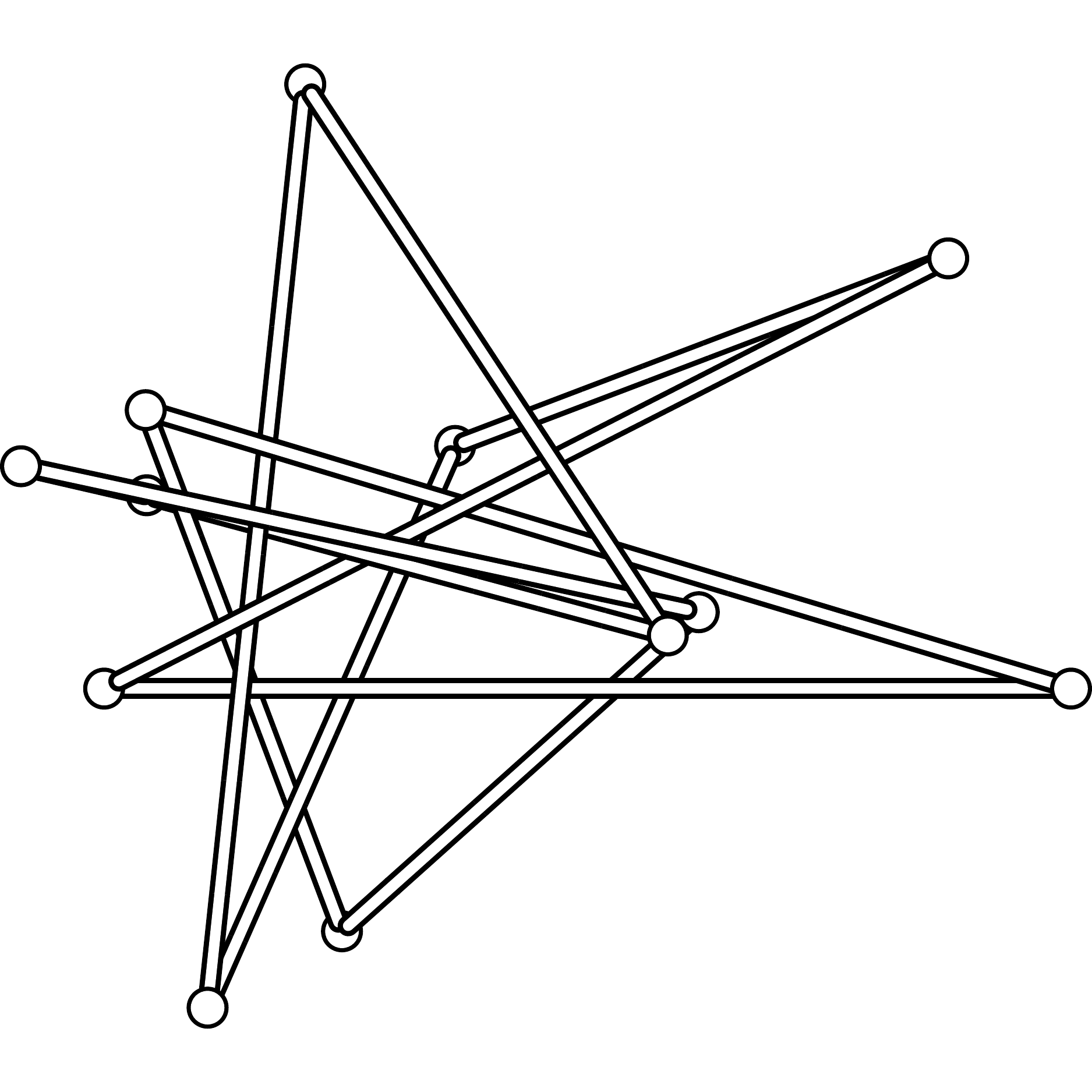}} & $0$ & $0$ & $0$ & $1$ \\
& $1000$ & $0$ & $0$ & $1$ \\
& $43$ & $288$ & $0$ & $1$ \\
& $246$ & $-251$ & $-817$ & $5248415052$ \\
& $615$ & $79$ & $52$ & $3055101237$ \\
& $-86$ & $230$ & $749$ & $1$ \\
& $44$ & $200$ & $-242$ & $1$ \\
& $583$ & $55$ & $588$ & $1$ \\
& $208$ & $625$ & $-144$ & $319244810$ \\
& $107$ & $-330$ & $136$ & $1047123048$ \\
& $363$ & $251$ & $-637$ & $727918502$ \\
& $873$ & $445$ & $201$ & $4590332083$ \\
\end{tabular*}

\medskip

\begin{tabular*}{0.85\textwidth}{C{2.2in} R{.4in} R{.4in} R{.4in} | R{1in}}
\multicolumn{5}{c}{$13n_{328}$} \\
\cline{1-5}\noalign{\smallskip}
\multirow{12}{*}{\includegraphics[height=1.8in]{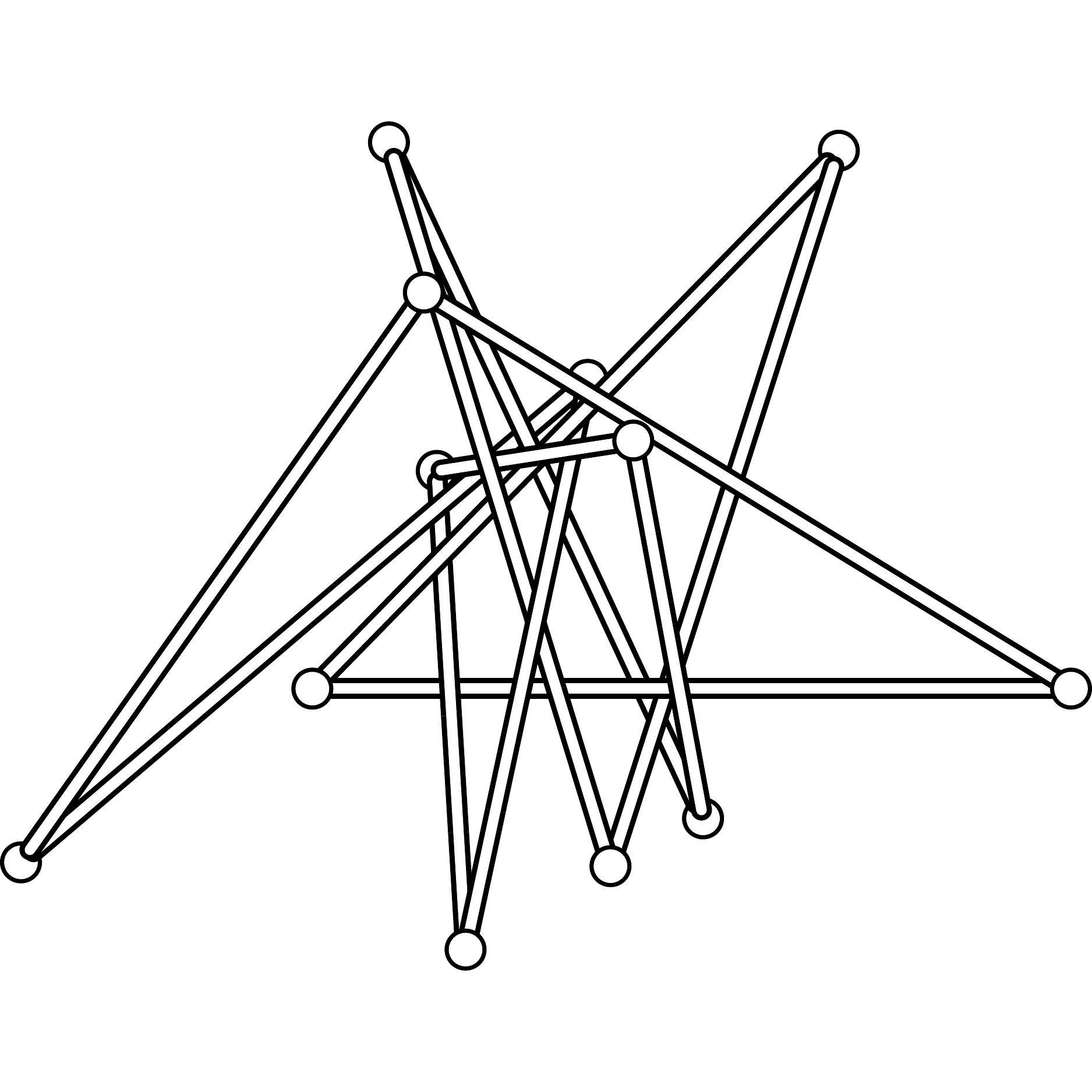}} & $0$ & $0$ & $0$ & $1$ \\
& $1000$ & $0$ & $0$ & $50328991$ \\
& $147$ & $522$ & $0$ & $104647397$ \\
& $-384$ & $-229$ & $-391$ & $1$ \\
& $363$ & $407$ & $-197$ & $1$ \\
& $202$ & $-344$ & $444$ & $1$ \\
& $164$ & $287$ & $-331$ & $1$ \\
& $423$ & $327$ & $634$ & $57317843$ \\
& $515$ & $-171$ & $-228$ & $2800943$ \\
& $101$ & $720$ & $-45$ & $32278655$ \\
& $393$ & $-234$ & $18$ & $21899391$ \\
& $694$ & $709$ & $-123$ & $31561413$ \\
\end{tabular*}

\medskip

\begin{tabular*}{0.85\textwidth}{C{2.2in} R{.4in} R{.4in} R{.4in} | R{1in}}
\multicolumn{5}{c}{$13n_{342}$} \\
\cline{1-5}\noalign{\smallskip}
\multirow{12}{*}{\includegraphics[height=1.8in]{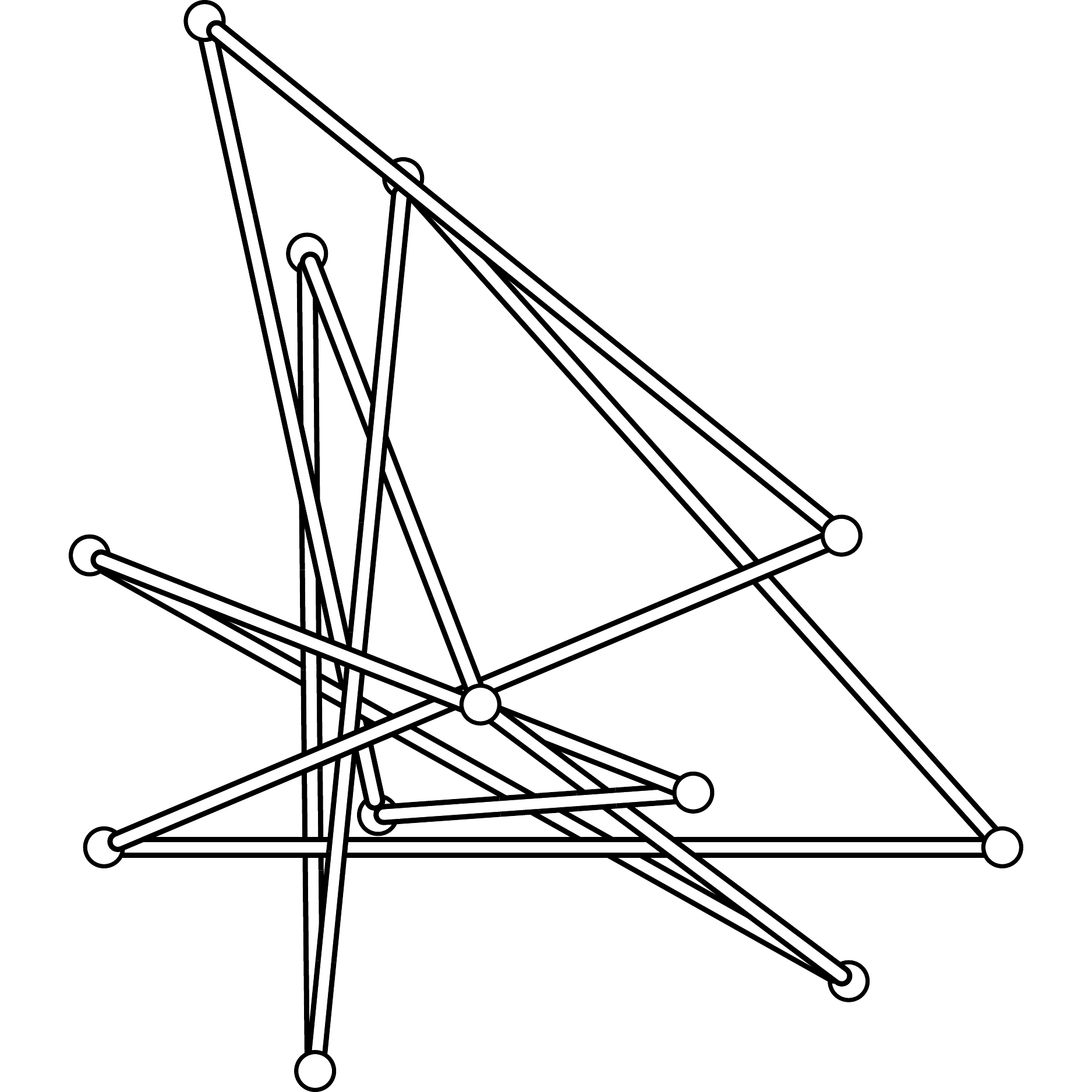}} & $0$ & $0$ & $0$ & $1$ \\
& $1000$ & $0$ & $0$ & $1$ \\
& $333$ & $745$ & $0$ & $1$ \\
& $235$ & $-249$ & $44$ & $1$ \\
& $226$ & $661$ & $-370$ & $1$ \\
& $419$ & $159$ & $472$ & $1$ \\
& $829$ & $-149$ & $-386$ & $2538632228$ \\
& $-16$ & $325$ & $-139$ & $1$ \\
& $656$ & $61$ & $553$ & $1239769501$ \\
& $305$ & $36$ & $-384$ & $1697104582$ \\
& $112$ & $920$ & $43$ & $715303757$ \\
& $821$ & $347$ & $454$ & $2126201923$ \\
\end{tabular*}

\medskip

\begin{tabular*}{0.85\textwidth}{C{2.2in} R{.4in} R{.4in} R{.4in} | R{1in}}
\multicolumn{5}{c}{$13n_{343}$} \\
\cline{1-5}\noalign{\smallskip}
\multirow{12}{*}{\includegraphics[height=1.8in]{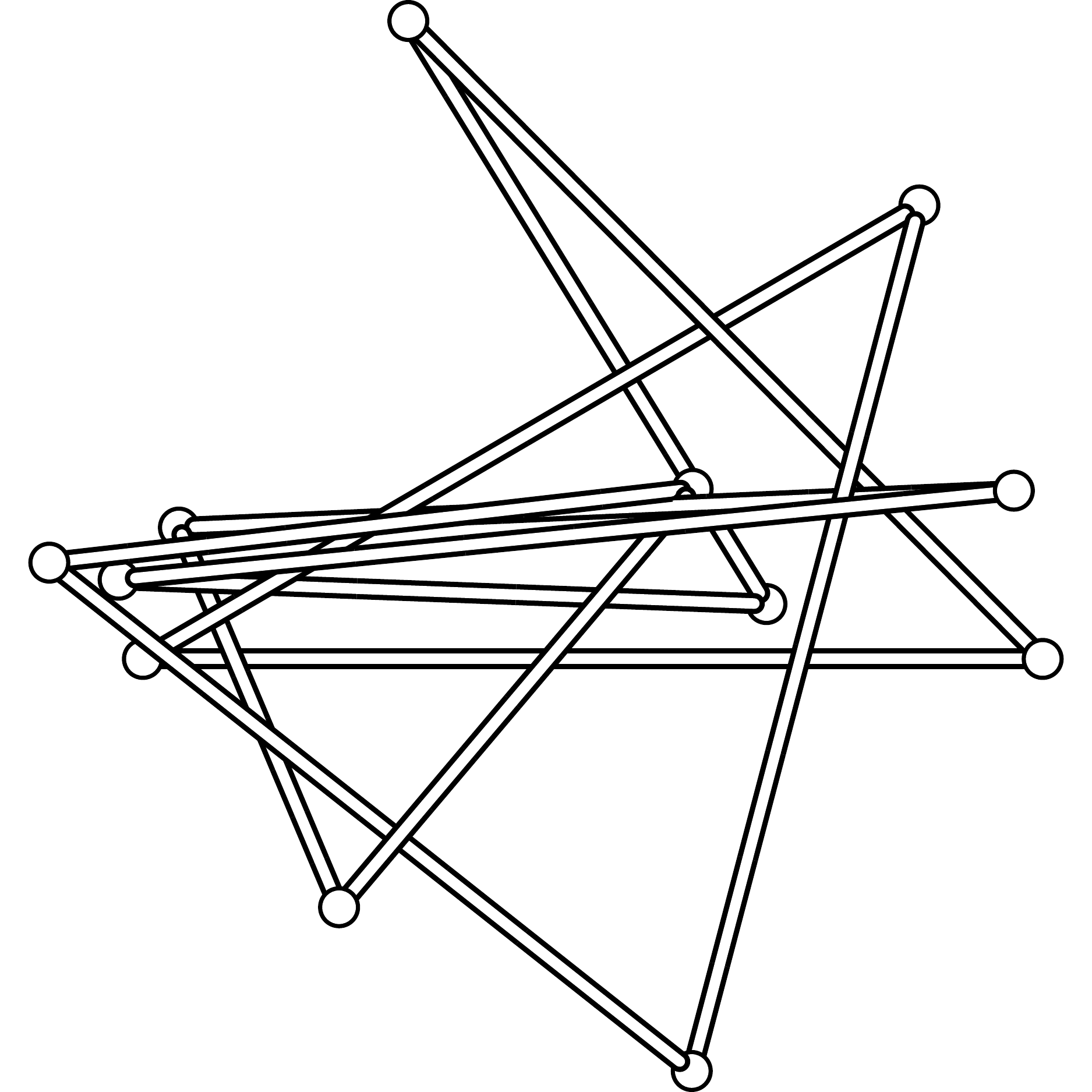}} & $0$ & $0$ & $0$ & $1$ \\
& $1000$ & $0$ & $0$ & $1$ \\
& $295$ & $709$ & $0$ & $2817234095$ \\
& $693$ & $61$ & $-649$ & $1$ \\
& $-27$ & $88$ & $44$ & $1$ \\
& $968$ & $187$ & $67$ & $1$ \\
& $40$ & $146$ & $-302$ & $1$ \\
& $218$ & $-276$ & $586$ & $1$ \\
& $611$ & $189$ & $-207$ & $1681541385$ \\
& $-104$ & $107$ & $487$ & $1751742527$ \\
& $610$ & $-458$ & $73$ & $240644281$ \\
& $863$ & $504$ & $-27$ & $1472658259$ \\
\end{tabular*}

\medskip

\begin{tabular*}{0.85\textwidth}{C{2.2in} R{.4in} R{.4in} R{.4in} | R{1in}}
\multicolumn{5}{c}{$13n_{350}$} \\
\cline{1-5}\noalign{\smallskip}
\multirow{12}{*}{\includegraphics[height=1.8in]{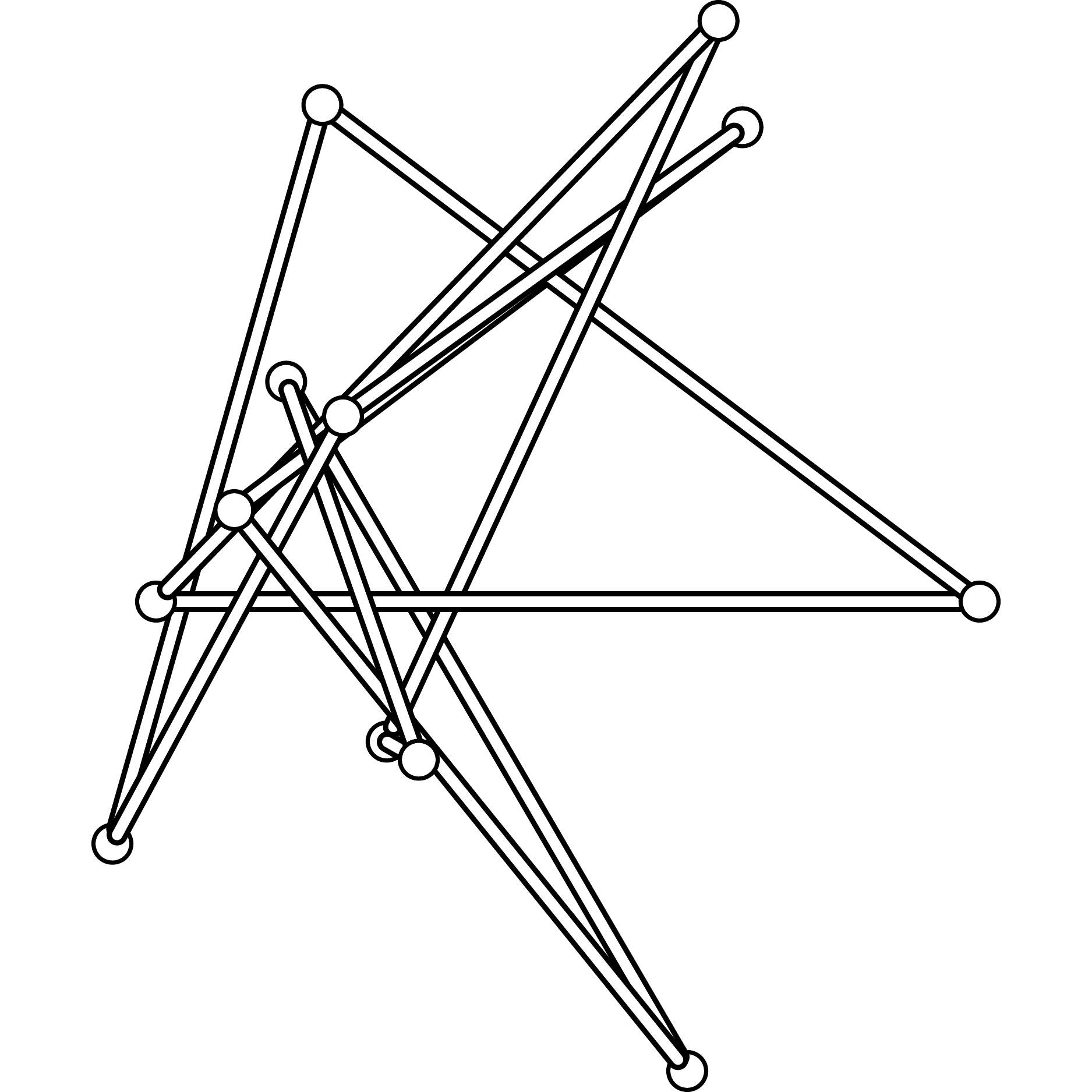}} & $0$ & $0$ & $0$ & $1$ \\
& $1000$ & $0$ & $0$ & $1$ \\
& $202$ & $603$ & $0$ & $1$ \\
& $-53$ & $-294$ & $-361$ & $48569447566$ \\
& $227$ & $225$ & $446$ & $1$ \\
& $712$ & $576$ & $-354$ & $1$ \\
& $95$ & $111$ & $281$ & $1$ \\
& $645$ & $-570$ & $-203$ & $1$ \\
& $158$ & $267$ & $46$ & $236541501$ \\
& $319$ & $-192$ & $920$ & $31084597980$ \\
& $280$ & $-170$ & $-79$ & $28872747294$ \\
& $683$ & $705$ & $189$ & $1044429945$ \\
\end{tabular*}

\medskip

\begin{tabular*}{0.85\textwidth}{C{2.2in} R{.4in} R{.4in} R{.4in} | R{1in}}
\multicolumn{5}{c}{$13n_{512}$} \\
\cline{1-5}\noalign{\smallskip}
\multirow{12}{*}{\includegraphics[height=1.8in]{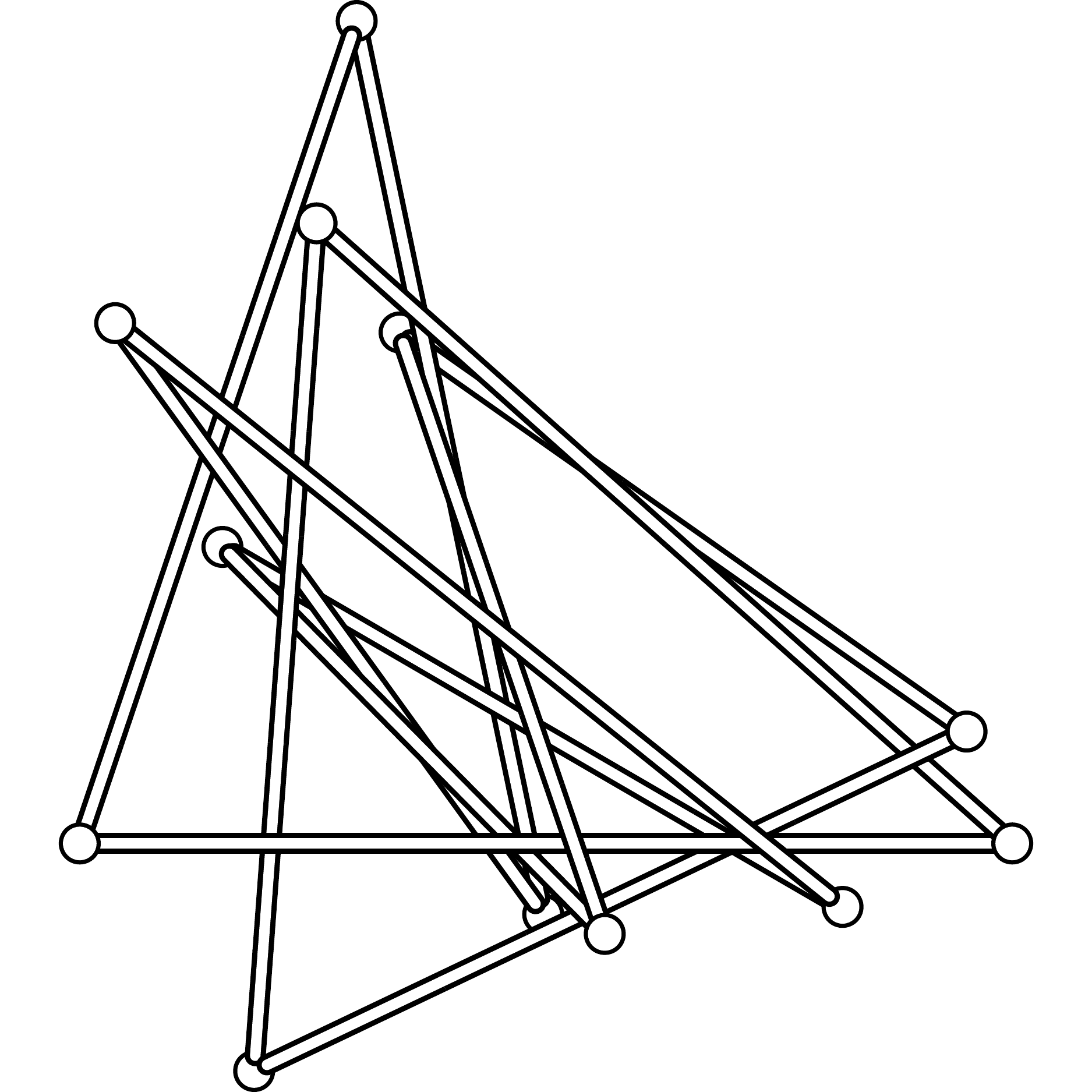}} & $0$ & $0$ & $0$ & $1$ \\
& $1000$ & $0$ & $0$ & $22492829556$ \\
& $254$ & $665$ & $0$ & $1$ \\
& $187$ & $-244$ & $-410$ & $1$ \\
& $951$ & $120$ & $123$ & $1$ \\
& $343$ & $548$ & $-545$ & $1$ \\
& $563$ & $-97$ & $187$ & $1388685548$ \\
& $153$ & $318$ & $-626$ & $1$ \\
& $818$ & $-68$ & $14$ & $20355556367$ \\
& $38$ & $558$ & $52$ & $109085245$ \\
& $497$ & $-76$ & $-572$ & $1552749788$ \\
& $297$ & $882$ & $-365$ & $93145719$ \\
\end{tabular*}

\medskip

\begin{tabular*}{0.85\textwidth}{C{2.2in} R{.4in} R{.4in} R{.4in} | R{1in}}
\multicolumn{5}{c}{$13n_{973}$} \\
\cline{1-5}\noalign{\smallskip}
\multirow{12}{*}{\includegraphics[height=1.8in]{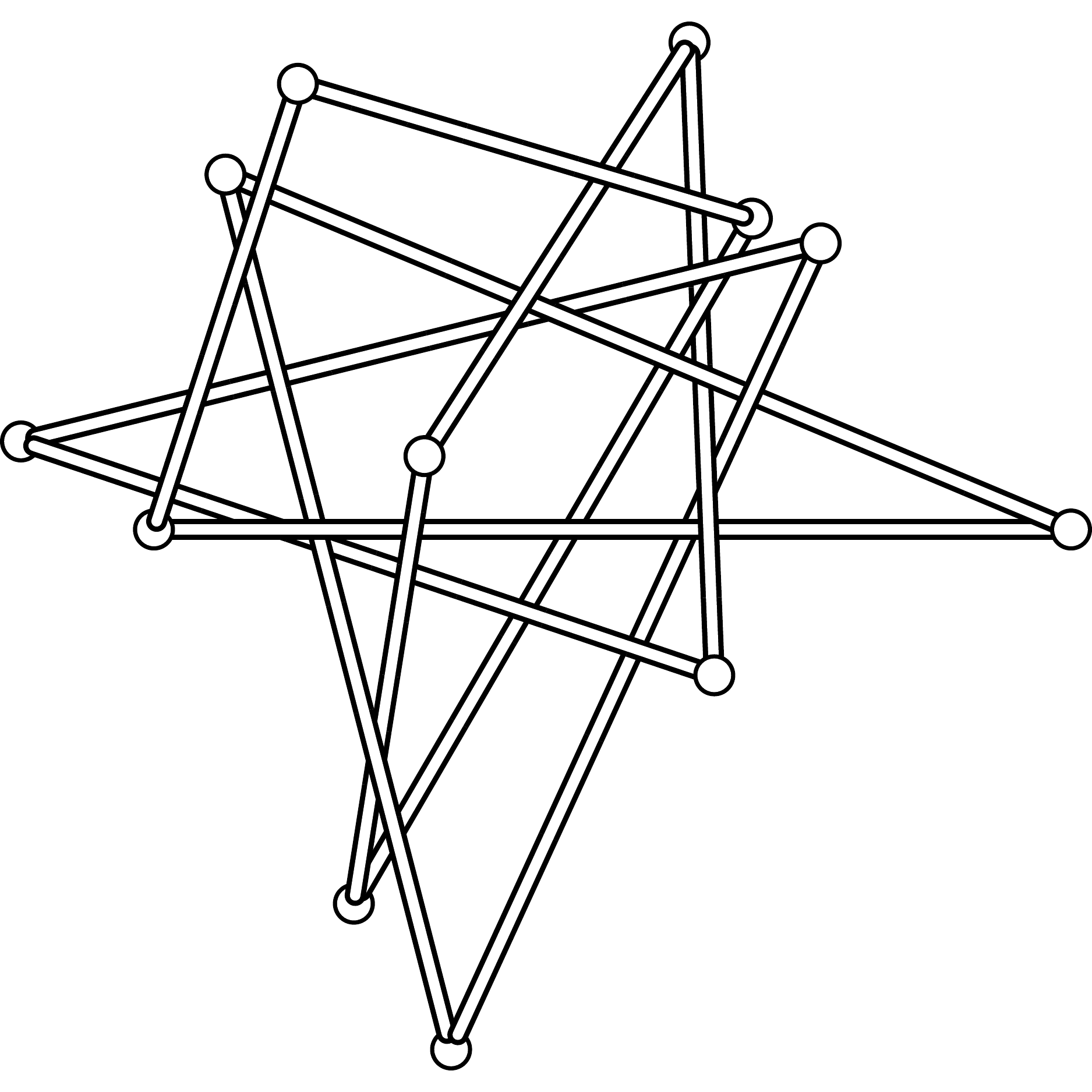}} & $0$ & $0$ & $0$ & $1$ \\
& $1000$ & $0$ & $0$ & $1$ \\
& $78$ & $387$ & $0$ & $1$ \\
& $324$ & $-567$ & $-169$ & $1$ \\
& $727$ & $312$ & $84$ & $1$ \\
& $-145$ & $96$ & $-356$ & $1$ \\
& $611$ & $-159$ & $248$ & $45081917$ \\
& $584$ & $531$ & $-476$ & $1$ \\
& $295$ & $80$ & $369$ & $484870544$ \\
& $218$ & $-408$ & $-501$ & $99346155$ \\
& $652$ & $339$ & $3$ & $16026827$ \\
& $157$ & $486$ & $860$ & $570711808$ \\
\end{tabular*}

\medskip

\begin{tabular*}{0.85\textwidth}{C{2.2in} R{.4in} R{.4in} R{.4in} | R{1in}}
\multicolumn{5}{c}{$13n_{2641}$} \\
\cline{1-5}\noalign{\smallskip}
\multirow{12}{*}{\includegraphics[height=1.8in]{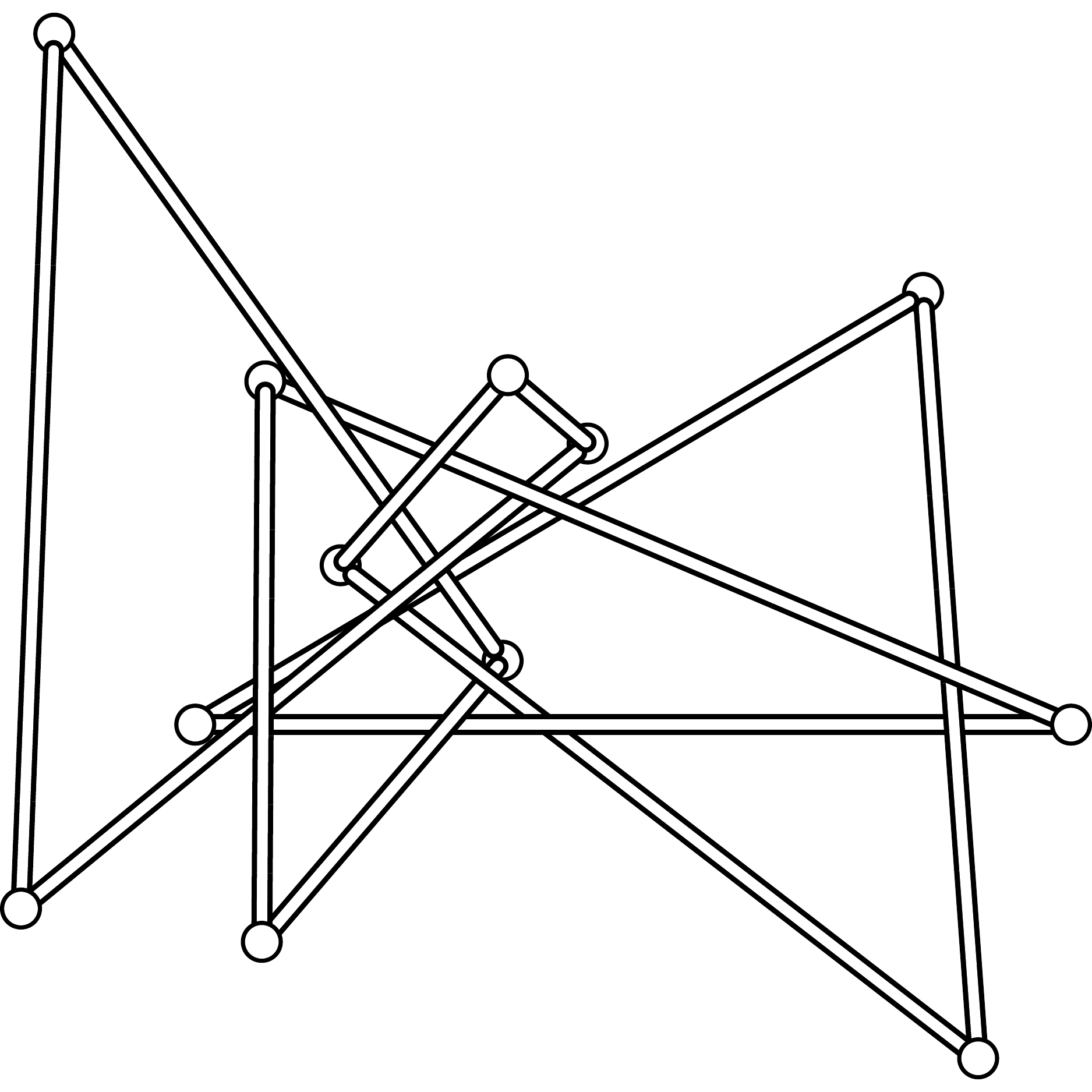}} & $0$ & $0$ & $0$ & $1$ \\
& $1000$ & $0$ & $0$ & $1$ \\
& $80$ & $392$ & $0$ & $4672994362$ \\
& $76$ & $-248$ & $769$ & $1$ \\
& $351$ & $73$ & $-138$ & $1$ \\
& $-161$ & $789$ & $336$ & $1$ \\
& $-199$ & $-210$ & $344$ & $1$ \\
& $448$ & $321$ & $-203$ & $1$ \\
& $357$ & $399$ & $790$ & $1665288265$ \\
& $166$ & $182$ & $-167$ & $2358435584$ \\
& $894$ & $-381$ & $224$ & $884270540$ \\
& $831$ & $493$ & $-257$ & $2538402146$ \\
\end{tabular*}

\medskip

\begin{tabular*}{0.85\textwidth}{C{2.2in} R{.4in} R{.4in} R{.4in} | R{1in}}
\multicolumn{5}{c}{$13n_{5018}$} \\
\cline{1-5}\noalign{\smallskip}
\multirow{12}{*}{\includegraphics[height=1.8in]{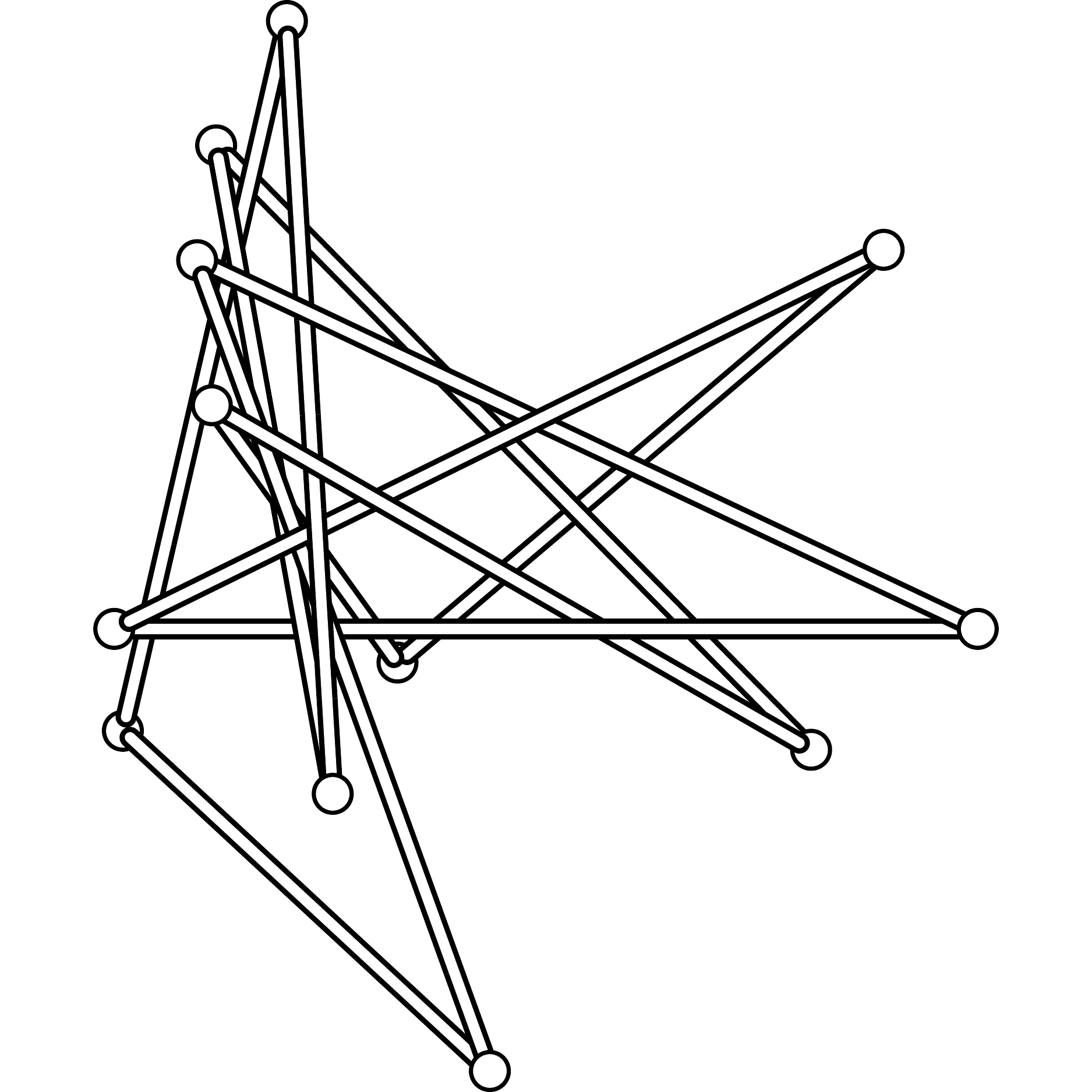}} & $0$ & $0$ & $0$ & $1$ \\
& $1000$ & $0$ & $0$ & $1$ \\
& $96$ & $427$ & $0$ & $1145198373$ \\
& $435$ & $-512$ & $68$ & $1$ \\
& $10$ & $-118$ & $-747$ & $1$ \\
& $200$ & $704$ & $-210$ & $1$ \\
& $253$ & $-191$ & $233$ & $3$ \\
& $118$ & $560$ & $-413$ & $1183487783$ \\
& $807$ & $-140$ & $-228$ & $2020126$ \\
& $113$ & $259$ & $371$ & $77920660$ \\
& $328$ & $-39$ & $-560$ & $17170427$ \\
& $891$ & $439$ & $114$ & $488988542$ \\
\end{tabular*}

\medskip

\begin{tabular*}{0.85\textwidth}{C{2.2in} R{.4in} R{.4in} R{.4in} | R{1in}}
\multicolumn{5}{c}{$14n_{1753}$} \\
\cline{1-5}\noalign{\smallskip}
\multirow{12}{*}{\includegraphics[height=1.8in]{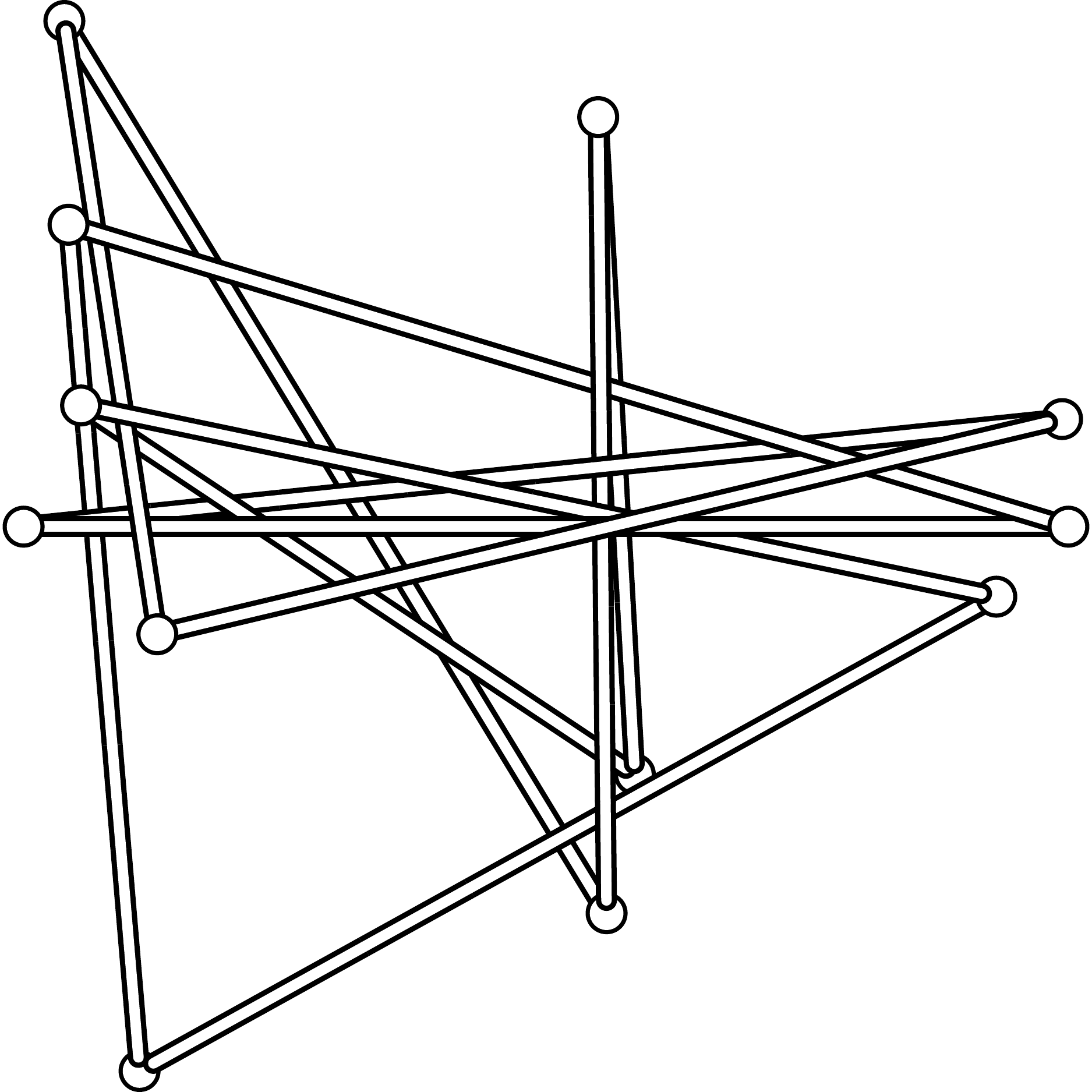}} & $0$ & $0$ & $0$ & $1$ \\
& $1000$ & $0$ & $0$ & $1$ \\
& $43$ & $289$ & $0$ & $1320037255$ \\
& $111$ & $-521$ & $-582$ & $1$ \\
& $931$ & $-67$ & $-234$ & $659509317$ \\
& $55$ & $116$ & $212$ & $1319018634$ \\
& $585$ & $-237$ & $-558$ & $1$ \\
& $550$ & $392$ & $219$ & $1$ \\
& $558$ & $-370$ & $-429$ & $65751782$ \\
& $39$ & $484$ & $-393$ & $442092810$ \\
& $128$ & $-103$ & $411$ & $315221136$ \\
& $994$ & $103$ & $-45$ & $993498425$ \\
\end{tabular*}

\end{center}

\clearpage

\section{Homomorphisms to the Symmetric Group} 
\label{sec:snappy}

SnapPy~\cite{snappy} diagrams of each of the knots with $\geq 13$ crossings mentioned in \Cref{thm:main}, together with surjective homomorphisms $\pi_1(S^3\!\backslash K) \twoheadrightarrow S_5$. The homomorphisms are specified by indicating the images of certain Wirtinger generators; the images of the remaining Wirtinger generators are determined by applying the Wirtinger relations at each crossing. These homomorphisms were found using a preliminary version of {\tt Wirt\_Hm\_Suite}~\cite{Wirt-Hm-Suite}.

The Wirtinger generators are described in terms of the corresponding strand in the diagram, which is specified by indicating the under crossings at which the strand begins and ends, and the over crossings (if any) along the way. For example, $(-11,12,10,-9)$ in the diagram for $13n_{226}$ below starts at the under strand of crossing 11, passes over crossings 12 and 10, and then ends as the under strand at crossing 9.

\setlength{\tabcolsep}{2pt}

\begin{center}
	\begin{tabular*}{0.7\textwidth}{C{2.5in} R{1in} L{.45in}}
	\multicolumn{3}{c}{$13n_{226}$} \\
	\midrule
	\multirow{13}{*}{\includegraphics[scale=0.5]{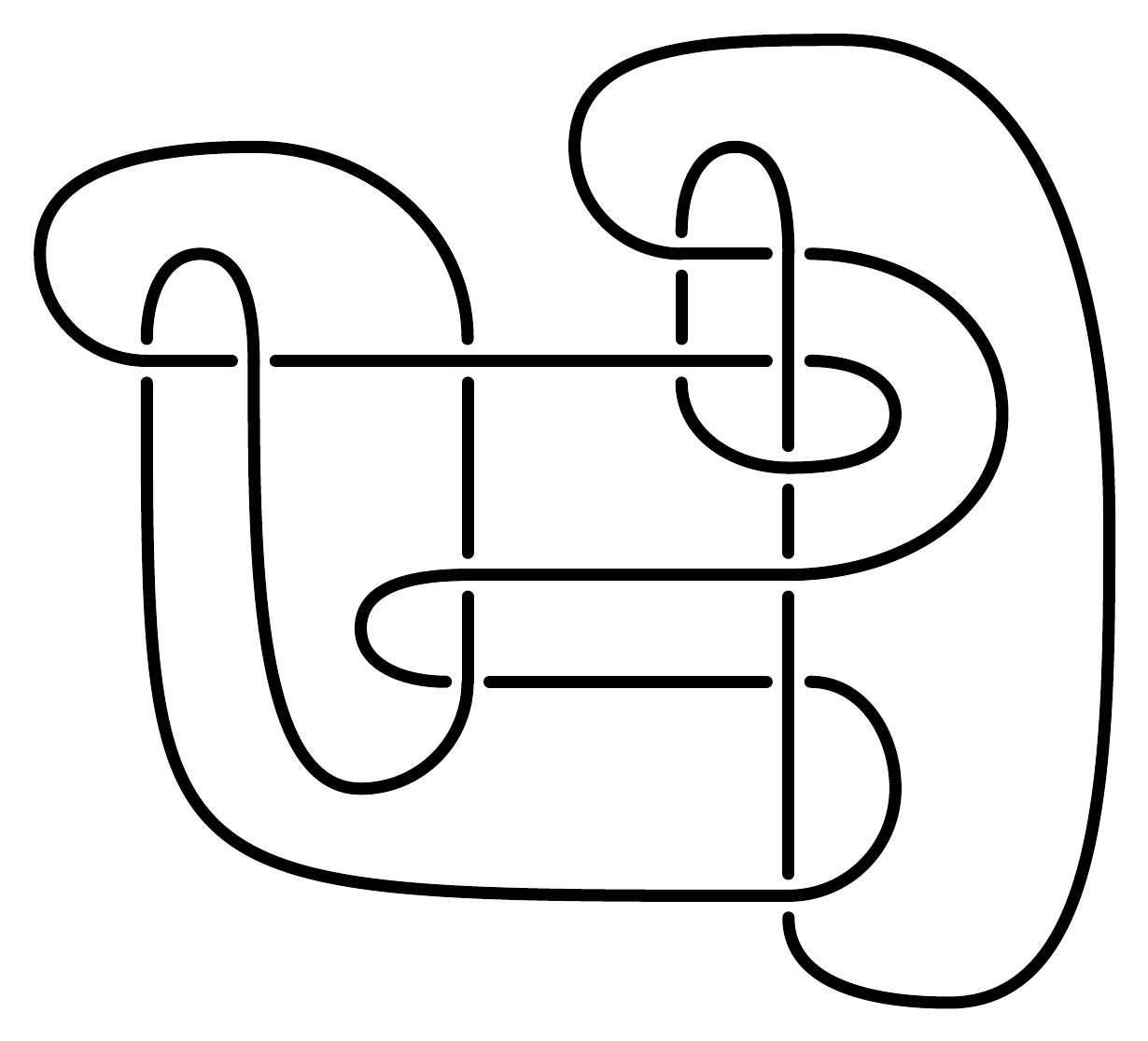}
			\put(-71,107){$1$}
			\put(-54,107){$2$}
			\put(-54,91){$3$}
			\put(-70,124){$4$}
			\put(-54,124){$5$}
			\put(-54,75){$6$}
			\put(-54,58){$7$}
			\put(-54,25){$8$}
			\put(-103,75){$9$}
			\put(-104,58){$10$}
			\put(-154,107){$11$}
			\put(-138,107){$12$}
			\put(-104,107){$13$}}
	& $(-3, -6)$ & $\mapsto (1\,2)$ \\
	& $(-2, 3, -1)$ & $\mapsto (1\,3)$ \\
	& $(-10, -7)$ & $\mapsto (1\,4)$ \\
	& $(-11, 12, 10, -9)$ & $\mapsto (2\,5)$\\
	& & \\
	& & \\
	& & \\
	& & \\
	& & \\
	& & \\
	& & \\
	& & \\
	& & \\
	\end{tabular*}

	\medskip

	\begin{tabular*}{0.7\textwidth}{C{2.5in} R{1in} L{.45in}}
	\multicolumn{3}{c}{$13n_{328}$} \\
	\midrule
	\multirow{13}{*}{\includegraphics[scale=0.5]{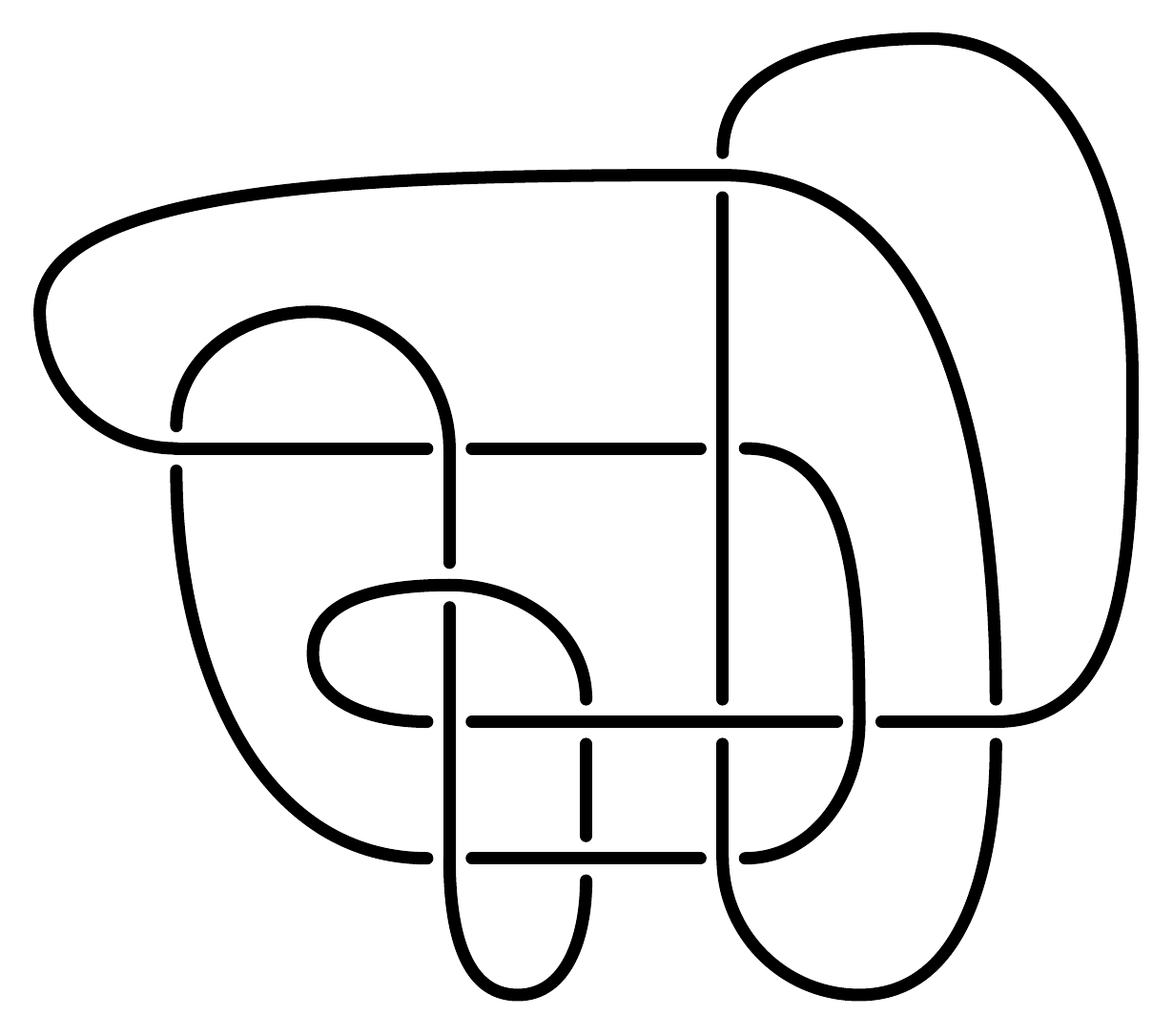}
			\put(-88,50){$1$}
			\put(-108,50){$2$}
			\put(-108,70){$3$}
			\put(-87,29){$4$}
			\put(-108,29){$5$}
			\put(-108,91){$6$}
			\put(-149,91){$7$}
			\put(-67,29){$8$}
			\put(-46,50){$9$}
			\put(-67,91){$10$}
			\put(-67,132){$11$}
			\put(-26,52){$12$}
			\put(-67,50){$13$}}
	& $(-7, -5)$ & $\mapsto (1\,2)$ \\
	& $(-2, 3, -1)$ & $\mapsto (1\,3)$ \\
	& $(-8, 9, -10)$ & $\mapsto (1\,4)$ \\
	& $(-6, 7, 11, -12)$ & $\mapsto (2\,5)$\\
	& & \\
	& & \\
	& & \\
	& & \\
	& & \\
	& & \\
	& & \\
	& & \\
	& & \\
	\end{tabular*}

	\medskip
	
	\begin{tabular*}{0.7\textwidth}{C{2.5in} R{1in} L{.45in}}
	\multicolumn{3}{c}{$13n_{342}$} \\
	\midrule
	\multirow{10}{*}{\includegraphics[scale=0.5]{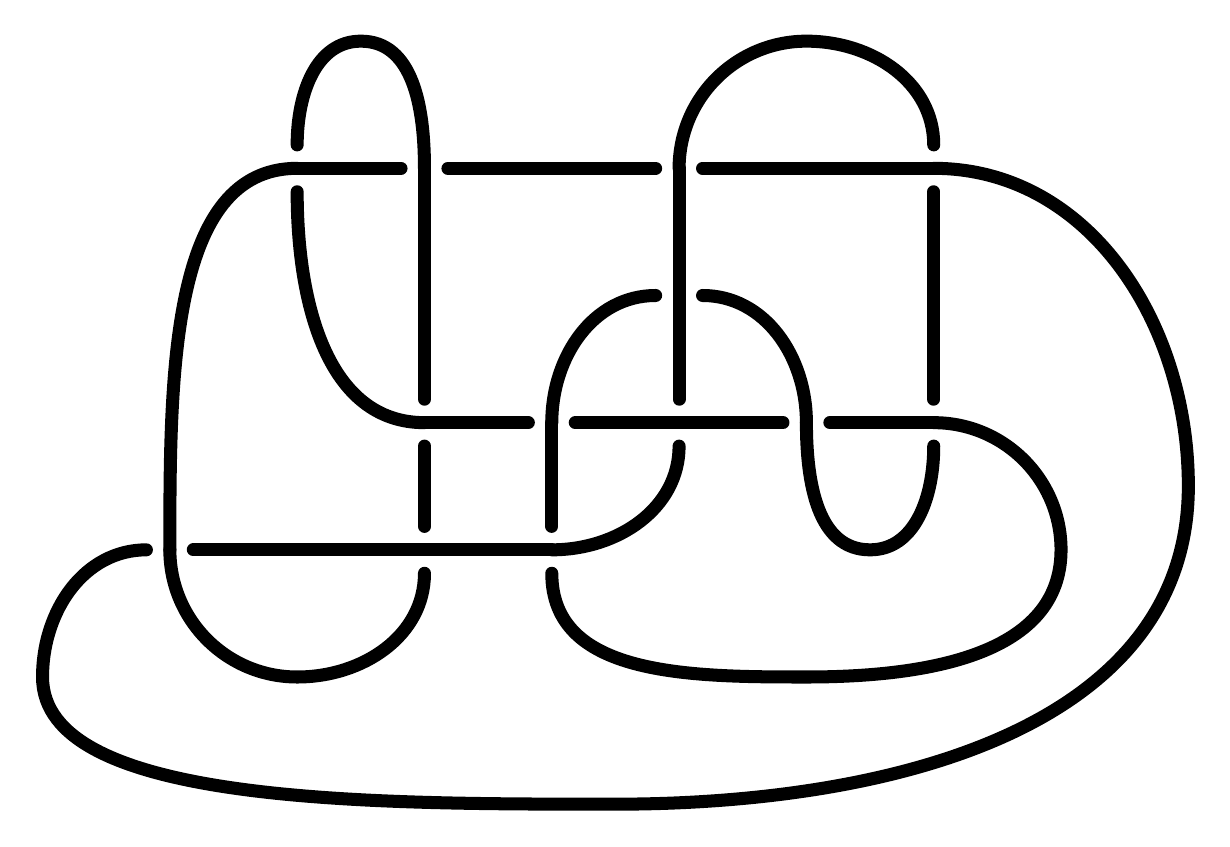}
			\put(-115,63){$1$}
			\put(-133,100){$2$}
			\put(-115,100){$3$}
			\put(-114.6,45){$4$}
			\put(-151,45){$5$}
			\put(-77,100){$6$}
			\put(-42,100){$7$}
			\put(-96,46){$8$}
			\put(-77,63){$9$}
			\put(-78,82){$10$}
			\put(-42,63){$11$}
			\put(-60,63){$12$}
			\put(-96,63){$13$}}
	& $(-7, -11)$ & $\mapsto (1\,2)$ \\
	& $(-5, 4, 8, -9)$ & $\mapsto (1\,3)$ \\
	& $(-2, 3, -1)$ & $\mapsto (1\,4)$ \\
	& $(-8, 11, -12)$ & $\mapsto (1\,5)$\\
	& & \\
	& & \\
	& & \\
	& & \\
	& & \\
	& & \\
	\end{tabular*}

	\medskip

	\begin{tabular*}{0.7\textwidth}{C{2.5in} R{1in} L{.45in}}
	\multicolumn{3}{c}{$13n_{343}$} \\
	\midrule
	\multirow{13}{*}{\includegraphics[scale=0.5]{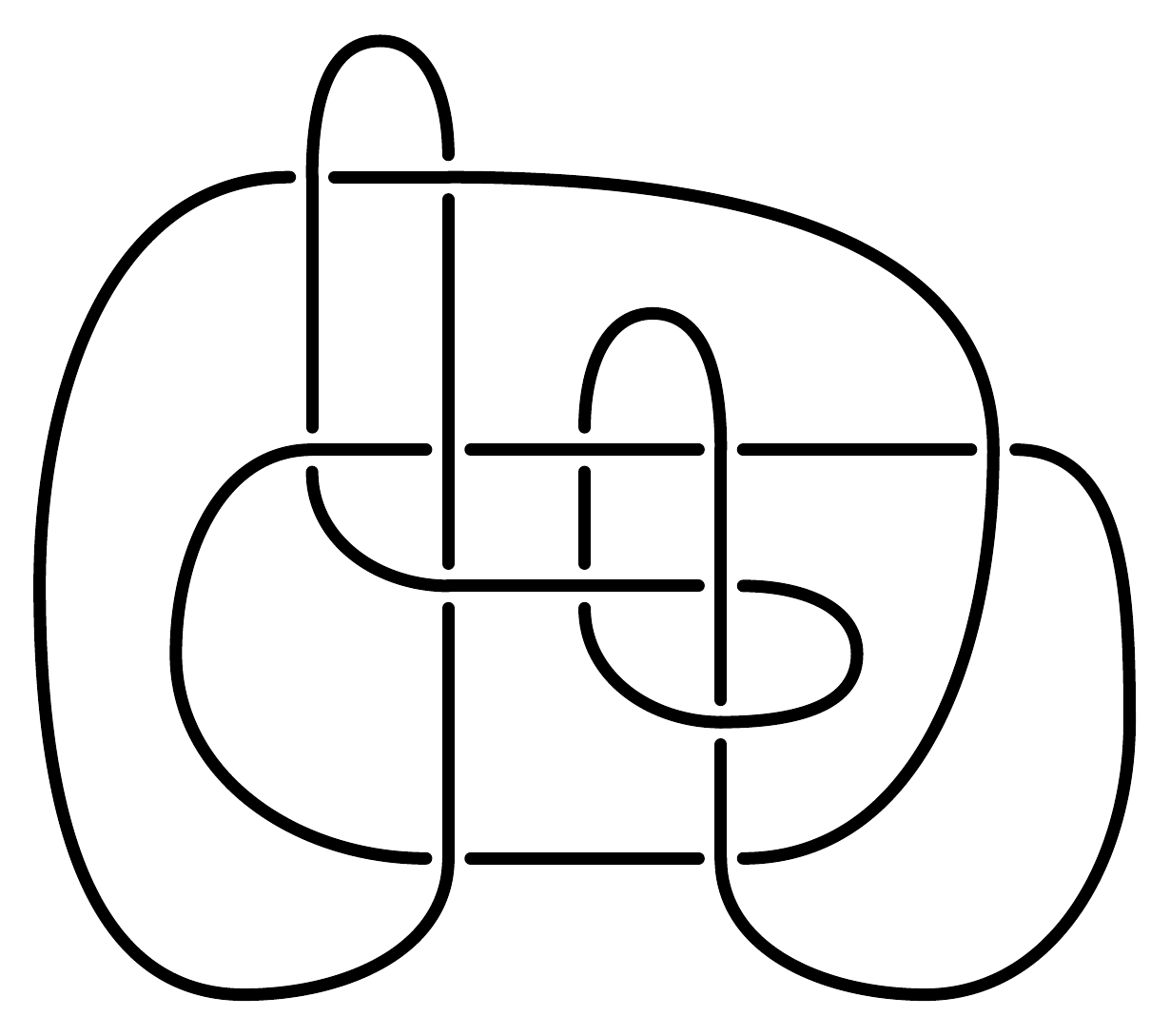}
			\put(-88,70){$1$}
			\put(-67,70){$2$}
			\put(-67,50){$3$}
			\put(-87,91){$4$}
			\put(-67,91){$5$}
			\put(-67,29){$6$}
			\put(-25,91){$7$}
			\put(-108,91){$8$}
			\put(-129,91){$9$}
			\put(-108.5,29){$10$}
			\put(-108.5,132){$11$}
			\put(-129,132){$12$}
			\put(-108.5,70){$13$}}
	& $(-7, -5)$ & $\mapsto (1\,2)$ \\
	& $(-2, 3, -1)$ & $\mapsto (1\,3)$ \\
	& $(-8, 9, -10)$ & $\mapsto (1\,4)$ \\
	& $(-6, 7, 11, -12)$ & $\mapsto (2\,5)$\\
	& & \\
	& & \\
	& & \\
	& & \\
	& & \\
	& & \\
	& & \\
	& & \\
	& & \\
	\end{tabular*}

	\medskip

	\begin{tabular*}{0.7\textwidth}{C{2.5in} R{1in} L{.45in}}
	\multicolumn{3}{c}{$13n_{350}$} \\
	\midrule
	\multirow{15}{*}{\includegraphics[scale=0.5]{K13n350SnapPy.pdf}
			\put(-47,50){$1$}
			\put(-46,70){$2$}
			\put(-25,70){$3$}
			\put(-66,50){$4$}
			\put(-66,70){$5$}
			\put(-66,153){$6$}
			\put(-66,112){$7$}
			\put(-129,71){$8$}
			\put(-129,111){$9$}
			\put(-129,132){$10$}
			\put(-109,111){$11$}
			\put(-150,111){$12$}
			\put(-150,71){$13$}}
	& $(-7, -5)$ & $\mapsto (1\,2)$ \\
	& $(-8, 9, -10)$ & $\mapsto (1\,3)$ \\
	& $(-6, 7, 11, -9)$ & $\mapsto (1\,4)$ \\
	& $(-2, 3, -1)$ & $\mapsto (2\,5)$\\
	& & \\
	& & \\
	& & \\
	& & \\
	& & \\
	& & \\
	& & \\
	& & \\
	& & \\
	& & \\
	& & \\
	\end{tabular*}

	\medskip

	\begin{tabular*}{0.7\textwidth}{C{2.5in} R{1in} L{.45in}}
	\multicolumn{3}{c}{$13n_{512}$} \\
	\midrule
	\multirow{11}{*}{\includegraphics[scale=0.5]{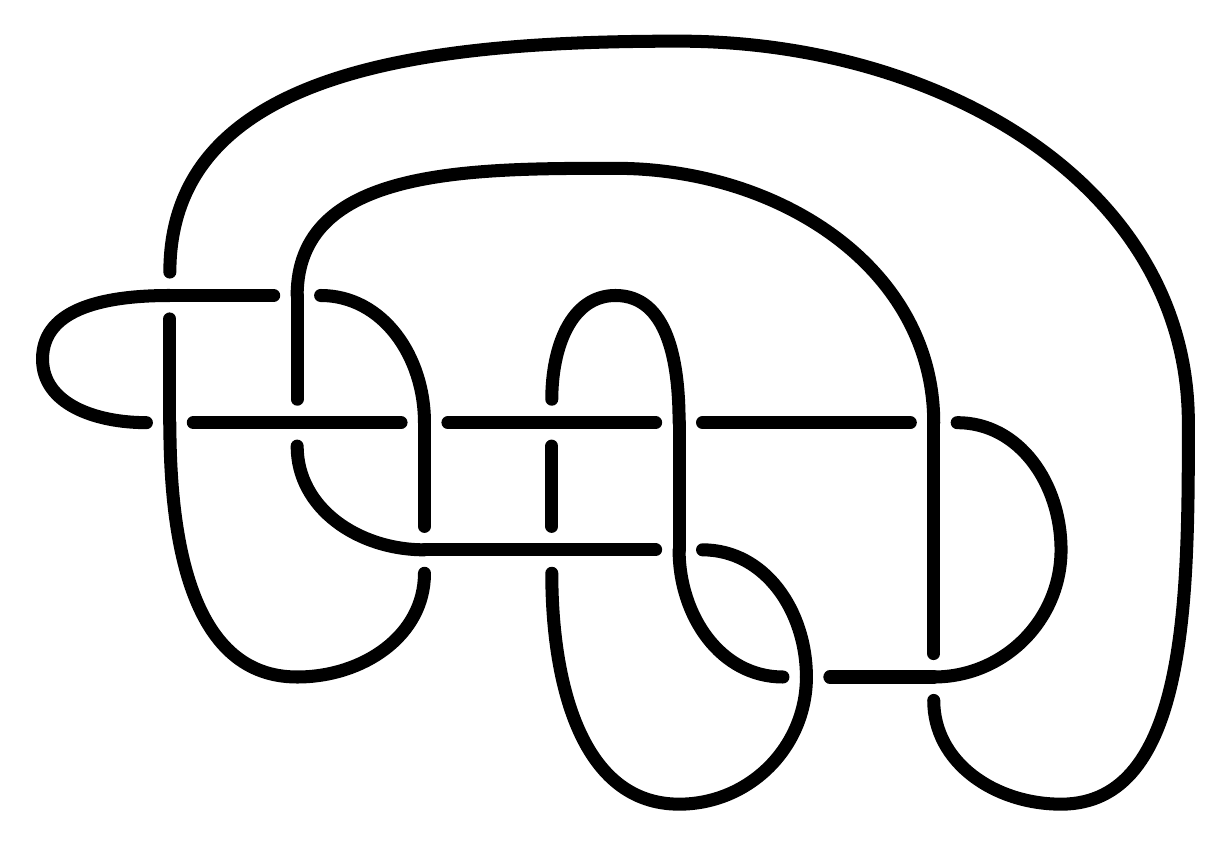}
			\put(-97,45){$1$}
			\put(-78,45){$2$}
			\put(-60,27){$3$}
			\put(-96,63){$4$}
			\put(-78,63){$5$}
			\put(-41,27){$6$}
			\put(-42,63){$7$}
			\put(-115,63){$8$}
			\put(-133,63){$9$}
			\put(-152,63){$10$}
			\put(-152,81){$11$}
			\put(-133,81){$12$}
			\put(-115,45){$13$}}
	& $(-7, -5)$ & $\mapsto (1\,2)$ \\
	& $(-6, 7, 12, -9)$ & $\mapsto (1\,3)$ \\
	& $(-2, 3, -1)$ & $\mapsto (2\,4)$ \\
	& $(-12, 8, -13)$ & $\mapsto (3\,5)$\\
	& & \\
	& & \\
	& & \\
	& & \\
	& & \\
	& & \\
	\end{tabular*}

	\medskip

	\begin{tabular*}{0.7\textwidth}{C{2.5in} R{1in} L{.45in}}
	\multicolumn{3}{c}{$13n_{973}$} \\
	\midrule
	\multirow{12}{*}{\includegraphics[scale=0.5]{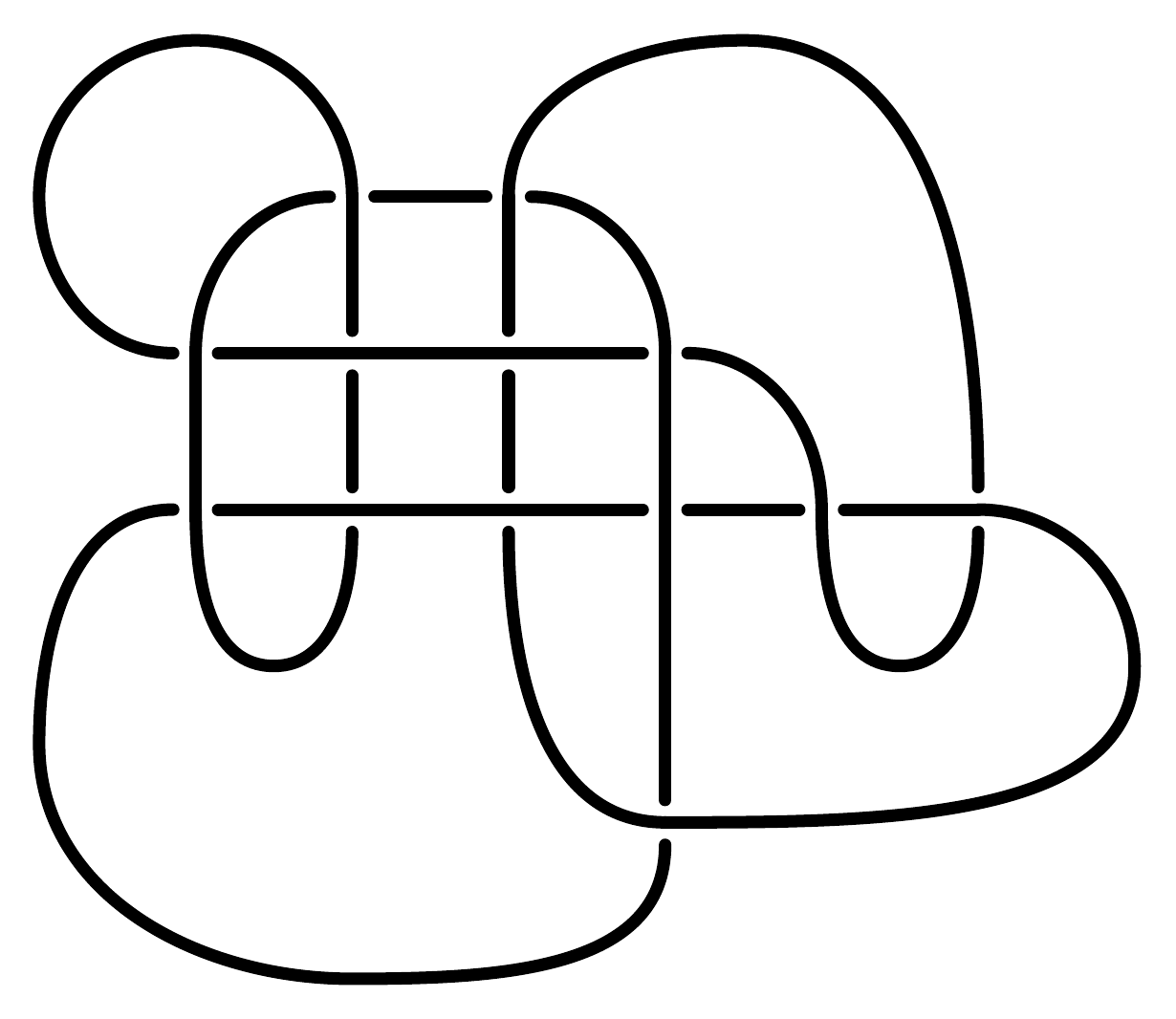}
			\put(-123,102){$1$}
			\put(-146,102){$2$}
			\put(-123,126){$3$}
			\put(-123,79){$4$}
			\put(-146,79){$5$}
			\put(-98,126){$6$}
			\put(-76,102){$7$}
			\put(-75,79){$8$}
			\put(-75,32){$9$}
			\put(-100,79){$10$}
			\put(-53,79){$11$}
			\put(-29,79){$12$}
			\put(-100,102){$13$}}
	& $(-13, 6, -12)$ & $\mapsto (1\,2)$ \\
	& $(-7, 13, 1, -2)$ & $\mapsto (1\,3)$ \\
	& $(-2, 3, -1)$ & $\mapsto (1\,4)$ \\
	& $(-5, 4, 10, -8)$ & $\mapsto (2\,5)$\\
	& & \\
	& & \\
	& & \\
	& & \\
	& & \\
	& & \\
	& & \\
	& & \\
	\end{tabular*}

	\medskip

	\begin{tabular*}{0.7\textwidth}{C{2.5in} R{1in} L{.45in}}
	\multicolumn{3}{c}{$13n_{2641}$} \\
	\midrule
	\multirow{12}{*}{\includegraphics[scale=0.5]{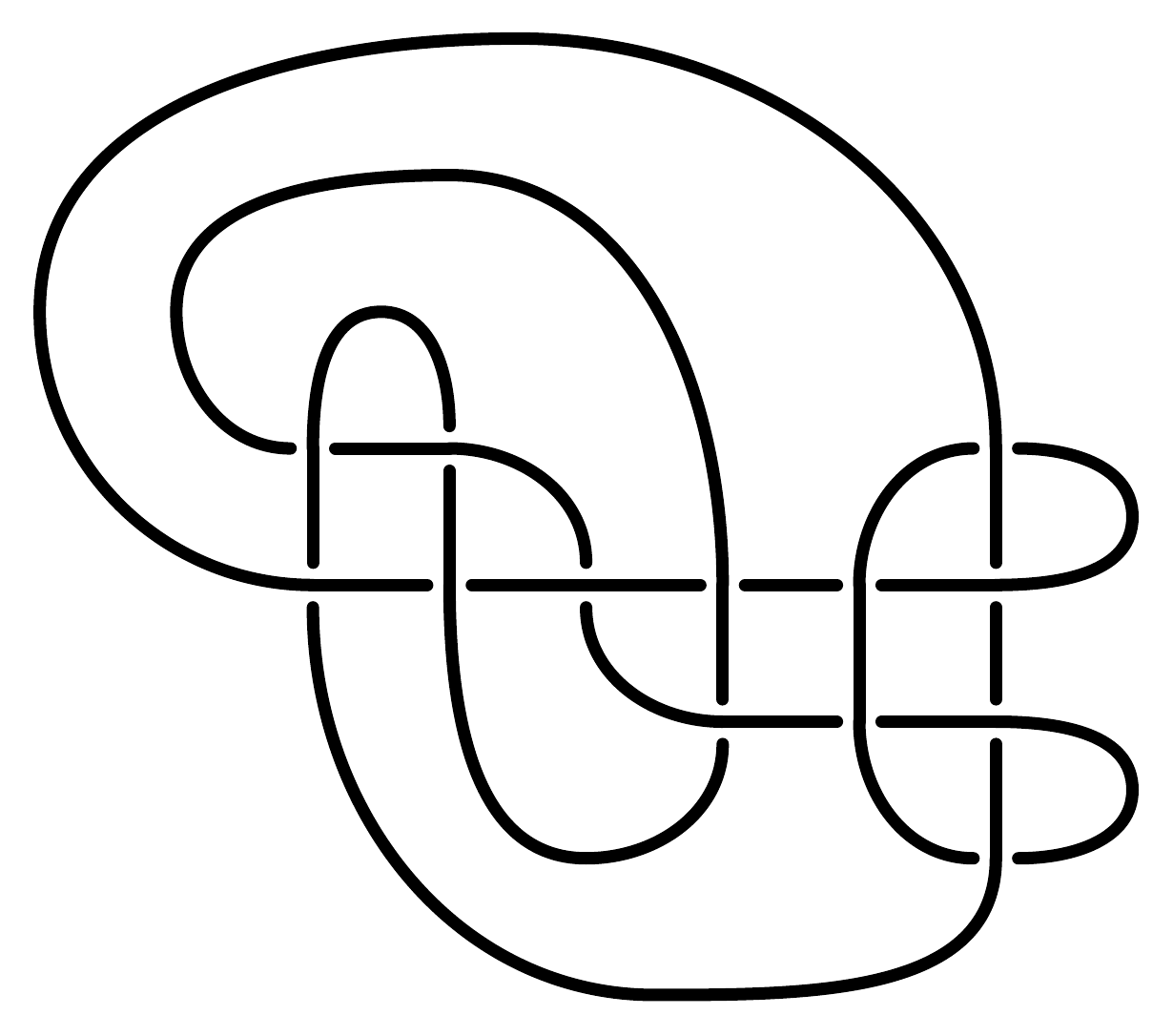}
			\put(-46,70){$1$}
			\put(-25,91){$2$}
			\put(-25,70){$3$}
			\put(-66,70){$4$}
			\put(-87,70){$5$}
			\put(-108,70){$6$}
			\put(-129,70){$7$}
			\put(-25,50){$8$}
			\put(-25,29){$9$}
			\put(-129,91){$10$}
			\put(-109,91){$11$}
			\put(-67,50){$12$}
			\put(-46,50){$13$}}
	& $(-6, 7, 2, -3)$ & $\mapsto (1\,2)$ \\
	& $(-12, 4, -10)$ & $\mapsto (1\,3)$ \\
	& $(-2, 3, -1)$ & $\mapsto (1\,4)$ \\
	& $(-8, 9, -7)$ & $\mapsto (2\,5)$\\
	& & \\
	& & \\
	& & \\
	& & \\
	& & \\
	& & \\
	& & \\
	& & \\
	& & \\
	\end{tabular*}

	\medskip

	\begin{tabular*}{0.7\textwidth}{C{2.5in} R{1in} L{.45in}}
	\multicolumn{3}{c}{$13n_{5018}$} \\
	\midrule
	\multirow{13}{*}{\includegraphics[scale=0.5]{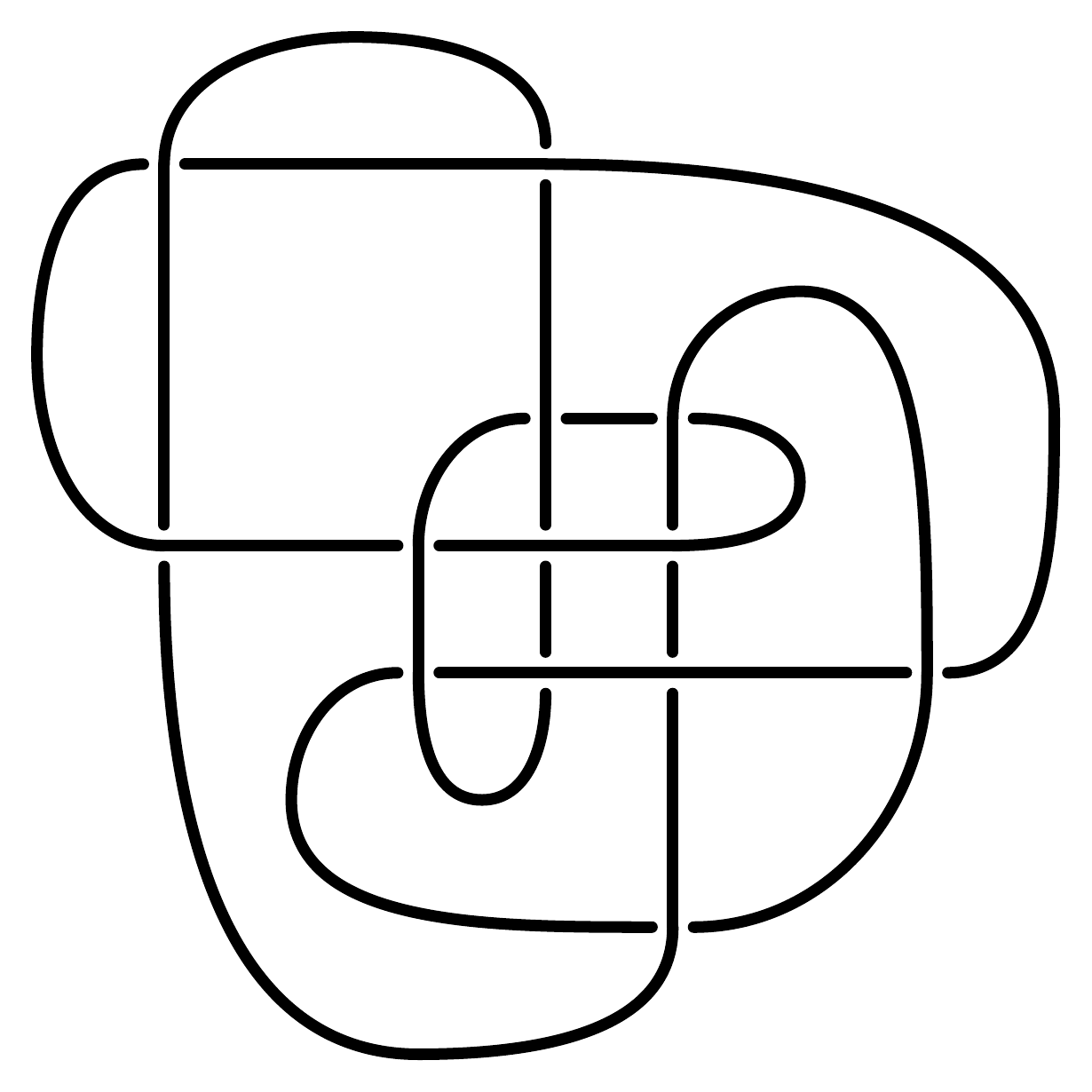}
			\put(-88,111){$1$}
			\put(-87,91){$2$}
			\put(-87,70){$3$}
			\put(-108,70){$4$}
			\put(-108,91){$5$}
			\put(-66,111){$6$}
			\put(-67,91){$7$}
			\put(-149,91){$8$}
			\put(-148,153){$9$}
			\put(-88,153){$10$}
			\put(-26,71){$11$}
			\put(-68,70){$12$}
			\put(-68,29){$13$}}
	& $(-6, 7, 2, -5)$ & $\mapsto (1\,2)$ \\
	& $(-2, -3)$ & $\mapsto (1\,3)$ \\
	& $(-8, 9, -10)$ & $\mapsto (2\,4)$ \\
	& $(-13, 11, 6, -7)$ & $\mapsto (3\,5)$\\
	& & \\
	& & \\
	& & \\
	& & \\
	& & \\
	& & \\
	& & \\
	& & \\
	& & \\
	& & \\
	\end{tabular*}

	\medskip

	\begin{tabular*}{0.7\textwidth}{C{2.5in} R{1in} L{.45in}}
	\multicolumn{3}{c}{$14n_{1753}$} \\
	\midrule
	\multirow{13}{*}{\includegraphics[scale=0.5]{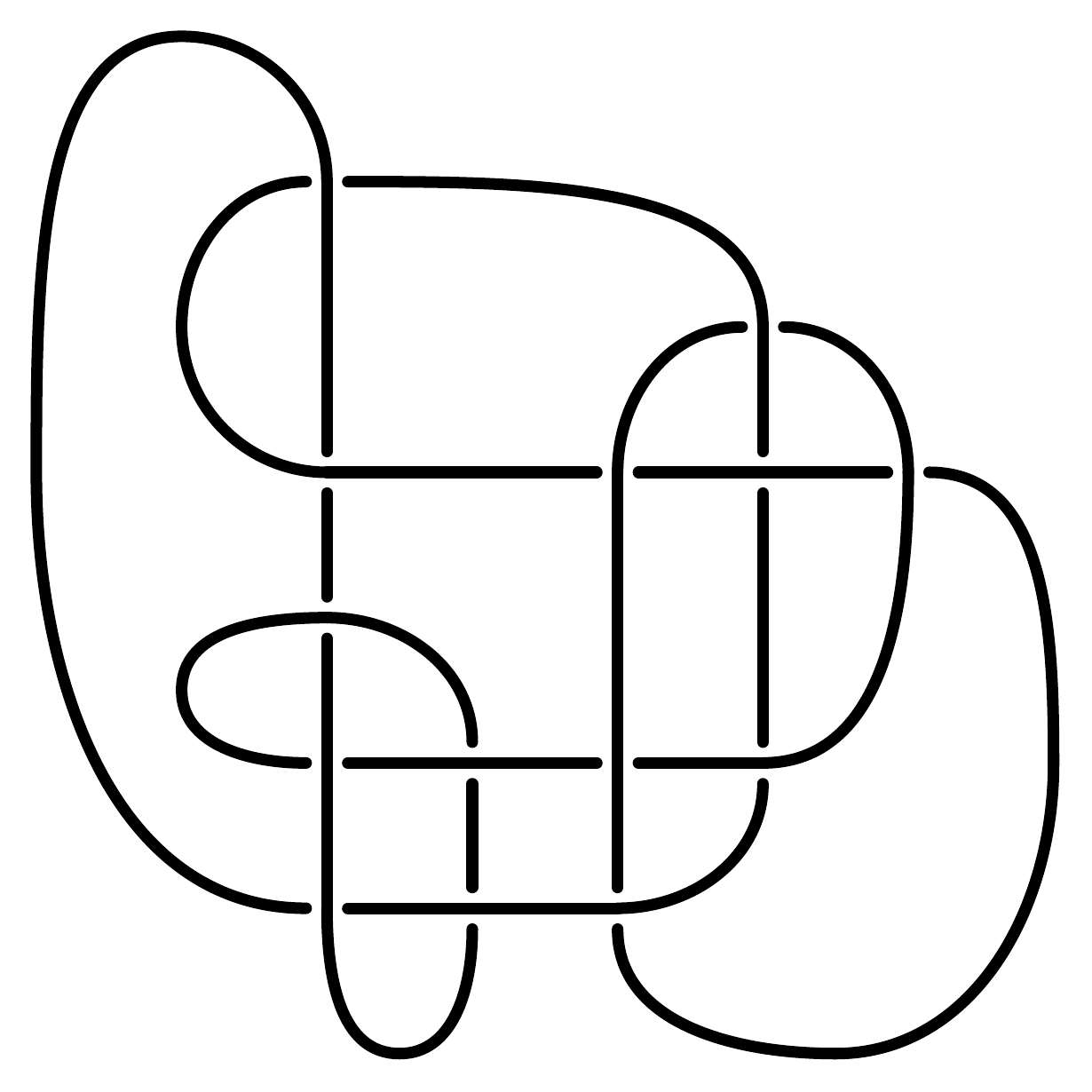}
			\put(-100,56){$1$}
			\put(-123,56){$2$}
			\put(-123,79){$3$}
			\put(-99,32){$4$}
			\put(-123,32){$5$}
			\put(-123,103){$6$}
			\put(-123,151){$7$}
			\put(-75,33){$8$}
			\put(-51.5,56){$9$}
			\put(-52,103){$10$}
			\put(-52,127){$11$}
			\put(-76,103){$12$}
			\put(-30,103){$13$}
			\put(-76,56){$14$}}
	& $(-5, 4, 8, -9)$ & $\mapsto (1\,2)$ \\
	& $(-2,3,-1)$ & $\mapsto (1\,3)$ \\
	& $(-3, -6)$ & $\mapsto (1\,4)$ \\
	& $(-8, 14, 12, -11)$ & $\mapsto (1\,5)$\\
	& & \\
	& & \\
	& & \\
	& & \\
	& & \\
	& & \\
	& & \\
	& & \\
	& & \\
	\end{tabular*}
\end{center}

\clearpage

\bibliography{stickknots-special,stickknots}

\end{document}